\documentclass[a4paper, 11pt]{article}

\usepackage[utf8]{inputenc}
\usepackage[OT1]{fontenc}
\usepackage{lmodern}     
\usepackage[english]{babel}
\usepackage{cite, hyphenat, hyperref}

\usepackage{amscd, amsfonts, amsmath}
\usepackage{amssymb, amsthm}
\usepackage{bbm}
\usepackage{enumerate}
\usepackage{makeidx}
\usepackage{calc}
\usepackage[shortlabels]{enumitem}

\usepackage{nccfoots}
\usepackage{appendix}
\usepackage[a4paper, top=2.5cm, bottom=2.5cm, left=3cm, right=3cm]{geometry}
\usepackage[normalem]{ulem}
\allowdisplaybreaks

\numberwithin{equation}{section}
\theoremstyle{plain}

\newtheorem{Lemma}{Lemma}[section]
\newtheorem{Proposition}[Lemma]{Proposition}
\newtheorem{Theorem}[Lemma]{Theorem}
\newtheorem{Corollary}[Lemma]{Corollary}
\theoremstyle{definition}
\newtheorem{Definition}[Lemma]{Definition}
\newtheorem{Example}[Lemma]{Example}
\newtheorem{Examples}[Lemma]{Examples}
\newtheorem{Remark}[Lemma]{Remark}

\def\oe{\"{o}}

\def\quoteleft{`}
\def\B{{\rm B}}
\def\R{{\mathbb{R}}}
\def\N{{\mathbb{N}}}

\DeclareMathOperator*{\esssup}{ess\,sup}

\begin{document}
\title{\textbf{Stability, uniqueness and existence of solutions to McKean-Vlasov SDEs: a multidimensional Yamada-Watanabe approach}}
\date{June 12, 2024}
\author{Alexander Kalinin\footnote{Department of Mathematics, LMU Munich, Germany. {\tt kalinin@math.lmu.de}, {\tt meyerbra@math.lmu.de}} \and Thilo Meyer-Brandis\footnotemark[1]
\and Frank Proske\setcounter{footnote}{6}\footnote{Department of Mathematics, University of Oslo, Norway. {\tt proske@math.uio.no}}
}
\maketitle

\begin{abstract}
We establish stability and pathwise uniqueness of solutions to Wiener noise driven McKean-Vlasov equations with random non-Lipschitz continuous coefficients. In the deterministic case, we also obtain the existence of unique strong solutions. By using our approach, which is based on an extension of the Yamada-Watanabe ansatz to the multidimensional setting and which does not rely on the construction of Lyapunov functions, we prove first moment and pathwise exponential stability. Furthermore, Lyapunov exponents are computed explicitly.
\end{abstract}


\section{Introduction}\label{se:1}

In this work, let $d,m\in\N$ and $(\Omega ,\mathcal{F},(\mathcal{F}_{t})_{t\geq 0},P)$ be a
filtered probability space that satisfies the usual conditions and carries a standard $d$-dimensional $(\mathcal{F}_{t})_{t\geq 0}$-Brownian motion $W$. McKean-Vlasov stochastic differential equations (McKean-Vlasov SDEs), or alternatively mean-field SDEs, are integral equations of the form 
\begin{equation}\label{eq:MFSDE}
X_{t}=X_{0}+\int_{0}^{t}\!b\big(s,X_{s},\mathcal{L}(X_{s})\big)\,ds + \int_{0}^{t}\!\sigma\big(s,X_{s},\mathcal{L}(X_{s})\big)\,dW_{s}\quad\text{for $t\geq 0$ a.s.}
\end{equation}
with the drift and diffusion coefficients $b$ and $\sigma$ defined on $\mathbb{R}_{+}\times \mathbb{R}^{m}\times \mathcal{P}(\mathbb{R}^{m})$ and taking their respective values in $\R^{m}$ and $\R^{m\times d}$, where $\mathcal{P}(\R^{m})$ denotes the convex space of all Borel probability measures on $\mathbb{R}^{m}$. 

Originally motivated by Boltzmann's equation in kinetic gas theory, McKean-Vlasov equations were first studied by Kac~\cite{Kac56}, McKean~\cite{McK66} and Vlasov~\cite{Vlas68} and utilised to describe the stochastic dynamics of large interacting particle systems. In particular, Vlasov analysed the propagation of chaos of charged particles with long-range interaction in a plasma, where the interaction of particles is modelled by means of a system of SDEs. It turns out that the convergence of such an $N$-particle system, as the positive integer $N$ tends to infinity, which is referred to as propagation of chaos, can be
described by an equation of the form~\eqref{eq:MFSDE}.

Since the seminal papers~\cite{Kac56},~\cite{McK66} and~\cite{Vlas68}, mean-field SDEs attracted much interest and sparked off new ground-breaking developments in both theory and applications in a variety of other research areas. For example, Lasry and Lions~\cite{LasLio07} employ mean-field techniques to study mean-field games in economic applications. There, the interaction of strategies of \quoteleft rational' players in an $N$-player differential game is described by a system of SDEs. Then, in the limiting case $N\rightarrow \infty $, the authors derive a system of partial differential equations that consists of a Hamilton-Jacobi and a Kolmogorov equation, which they use to examine the behaviour of agents in a vast network.

As for other works in this direction, but based on different methods, which rely on a probabilistic analysis, we refer to~\cite{CarDel13},~\cite{CarDel14},~\cite{CarDel18-1},~\cite{CarDel18-2} and~\cite{CarDelLac13}. Recent applications of mean-field methods include e.g.~the modelling of systemic risk in financial networks, see~\cite{CarFouMouSun18},~\cite{CarFouSun15},~\cite{FouIch13},~\cite{FouSun13},~\cite{GarPapYan13},~\cite{HamSoj19} and~\cite{KleKluRei15}. Further, extensions of~\eqref{eq:MFSDE} to
the setting of mean-field SDEs driven by L\'{e}vy processes or backward mean-field SDEs were studied in~\cite{BucDjeLiPen09},~\cite{BucLiPen09},~\cite{BucLiPenRai17} and~\cite{JouMelWoy08}. More recently, a robust solution theory of mean-field SDEs from the perspective of rough path theory was developed in~\cite{BaiCatDel20}.

Mean-field equations of type~\eqref{eq:MFSDE} were also examined for non-Lipschitz vector fields $b$ and $\sigma $. See e.g.~\cite{Jou95}, where the author invokes martingale techniques to construct unique weak solutions in the case of a bounded drift $b$, which is Lipschitz continuous in the law variable. We also refer to~\cite{Chi94},~\cite{MisVer20} in the case of weak solutions. As for path-dependent coefficients, see~\cite{LiMin16}. More recently, by using techniques based on Malliavin calculus, the existence of unique strong solutions to mean-field SDEs with additive Brownian noise and singular drift $b$ were obtained in~\cite{BauMeyPro18},~\cite{BauMey19-1} and~\cite{BauMey19-2} and a Bismut-Elworthy-Li formula for such equations was derived there. See also~\cite{BauMey19-3} for the case of a certain non-Markovian and \quoteleft rough' Gaussian driving noise in a Hilbert space setting.

Now let $t_{0}\geq 0$ and $\mathcal{P}$ be a separable metrisable space in $\mathcal{P}(\R^{m})$. 
In this article, we consider a \emph{McKean-Vlasov SDE} with \emph{random drift} and \emph{diffusion coefficients} of the form
\begin{equation}\label{eq:McKean-Vlasov}
dX_{t}=\B_{t}\big(X_{t},\mathcal{L}(X_{t})\big)\,dt +\Sigma _{t}(X_{t})\,dW_{t}\quad 
\text{for $t\geq t_{0}$},
\end{equation}
where $\B:[t_{0},\infty[\times \Omega \times \mathbb{R}^{m}\times \mathcal{P}\rightarrow\mathbb{R}^{m}$ and $\Sigma:[t_{0},\infty[\times \Omega \times \mathbb{R}^{m}\rightarrow\mathbb{R}^{m\times d}$ are measurable in an appropriate meaning. This includes \emph{controlled SDEs}, as Example~\ref{ex:controlled processes} shows. While the diffusion $\Sigma$ does not depend on the measure variable $\mu\in\mathcal{P}$, both $\B$ and $\Sigma$ are allowed to be non-Lipschitz, as stated more precisely below. 

The measure state space $\mathcal{P}$, whose topology may be finer than the topology of weak convergence, is required to be admissible in a suitable measurable sense, as introduced in Section~\ref{se:2}. For instance, $\mathcal{P}$ may stand for $\mathcal{P}(\mathbb{R}^{m})$, endowed with the Prokhorov metric, or the Polish space $\mathcal{P}_{1}(\mathbb{R}^{m})$ of all measures in $\mathcal{P}(\mathbb{R}^{m})$ with a finite first absolute moment, equipped with the Wasserstein metric.

In this general framework, the main objective of our work is to establish Lyapunov stability, uniqueness and existence of solutions to McKean-Vlasov SDEs of type~\eqref{eq:McKean-Vlasov}. Our methodology is based on a \emph{multidimensional Yamada-Watanabe approach} and allows for irregular drifts. More precisely, under an Osgood condition on compact sets on $\Sigma$, our paper offers the following \emph{novel contributions} compared to the existing literature:
\begin{enumerate}[(1)]
\item \emph{Pathwise uniqueness} for~\eqref{eq:McKean-Vlasov} follows from Corollary~\ref{co:pathwise uniqueness} under a partial Osgood condition on $\B$, and if the drift is independent of $\mu\in\mathcal{P}$, in which case~\eqref{eq:McKean-Vlasov} reduces to a SDE, then this condition is imposed on compact sets only.

\item \emph{(Asymptotic) moment stability} for~\eqref{eq:McKean-Vlasov} is inferred from Proposition~\ref{pr:moment stability estimate}, which yields a general \emph{$L^{1}$-comparison estimate}, and stated in Corollary~\ref{co:moment stability} under a partial mixed H\oe lder continuity condition on $\B$ and verifiable integrability conditions on the random partial H\oe lder coefficients with respect to $(x,\mu)\in\R^{m}\times\mathcal{P}$.

\item \emph{Exponential moment stability} is asserted by Corollary~\ref{co:exponential moment stability} if $\B$ satisfies a partial Lipschitz condition and the partial Lipschitz coefficients are bounded by a sum of power functions that determines the moment Lyapunov exponent explicitly.

\item \emph{Pathwise exponential stability} is implied by Corollary~\ref{co:special pathwise stability} if the preceding conditions on $\B$ hold, the random Lipschitz coefficient of $\B$ relative to $\mu\in\mathcal{P}$ is of suitable growth and the Osgood condition on $\Sigma$ is replaced by an $\frac{1}{2}$-H\oe lder condition. In particular, the pathwise Lyapunov exponent is half the moment Lyapunov exponent.

\item Thereby, we show that all these stability results can be obtained under verifiable assumptions, \emph{without resorting to the existence of Lyapunov functions}.

\item \emph{Existence of unique strong solutions} is established in Theorem~\ref{th:strong existence} when $\B$ and $\Sigma$ are deterministic, $\B$ satisfies a partial affine growth and a partial Lipschitz condition and $\B_{s}(\cdot,\cdot)$ is continuous and $\Sigma_{s}(0) = 0$ for all $s\geq t_{0}$. In particular, $\B$ and $\Sigma$ may fail to be of affine growth.
\end{enumerate}

Concisely, our methods for proving these main results rest on the pathwise uniqueness approach of Yamada and Watanabe~\cite{YamWat71}, which we extend to the multidimensional setting. In this context, let us consider two articles, which employ the Yamada-Watanabe ansatz to show pathwise uniqueness and stability of solutions:

In~\cite{AltSch16} the authors verify pathwise uniqueness of solutions to SDEs with deterministic coefficients. However, the Yamada-Watanabe condition given in the multidimensional case is rather restrictive and essentially reduces to a Lipschitz condition on the diffusion. The article~\cite{BahlMezMez20} pertains to the study of one-dimensional mean-field SDEs with bounded drift and diffusion coefficients by means of the Yamada-Watanabe approach. There, it is essentially assumed that the drift coefficient is Lipschitz continuous in the spatial and the law variable, while the diffusion satisfies a global Osgood condition.

In our paper, however, the conditions on the coefficients, even in the one-dimensional deterministic case, are weaker than in~\cite{BahlMezMez20}, since no growth conditions are imposed.
Further, we only require that $\B$ and $\Sigma$ satisfy a partial Osgood condition and an Osgood condition on compact sets, respectively. Finally, in the context of stability results for SDEs with irregular coefficients, we also mention the work~\cite{AryPil19}, where the authors prove moment exponential stability of solutions to SDEs driven by a drift vector field with discontinuities on a hyperplane.\smallskip

Our paper is organised as follows. In Section~\ref{se:2} we prove auxiliary results on measure state spaces and introduce the probabilistic setting. In Section~\ref{se:3} our main results are presented, whose proofs are given in Section~\ref{se:5}. Section~\ref{se:4} is devoted to a priori estimates and a pathwise asymptotic analysis for random It\^{o} processes. The results of this section are of independent interest and serve as basis for the derivation of the main results.

\section{Preliminaries}\label{se:2}

In the following, $|\cdot|$ is used as absolute value function, Euclidean norm or Hilbert-Schmidt norm, $|\cdot|_{1}$ stands for the $1$-norm and $A'$ is the transpose of a matrix $A\in\R^{m\times d}$. Further, for any interval $I$ with infimum $a$ and supremum $b$ and each monotone function $f:I\rightarrow\R$, we set $f(a):=\lim_{v\downarrow a}f(v)$, if $a\notin I$, and $f(b):= \lim_{v\uparrow b} f(v)$, if $b\notin I$.

\subsection{Admissible Polish spaces of Borel probability measures}

We consider a tractable class of spaces of Borel probability measures on $\R^{m}$, which serve as part of the domain of the random drift and diffusion coefficients of the McKean-Vlasov equation~\eqref{eq:McKean-Vlasov}.

\begin{Definition}
A metrisable space $\mathcal{P}$ in $\mathcal{P}(\R^{m})$ is called \emph{admissible} if for any metrisable space $S$, each probability space $(\tilde{\Omega},\tilde{\mathcal{F}},\tilde{P})$ and every process $X:S\times\tilde{\Omega}\rightarrow\R^{m}$ with continuous paths satisfying
\begin{equation*}
\mathcal{L}(X_{s})\in\mathcal{P}\quad\text{for all $s\in S$,}
\end{equation*}
the map $S\rightarrow\mathcal{P}$, $s\mapsto\mathcal{L}(X_{s})$ is Borel measurable.
\end{Definition}

To give sufficient conditions for the admissibility of a metrisable space $\mathcal{P}$ in $\mathcal{P}(\R^{m})$, we introduce \emph{stochastic convergence of probability measures}. Namely, a sequence $(\mu_{n})_{n\in\mathbb{N}}$ in $\mathcal{P}(\R^{m})$ converges stochastically to some $\mu\in\mathcal{P}(\R^{m})$ if there is a sequence $(\theta_{n})_{n\in\mathbb{N}}$ of Borel measures on $\R^{m}\times\R^{m}$ such that $\theta_{n}\in\mathcal{P}(\mu_{n},\mu)$ for each $n\in\N$ and
\begin{equation}\label{eq:stochastic convergence of measures}
\lim_{n\uparrow\infty}\theta_{n}\big(\{(x,y)\in\R^{m}\times\R^{m}\,|\, |x-y|\geq \delta\}\big) = 0\quad\text{for all $\delta > 0$.}
\end{equation}
Thereby, $\mathcal{P}(\mu,\nu)$ denotes the convex space of all Borel probability measures $\theta$ on $\R^{m}\times\R^{m}$ with first and second marginal distributions $\mu$ and $\nu$, respectively, for all $\mu,\nu\in\mathcal{P}(\R^{m})$.
 
\begin{Proposition}\label{pr:admissibility}
Let $(\varphi_{n})_{n\in\mathbb{N}}$ be a sequence of $\R^{m}$-valued bounded uniformly continuous maps on $\R^{m}$ such that $|\varphi_{k}|\leq |\varphi_{k+1}|$ and $\lim_{n\uparrow\infty} \varphi_{n}(x) = x$ for all $k\in\mathbb{N}$ and $x\in\mathbb{\R}^{m}$. Then $\mathcal{P}$ is admissible under the following two conditions:
\begin{enumerate}[(i)]
\item Each $\mu\in\mathcal{P}$ satisfies $\mu\circ\varphi_{n}^{-1}\in\mathcal{P}$ for all $n\in\mathbb{N}$ and $\lim_{n\uparrow\infty} \mu\circ\varphi_{n}^{-1} = \mu$ in $\mathcal{P}$.

\item For any sequence $(\mu_{k})_{k\in\mathbb{N}}$ in $\mathcal{P}$ that converges stochastically to some $\mu\in\mathcal{P}$ we have $\lim_{k\uparrow\infty}\mu_{k}\circ\varphi_{n}^{-1} =\mu\circ\varphi_{n}^{-1}$ in $\mathcal{P}$ for any $n\in\mathbb{N}$.
\end{enumerate}
\end{Proposition}

\begin{Example}\label{ex:radial retraction}
We may let $\varphi_{n}$ be the radial retraction of the closed ball with center $0$ and radius $n$ for each $n\in\N$. That is, $\varphi_{n}(x) = x$, if $|x| \leq n$, and $\varphi_{n}(x) = \frac{n}{|x|}x$, if $|x| > n$, for any $x\in\R^{m}$. Indeed, $\varphi_{n}$ is Lipschitz continuous,
\begin{equation*}
|\varphi_{n}| \leq n\wedge |\varphi_{n+1}|\quad\text{and}\quad |\varphi_{n}(x) - x|= (|x|-n)^{+}.
\end{equation*}
\end{Example}

To give more concrete conditions, let $\rho:\R_{+}\rightarrow\R_{+}$ be increasing and $\mathcal{P}_{\rho}(\R^{m})$ denote the convex space of all $\mu\in\mathcal{P}(\R^{m})$ for which $\int_{\R^{m}}\!\rho(|x|)\,\mu(dx)$ is finite. 

\begin{Corollary}\label{co:admissibility}
Assume that $\mathcal{P}\subseteq\mathcal{P}_{\rho}(\R^{m})$, $\rho$ is continuous and $\rho(0) = 0$. Then the admissibility of $\mathcal{P}$ holds under the following two conditions:
\begin{enumerate}[(i)]
\item If $\varphi:\R^{m}\rightarrow\R^{m}$ is bounded and uniformly continuous and $|\varphi(x)|\leq |x|$ for all $x\in\R^{m}$, then $\mu\circ\varphi^{-1}\in\mathcal{P}$ for any $\mu\in\mathcal{P}$.
\item A sequence $(\mu_{n})_{n\in\N}$ in $\mathcal{P}$ converges to some $\mu\in\mathcal{P}$ if there is a sequence $(\theta_{n})_{n\in\N}$ of Borel measures on $\R^{m}\times\R^{m}$ such that $\theta_{n}\in\mathcal{P}(\mu_{n},\mu)$ for all $n\in\N$ and
\begin{equation*}
\lim_{n\uparrow\infty}\int_{\R^{m}\times\R^{m}}\!\rho(|x-y|)\,d\theta_{n}(x,y) = 0.
\end{equation*}
\end{enumerate}
\end{Corollary}

\begin{Remark}
By the measure transformation formula, any measurable map $\varphi:\R^{m}\rightarrow\R^{m}$ satisfies $\int_{\R^{m}}\!\rho(|x|)\,(\mu\circ\varphi^{-1})(dx) = \int_{\R^{m}}\!\rho(|\varphi(x)|)\,\mu(dx)$. Thus, the first condition of the corollary holds for $\mathcal{P}_{\rho}(\R^{m})$.
\end{Remark}

\begin{Examples}\label{ex:admissible spaces of measures} 
(i) The \emph{Prokhorov metric} $\vartheta_{P}$ turns $\mathcal{P}(\R^{m})$ into a Polish space and can be represented as follows: For $\varepsilon > 0$ let $N_{\varepsilon}(B)$ be the $\varepsilon$-neighbourhood of a set $B$ in $\R^{m}$. Then the $[0,1]$-valued functional on $\mathcal{P}(\R^{m})\times\mathcal{P}(\R^{m})$ given by
\begin{equation*}
\vartheta_{0}(\mu,\nu):=\inf\{\varepsilon > 0\,|\,\forall B\in\mathcal{B}(\R^{m}): \mu(B)\leq \nu(N_{\varepsilon}(B)) + \varepsilon\}
\end{equation*}
satisfies the triangle inequality and $\vartheta_{P}(\mu,\nu) = \vartheta_{0}(\mu,\nu)\vee \vartheta_{0}(\nu,\mu)$ for all $\mu,\nu\in\mathcal{P}(\R^{m})$. As convergence with respect to $\vartheta_{P}$ is equivalent to weak convergence, all requirements of Proposition~\ref{pr:admissibility} are met by $\mathcal{P}(\R^{m})$ if endowed with $\vartheta_{P}$.\smallskip

(ii) For $p\geq 1$ consider the Polish space $\mathcal{P}_{p}(\R^{m})$ of all $\mu\in\mathcal{P}(\R^{m})$ with finite $p$-th absolute moment $\int_{\R^{m}}\!|x|^{p}\,\mu(dx)$, equipped with the \emph{$p$-th Wasserstein metric} given by
\begin{equation}\label{eq:Wasserstein metric}
\vartheta_{p}(\mu,\nu) := \inf_{\theta\in\mathcal{P}(\mu,\nu)}\bigg(\int_{\R^{m}\times\R^{m}}\!|x - y|^{p}\,d\theta(x,y)\bigg)^{\frac{1}{p}}.
\end{equation}
As the definition of $\vartheta_{p}$ entails that the second condition of Corollary~\ref{co:admissibility} is valid for the choice $\rho(v) = v^{p}$ for all $v\geq 0$, we see that $\mathcal{P}_{p}(\R^{m})$ is admissible.
\end{Examples}

\subsection{Notions of solutions, pathwise uniqueness and stability}

In what follows, let $\mathcal{P}$ be an admissible separable metrisable space in $\mathcal{P}(\R^{m})$ and $\mathcal{A}$ denote the progressive $\sigma$-field on $[t_{0},\infty[\times\Omega$. Thus, a set $A$ in $[t_{0},\infty[\times\Omega$ lies in $\mathcal{A}$ if and only if $\mathbbm{1}_{A}$ is progressively measurable. We shall call a map
\begin{equation*}
F:[t_{0},\infty[\times\Omega\times\R^{m}\times\mathcal{P}\rightarrow\R^{m\times d},\quad (s,\omega,x,\mu)\mapsto F_{s}(x,\mu)(\omega)
\end{equation*}
\emph{admissible} if it is measurable relative to the product $\sigma$-field $\mathcal{A}\otimes\mathcal{B}(\R^{m})\otimes\mathcal{B}(\mathcal{P})$. In this case, for any $\R^{m}$-valued progressively measurable process $X$ and each Borel measurable map $\mu:[t_{0},\infty[\rightarrow\mathcal{P}$, the process
\begin{equation*}
[t_{0},\infty[\times\Omega\rightarrow \R^{m\times d}, \quad(s,\omega)\mapsto F_{s}(X_{s}(\omega),\mu(s))(\omega)
\end{equation*}
is progressively measurable. In particular, if $X$ is continuous and $\mathcal{L}(X_{t})\in\mathcal{P}$ for all $t\geq t_{0}$, then the \emph{law map} $[t_{0},\infty[\rightarrow\mathcal{P}$, $t\mapsto\mathcal{L}(X_{t})$ is a feasible choice for $\mu$. In this sense, we require the drift $\B$ and the diffusion $\Sigma$ of the McKean-Vlasov equation~\eqref{eq:McKean-Vlasov} to be admissible.

\begin{Definition}
A \emph{solution} to~\eqref{eq:McKean-Vlasov} is an $\R^{m}$-valued adapted continuous process $X$ such that $\mathcal{L}(X_{s})\in\mathcal{P}$ for all $s\geq t_{0}$, $
\int_{t_{0}}^{\cdot}\!|\B_{s}(X_{s},\mathcal{L}(X_{s}))| + |\Sigma_{s}(X_{s})|^{2}\,ds < \infty$ and
\begin{equation*}
X = X_{t_{0}} + \int_{t_{0}}^{\cdot}\!\B_{s}\big(X_{s},\mathcal{L}(X_{s})\big)\,ds + \int_{t_{0}}^{\cdot}\!\Sigma_{s}(X_{s})\,dW_{s}\quad\text{a.s.}
\end{equation*}
\end{Definition}

\begin{Example}\label{ex:controlled processes}
For $l\in\N$ let $Y$ be an $\R^{l}$-valued progressively measurable process and $b:[t_{0},\infty[\times\R^{m}\times\mathcal{P}\times\R^{l}\rightarrow\R^{m}$ and $\sigma:[t_{0},\infty[\times\R^{m}\times\R^{l}\rightarrow\R^{m\times d}$ be Borel measurable such that
\begin{equation*}
\B_{s}(x,\mu) = b(s,x,\mu,Y_{s})\quad\text{and}\quad\Sigma_{s}(x) = \sigma(s,x,Y_{s})
\end{equation*}
for any $(s,x,\mu)\in [t_{0},\infty[\times\R^{m}\times\mathcal{P}$. Then $\B$ and $\Sigma$ are indeed admissible and~\eqref{eq:McKean-Vlasov} turns into a McKean-Vlasov SDE whose drift and diffusion coefficients are \emph{controlled} by $Y$.
\end{Example}

We observe that $\B$ and $\Sigma$ are independent of $\omega\in\Omega$ if and only if there are two Borel measurable maps $b:[t_{0},\infty[\times\R^{m}\times\mathcal{P}\rightarrow\R^{m}$ and $\sigma:[t_{0},\infty[\times\R^{m}\rightarrow\R^{m\times d}$ such that
\begin{equation}\label{eq:deterministic coefficients}
\B_{s}(x,\mu) = b(s,x,\mu)\quad\text{and}\quad\Sigma_{s}(x) = \sigma(s,x)
\end{equation}
for all $(s,x,\mu)\in [t_{0},\infty[\times\R^{m}\times\mathcal{P}$. For the deterministic coefficients $b$ and $\sigma$, we may also consider \emph{solutions in the strong and weak sense} and write~\eqref{eq:McKean-Vlasov} formally as
\begin{equation}\label{eq:deterministic McKean-Vlasov}
dX_{t} = b\big(t,X_{t},\mathcal{L}(X_{t})\big)\,dt + \sigma(t,X_{t})\,dW_{t}\quad\text{for $t\geq t_{0}$.}
\end{equation}
Namely, for a fixed $\R^{m}$-valued $\mathcal{F}_{t_{0}}$-measurable random vector $\xi$ we let $(\mathcal{E}_{t}^{\xi})_{t\geq t_{0}}$ denote the natural filtration of the process $[t_{0},\infty[\times\Omega\rightarrow\R^{m}\times\R^{d}$, $(t,\omega)\mapsto (\xi,W_{t} - W_{t_{0}})(\omega)$. That means,
\begin{equation*}
\mathcal{E}_{t}^{\xi} =\sigma(\xi)\vee\sigma(W_{s} - W_{t_{0}}: s\in [t_{0},t])\quad\text{for all $t\geq t_{0}$.}
\end{equation*}
Then a solution $X$ to~\eqref{eq:deterministic McKean-Vlasov} satisfying $X_{t_{0}} = \xi$ a.s.~is called \emph{strong} if it is adapted to the right-continuous augmented filtration of $(\mathcal{E}_{t}^{\xi})_{t\geq t_{0}}$. A \emph{weak solution} to~\eqref{eq:deterministic McKean-Vlasov} is a solution $X$ defined on some filtered probability space
\begin{equation*}
(\tilde{\Omega},\tilde{\mathcal{F}},(\tilde{\mathcal{F}}_{t})_{t\geq 0},\tilde{P})
\end{equation*}
that satisfies the usual conditions and on which there exists a standard $d$-dimensional $(\tilde{\mathcal{F}}_{t})_{t\geq 0}$-Brownian motion $\tilde{W}$. That is, $X$ is an $\R^{m}$-valued $(\tilde{\mathcal{F}}_{t})_{t\geq t_{0}}$-adapted continuous process satisfying $\mathcal{L}(X_{s})\in\mathcal{P}$ for all $s\geq t_{0}$,
\begin{equation*}
\int_{t_{0}}^{\cdot}\!\big|b(s,X_{s},\mathcal{L}(X_{s})\big)\big| + |\sigma(s,X_{s})|^{2}\,ds < \infty
\end{equation*}
and $X = X_{t_{0}} + \int_{t_{0}}^{\cdot}\!b\big(s,X_{s},\mathcal{L}(X_{s})\big)\,ds + \int_{t_{0}}^{\cdot}\!\sigma(s,X_{s})\,d\tilde{W}_{s}$ a.s. In this case, we will say that $X$ solves~\eqref{eq:deterministic McKean-Vlasov} weakly relative to $\tilde{W}$. Let us now return to stochastic coefficients.

The regularity of the drift $\B$ relative to the variable $\mu\in\mathcal{P}$ will be stated in terms of an $\R_{+}$-valued Borel measurable functional $\vartheta$ on $\mathcal{P}\times\mathcal{P}$ satisfying
\begin{equation}\label{eq:specific domination condition}
\vartheta\big(\mathcal{L}(X),\mathcal{L}(\tilde{X})\big) \leq E\big[|X - \tilde{X}|\big]
\end{equation}
for any two $\R^{m}$-valued random vectors $X,\tilde{X}$ with $\mathcal{L}(X),\mathcal{L}(\tilde{X})\in\mathcal{P}$. For instance, this estimate holds if $\vartheta$ is \emph{dominated} by the Wasserstein metric $\vartheta_{1}$, introduced in Examples~\ref{ex:admissible spaces of measures}, in the sense that
\begin{equation}\label{eq:domination condition}
\vartheta(\mu,\nu) \leq \vartheta_{1}(\mu,\nu)\quad\text{for all $\mu,\nu\in\mathcal{P}$,}
\end{equation}
where the definition of $\vartheta_{1}$ in~\eqref{eq:Wasserstein metric} is extended for any $\mu,\nu\in\mathcal{P}(\R^{m})$ by allowing infinite values. If in fact $\mathcal{P}\subseteq\mathcal{P}_{1}(\R^{m})$, then this extension is not used. In this case, $\mathcal{P}$ is admissible as soon as $\vartheta$ is a metric inducing its topology and condition~(i) of Corollary~\ref{co:admissibility} holds.

\begin{Example}
Let $\phi:[-\infty,\infty[\rightarrow\R_{+}$ and $\varphi:\R^{m}\times\R^{m}\rightarrow\R$ be measurable, $\rho\in C(\R_{+})$ be increasing and vanish at $0$ and $c > 0$ be such that
\begin{equation*}
|\varphi(x,y)|\leq \rho(|x-y|)\quad\text{and}\quad \rho(v+w)/c \leq \rho(v) + \rho(w)
\end{equation*}
for all $x,y\in\R^{m}$ and $v,w\geq 0$. For instance, we may take $\rho(v) = \lambda v^{p}$ for each $v\geq 0$ and some $\lambda,p > 0$. Suppose that $\mathcal{P}\subseteq\mathcal{P}_{\rho}(\R^{m})$ and
\begin{equation*}
\vartheta(\mu,\nu) = \phi\bigg(\inf_{\theta\in\mathcal{P}(\mu,\nu)}\int_{\R^{m}\times\R^{m}}\!\varphi(x,y)\,d\theta(x,y)\bigg)\quad\text{for any $\mu,\nu\in\mathcal{P}$.}
\end{equation*}
Then $\vartheta$ is well-defined and the following three assertions are readily checked:
\begin{enumerate}[(1)]
\item If $\varphi(x,y) = \psi(x) - \psi(y)$ for all $x,y\in\R^{m}$ and some uniformly continuous function $\psi:\R^{m}\rightarrow\R$ that admits $\rho$ as modulus of continuity, then $\vartheta$ is of the form
\begin{equation*}
\vartheta(\mu,\nu) = \phi\bigg(\int_{\R^{m}}\!\psi(x)\,\mu(dx) - \int_{\R^{m}}\!\psi(y)\,\nu(dy)\bigg)\quad\text{for any $\mu,\nu\in\mathcal{P}$}.
\end{equation*}

\item In the case $\phi(v) = v$ for all $v\geq 0$, $\varphi(x,y) = |x-y|$ for any $x,y\in\R^{m}$ and $\rho(v) = v$ for each $v\geq 0$, we obtain that $\vartheta(\mu,\nu) = \vartheta_{1}(\mu,\nu)$ for any $\mu,\nu\in\mathcal{P}$.

\item Assume that $\phi(v) \leq v^{+}$ for all $v\in [-\infty,\infty[$ and $\rho(v) = v$ for each $v\geq 0$. Then the general domination estimate~\eqref{eq:domination condition} holds.
\end{enumerate}
\end{Example}

Let for the moment $\mathcal{P}=\mathcal{P}_{1}(\R^{m})$, $\vartheta = \vartheta_{1}$ and the following affine growth condition be valid: $|\B(\cdot,\mu)| \leq c_{0} +  c_{1}\vartheta_{1}(\mu,\delta_{0})$ for all $\mu\in\mathcal{P}_{1}(\R^{m})$ and some $c_{0},c_{1} \geq 0$, where $\delta_{0}$ is the Dirac measure in $0$. Then any $\R^{m}$-valued continuous process $X$ satisfies
\begin{equation*}
\int_{t_{0}}^{t}\!\big|\B_{s}\big(X_{s},\mathcal{L}(X_{s})\big)\big|\,ds \leq  \int_{t_{0}}^{t}\!c_{0} +  c_{1}E\big[|X_{s}|\big]\,ds\quad\text{for all $t\geq t_{0}$,}
\end{equation*}
by using that $\vartheta_{1}(\mu,\delta_{0}) = \int_{\R^{m}}|x|\,\mu(dx)$ for any $\mu\in\mathcal{P}_{1}(\R^{m})$. While the left-hand integral is finite if $X$ solves~\eqref{eq:McKean-Vlasov}, the right-hand integral may be infinite. Indeed, the \emph{absolute moment function}
\begin{equation*}
[t_{0},\infty[\rightarrow\R_{+},\quad s\mapsto E[|X_{s}|]
\end{equation*}
of $X$ is lower semicontinuous and hence, measurable, by Fatou's lemma, but it is not necessarily locally integrable. If $c_{1} = 0$, then $E[|X|]$ would even be locally bounded for any solution $X$ to~\eqref{eq:McKean-Vlasov}, by Lemmas~\ref{le:abstract growth estimate} and~\ref{le:moment growth estimate}.

However, to allow for \emph{growth in the measure variable}, we study uniqueness, stability and existence under a local integrability condition, which leads to more generality. To this end, let $\Theta$ be an $[0,\infty]$-valued functional on $[t_{0},\infty[\times\mathcal{P}\times\mathcal{P}\times\mathcal{P}(\R^{m})$.

\begin{Definition}\label{de:pathwise uniqueness}
We say that \emph{pathwise uniqueness} holds for~\eqref{eq:McKean-Vlasov} (relative to $\Theta$) if any two solutions $X$ and $\tilde{X}$ with $X_{t_{0}} = \tilde{X}_{t_{0}}$ a.s.~(and for which the function
\begin{equation}\label{eq:integrable function}
[t_{0},\infty[\rightarrow [0,\infty],\quad s\mapsto\Theta\big(s,\mathcal{L}(X_{s}),\mathcal{L}(\tilde{X}_{s}),\mathcal{L}(X_{s}-\tilde{X}_{s})\big)
\end{equation}
is measurable and locally integrable) are indistinguishable.
\end{Definition}

\begin{Example}
Suppose that $\rho,\varrho\in C(\R_{+})$ are positive on $]0,\infty[$ and vanish at $0$, $\rho$ is concave and $\eta,\lambda:[t_{0},\infty[\rightarrow\R_{+}$ are measurable and locally integrable such that
\begin{equation*}
\Theta(s,\mu,\tilde{\mu},\nu) = \lambda(s) \varrho\big(\vartheta(\mu,\tilde{\mu})\big) + \eta(s)\int_{\R^{m}}\!\rho(|x|)\,\nu(dx)
\end{equation*}
for all $s\geq t_{0}$, $\mu,\tilde{\mu}\in\mathcal{P}$ and $\nu\in\mathcal{P}(\R^{m})$. Then $\Theta(s,\mu,\tilde{\mu},\nu)$ is finite if $\nu\in\mathcal{P}_{1}(\R^{m})$, by Jensen's inequality, and for any two continuous processes $X$ and $\tilde{X}$ with $\mathcal{L}(X_{s}),\mathcal{L}(\tilde{X}_{s})\in\mathcal{P}$ we have
\begin{equation*}
\Theta\big(s,\mathcal{L}(X_{s}),\mathcal{L}(\tilde{X}_{s}),\mathcal{L}(X_{s}-\tilde{X}_{s})\big) = \eta(s) \varrho\big(\vartheta(\mathcal{L}(X_{s}),\mathcal{L}(\tilde{X}_{s}))\big) + \lambda(s) E\big[\rho(|X-\tilde{X}|)\big].
\end{equation*}
Thus, the measurable function~\eqref{eq:integrable function} is locally integrable as soon as $E[|X-\tilde{X}|]$ is locally bounded.
\end{Example}

Now we present \emph{generalised notions of stability} for~\eqref{eq:McKean-Vlasov} in a \emph{global sense}, which apply directly without shifting the stochastic drift and diffusion. Namely, in the literature for stability of SDEs, it is a convenient assumption that drift and diffusion vanish at all times at the origin of $\R^{m}$, ensuring that the constant zero process is a solution.

If a reader seeks to use stability results for McKean-Vlasov SDEs and the required normalisations $\B(0,\delta_{0}) =0$ and $\Sigma(0) = 0$ fail, then there should exist at least one solution $\hat{X}$ to~\eqref{eq:McKean-Vlasov}. In this case, the maps $\hat{\B}$ and $\hat{\Sigma}$ on $[t_{0},\infty[\times\Omega\times\R^{m}\times\mathcal{P}$ and $[t_{0},\infty[\times\Omega\times\R^{m}$ with values in $\R^{m}$ and $\R^{m\times d}$, respectively, given by
\begin{equation*}
\hat{\B}_{t}(x,\mu) := \B_{t}(\hat{X}_{t} + x, \hat{l}(t,\mu)) - \B_{t}\big(\hat{X}_{t},\mathcal{L}(\hat{X}_{t})\big)\quad\text{and}\quad \hat{\Sigma}_{t}(x) := \Sigma_{t}(\hat{X}_{t} + x) - \Sigma_{t}(\hat{X}_{t})
\end{equation*}
are admissible and satisfy $\hat{\B}(0,\delta_{0}) = 0$ and $\hat{\Sigma}(0)=0$ for any fixed Borel measurable map $\hat{l}:[t_{0},\infty[\times\mathcal{P}\rightarrow\mathcal{P}$ such that $\hat{l}(t,\delta_{0}) = \mathcal{L}(\hat{X}_{t})$ for all $t\geq t_{0}$. In effect, the reader is forced to replace the drift $\B$ and the diffusion $\Sigma$ by $\hat{\B}$ and $\hat{\Sigma}$, respectively, and use stability concepts that are stated in terms of the particular solution $\hat{X}$.

Further, even if $\B$ and $\Sigma$ were deterministic as in~\eqref{eq:deterministic coefficients}, the coefficients $\hat{\B}$ and $\hat{\Sigma}$ would in general become random, unless $\hat{X}_{t}$ is constant for each $t\geq t_{0}$. This translation procedure may certainly have its justification for a local stability analysis, but, as it is not necessary for \emph{global comparisons of solutions}, we do not apply it.

\begin{Definition}\label{de:stability}
Let $\alpha > 0$.
\begin{enumerate}[(i)]
\item We say that~\eqref{eq:McKean-Vlasov} is \emph{stable in moment} (with respect to $\Theta$) if any two solutions $X$ and $\tilde{X}$ (for which the function~\eqref{eq:integrable function} is measurable and locally integrable) satisfy
\begin{equation*}
\sup_{t\geq t_{0}} E\big[|X_{t} - \tilde{X}_{t}|\big] < \infty
\end{equation*}
under the condition that $E[|X_{t_{0}} - \tilde{X}_{t_{0}}|] < \infty$. If in addition $\lim_{t\uparrow\infty} E[|X_{t} - \tilde{X}_{t}|] = 0$, then we speak about \emph{asymptotic stability in moment}.

\item Equation~\eqref{eq:McKean-Vlasov} is said to be \emph{$\alpha$-exponentially stable in moment} (relative to $\Theta$) if there are $\lambda < 0$ and $c\geq 0$ such that for any two solutions $X$ and $\tilde{X}$ to~\eqref{eq:McKean-Vlasov},
\begin{equation}\label{eq:exponential moment stability}
E\big[|X_{t} - \tilde{X}_{t}|\big] \leq c e^{\lambda (t-t_{0})^{\alpha}} E\big[|X_{t_{0}} - \tilde{X}_{t_{0}}|\big]\quad\text{for all $t\geq t_{0}$}
\end{equation}
(whenever~\eqref{eq:integrable function} is measurable and locally integrable). In this case, $\lambda$ is said to be a \emph{moment $\alpha$-Lyapunov exponent} for~\eqref{eq:McKean-Vlasov}.

\item We call~\eqref{eq:McKean-Vlasov} \emph{pathwise $\alpha$-exponentially stable} (relative to an initial absolute moment and $\Theta$) if there is $\lambda < 0$ such that for any two solutions $X$ and $\tilde{X}$ we have
\begin{equation*}
\limsup_{t\uparrow\infty} \frac{1}{t^{\alpha}} \log\big(|X_{t} - \tilde{X}_{t}|\big) \leq \lambda\quad \text{a.s.}
\end{equation*}
(as soon as $E[|X_{t_{0}} - \tilde{X}_{t_{0}}|] < \infty$ and~\eqref{eq:integrable function} is measurable and locally integrable). If this is the case, then $\lambda$ is called a \emph{pathwise $\alpha$-Lyapunov exponent} for~\eqref{eq:McKean-Vlasov}.
\end{enumerate}
\end{Definition}

\begin{Remark}\label{re:Lyapunov exponent}
If $\psi,u:[t_{0},\infty[\rightarrow\R_{+}$ are locally bounded, $\psi > 0$ on $]t_{0},\infty[$ and $\lambda\in\R$ satisfies $\limsup_{t\uparrow\infty} \frac{1}{\psi(t)}\log(u(t))\leq \lambda$, then for any $\varepsilon > 0$ there is $c_{\varepsilon} > 0$ such that
\begin{equation*}
u(t) \leq c_{\varepsilon} e^{\psi(t)(\lambda + \varepsilon)}\quad\text{for all $t\geq t_{0}$.}
\end{equation*}
Thus, while the exponential estimate~\eqref{eq:exponential moment stability} for two given solutions $X$ and $\tilde{X}$ to~\eqref{eq:McKean-Vlasov} satisfying $E[|X_{t_{0}} - \tilde{X}_{t_{0}}|] < \infty$ readily implies
\begin{equation*}
\limsup_{t\uparrow\infty}\frac{1}{t^{\alpha}}\log\big(E\big[|X_{t} - \tilde{X}_{t}|\big]\big) \leq \lambda,
\end{equation*}
the latter bound entails the former only when $\lambda$ is replaced by $\lambda +\varepsilon$ for any $\varepsilon > 0$, provided $E[|X-\tilde{X}|]$ is locally bounded.
\end{Remark}

Note that for our general equation~\eqref{eq:McKean-Vlasov} with random coefficients $\B$ and $\Sigma$, we have to formulate Definitions~\ref{de:pathwise uniqueness} and~\ref{de:stability} for uniqueness and stability with respect to the underlying probability space. These concepts carry over to the case~\eqref{eq:deterministic coefficients} of deterministic coefficients, applied to each filtered probability space satisfying the usual conditions and on which there is a standard $d$-dimensional Brownian motion. 

To be precise,~pathwise uniqueness holds for~\eqref{eq:deterministic McKean-Vlasov} in the usual sense (relative to $\Theta$) if any two weak solutions $X$ and $\tilde{X}$ on a common filtered probability space relative to one standard $d$-dimensional Brownian motion with $X_{t_{0}} = \tilde{X}_{t_{0}}$ a.s.~(and for which~\eqref{eq:integrable function} is measurable and locally integrable) are indistinguishable. Similarly, by considering weak solutions on common filtered probability spaces instead of solutions on the underlying space, \emph{each notion of stability applies to~\eqref{eq:deterministic McKean-Vlasov}}.

\section{Main results}\label{se:3}

\subsection{A quantitative first moment estimate and pathwise uniqueness}

We seek to compare solutions to~\eqref{eq:McKean-Vlasov} with varying drifts and thereby show pathwise uniqueness. For this purpose, let the map
\begin{equation*}
\tilde{\B}:[t_{0},\infty[\times\Omega\times\R^{m}\times\mathcal{P}\rightarrow\R^{m}
\end{equation*}
be admissible. For $p\geq 1$ we define two sublinear functionals $[\cdot]_{p}$ and $[\cdot]_{\infty}$ with respective values in $[0,\infty]$ and $]-\infty,\infty]$ on the linear space of all random variables by
\begin{equation*}
[X]_{p} :=E\big[(X^{+})^{p}\big]^{\frac{1}{p}}\quad\text{and}\quad [X]_{\infty}:=\esssup X.
\end{equation*}
Note that if $\alpha,\beta\in [0,1]$ satisfy $\alpha + \beta\leq 1$ and $X$ and $Y$ are two random variables such that $Y\geq 0$, $[X]_{\frac{1}{1-\alpha}} < \infty$ and $E[Y] < \infty$, then the inequalities of H\oe lder and Young entail that $XY^{\alpha}$ is quasi-integrable and
\begin{equation}\label{eq:essential inequality}
E[XY^{\alpha}]E[Y]^{\beta} \leq [X]_{\frac{1}{1-\alpha}}\big(1 - (\alpha + \beta) + (\alpha + \beta)E[Y]\big).
\end{equation}
This bound leads to the quantitative $L^{1}$-estimates of Theorems~\ref{th:abstract moment estimate} and~\ref{th:moment stability estimate}, on which our main results are based. As first requirement, we introduce an \emph{Osgood continuity condition on compact sets} for the rows of $\Sigma$.
\begin{enumerate}[label=(C.\arabic*), ref=C.\arabic*, leftmargin=\widthof{(C.1)} + \labelsep]
\item\label{co:1} For any $n\in\N$ there are an increasing $\hat{\rho}_{n}\in C(\R_{+})$ and an $\R_{+}$-valued progressively measurable process $\hat{\eta}^{(n)}$ with locally square-integrable paths such that
\begin{equation*}
\text{$\hat{\rho}_{n} > 0$ on $]0,\infty[$,}\quad \int_{0}^{1}\!\frac{1}{\hat{\rho}_{n}(v)^{2}}\,dv = \infty\quad\text{and}\quad\big|e_{i}'\big(\Sigma(x) - \Sigma(\tilde{x})\big)\big| \leq \hat{\eta}^{(n)}\hat{\rho}_{n}(|x_{i} -\tilde{x}_{i}|)
\end{equation*}
for all $x,\tilde{x}\in\R^{m}$ with $|x|\vee|\tilde{x}|\leq n$ a.s.~for any $i\in\{1,\dots,m\}$.
\end{enumerate}

\begin{Remark}
The condition forces the $i$-th row of $\Sigma$ to depend only on the $i$-th coordinate of the space variable $x\in\R^{m}$ a.s.~for all $i\in\{1,\dots,m\}$. This is the case if and only if
\begin{equation*}
\Sigma(x) = 
\begin{pmatrix}
\hat{\Sigma}^{(1,1)}(x_{1}) & \cdots & \hat{\Sigma}^{(1,d)}(x_{1})\\
\vdots & \ddots & \vdots\\
\hat{\Sigma}^{(m,1)}(x_{m}) & \cdots & \hat{\Sigma}^{(m,d)}(x_{m})
\end{pmatrix}
\quad\text{for any $x\in\R^{m}$}\quad\text{a.s.}
\end{equation*}
for some admissible map $\hat{\Sigma}:[t_{0},\infty[\times\Omega\times\R\rightarrow\R^{m\times d}$. Further, for every $n\in\N$ and $\alpha_{n}\in [\frac{1}{2},1]$, we may take $\hat{\rho}_{n}(v) = v^{\alpha_{n}}$ for all $v\geq 0$ as appearing modulus of continuity.
\end{Remark}

\begin{Example}\label{ex:sum of power functions}
In the \emph{one-dimensional setting} $m=1$ let $l\in\N$, $\varphi:\R\rightarrow\R^{d\times l}$ be locally $\frac{1}{2}$-H\oe lder continuous and $\zeta$ and $\eta$ be progressively measurable processes with values in $\R^{d}$ and $\R^{d\times l}$, respectively, such that
\begin{equation*}
\Sigma^{(i)}(x) = \zeta^{(i)} + \eta^{(i,1)}\varphi_{i,1}(x) + \cdots + \eta^{(i,l)}\varphi_{i,l}(x) 
\end{equation*}
for any $x\in\R$ and $i\in\{1,\dots,d\}$. Then~\eqref{co:1} is satisfied as soon as $\eta$ admits locally square-integrable paths. In particular, for $\alpha\in [\frac{1}{2},\infty[^{d\times l}$ the choice
\begin{equation*}
\varphi_{i,1}(x) = |x|^{\alpha_{i,1}},\dots,\varphi_{i,l}(x) = |x|^{\alpha_{i,l}}
\end{equation*}
for all $x\in\R$ and $i\in\{1,\dots,d\}$ is possible.
\end{Example}

Next, we consider a \emph{partial uniform error and continuity condition} on the coordinates of $\B$ and $\tilde{\B}$, which allows for discontinuities in the space variable.
\begin{enumerate}[label=(C.\arabic*), ref=C.\arabic*, leftmargin=\widthof{(C.2)} + \labelsep]
\setcounter{enumi}{1}
\item\label{co:2} There are $\rho,\varrho\in C(\R_{+})$ that are positive on $]0,\infty[$ and vanish at $0$ and $\R_{+}^{m}$-valued progressively measurable processes $\varepsilon,\eta,\lambda$ with locally integrable paths such that 
\begin{equation*}
\begin{split}
\mathrm{sgn}(x_{i}-\tilde{x}_{i})\big(\B^{(i)}(x,\mu) - \tilde{\B}^{(i)}(\tilde{x},\tilde{\mu})\big) &\leq \varepsilon^{(i)} + \eta^{(i)}\rho(|x-\tilde{x}|_{1}) + \lambda^{(i)}\varrho\big(\vartheta(\mu,\tilde{\mu})\big)
\end{split}
\end{equation*}
for all $x,\tilde{x}\in\R^{m}$ and $\mu,\tilde{\mu}\in\mathcal{P}$ a.s.~for any $i\in\{1,\dots,m\}$. Further, $\rho^{\frac{1}{\alpha}}$ is concave for some $\alpha\in ]0,1]$, $\varrho$ is increasing and $E[|\varepsilon|_{1}]$, $[|\eta|_{1}]_{\frac{1}{1-\alpha}}$, $E[|\lambda|_{1}]$ are locally integrable.
\end{enumerate}

\begin{Remark}
If $\B = \tilde{\B}$ and $\varepsilon = 0$, then~\eqref{co:2} reduces to a \emph{partial uniform continuity condition} on $\B$. In the particular case that there are $\alpha_{0},\beta_{0}\in ]0,1]$ such that
\begin{equation*}
\alpha_{0}\leq \alpha,\quad \rho(v) = v^{\alpha_{0}}\quad\text{and}\quad \varrho(v) = v^{\beta_{0}}\quad\text{for all $v\geq 0$,}
\end{equation*}
we obtain a \emph{partial H\oe lder condition}. Following this reasoning, the term $\varepsilon$ provides an \emph{error estimate}. That is, if~\eqref{co:2} holds for $\B = \tilde{\B}$ and $\varepsilon = 0$, then it is valid in general as soon as $|\B^{(i)} - \tilde{\B}^{(i)}| \leq \varepsilon^{(i)}$ for any $i\in\{1,\dots,m\}$.
\end{Remark}

As our first estimation result is based on Bihari's inequality, we recall that for any $\rho\in C(\R_{+})$ that is positive on $]0,\infty[$ and vanishes at $0$, the function $\Phi_{\rho}\in C^{1}(]0,\infty[)$ defined via
\begin{equation}\label{eq:rho-map 1}
\Phi_{\rho}(w) := \int_{1}^{w}\!\frac{1}{\rho(v)}\,dv
\end{equation}
is a strictly increasing $C^{1}$-diffeomorphism onto the interval $]\Phi_{\rho}(0),\Phi_{\rho}(\infty)[$. Let $D_{\rho}$ denote the set of all $(v,w)\in\R_{+}^{2}$ with $\Phi_{\rho}(v) + w < \Phi_{\rho}(\infty)$. Then $\Psi_{\rho}:D_{\rho}\rightarrow\R_{+}$ given by
\begin{equation}\label{eq:rho-map 2}
\Psi_{\rho}(v,w) := \Phi_{\rho}^{-1}\big(\Phi_{\rho}(v) + w\big)
\end{equation}
is a continuous extension of a locally Lipschitz continuous function and it is increasing in each variable.

Based on these considerations, we obtain a \emph{quantitative $L^{1}$-bound} under~\eqref{co:2} by introducing for fixed $\beta\in ]0,1]$ the two measurable locally integrable functions
\begin{equation*}
\gamma := \alpha \big[|\eta|_{1}\big]_{\frac{1}{1-\alpha}} + \beta E\big[|\lambda|_{1}\big] \quad\text{and}\quad \delta := (1-\alpha)\big[|\eta|_{1}\big]_{\frac{1}{1-\alpha}} + (1-\beta)E\big[|\lambda|_{1}\big].
\end{equation*}

\begin{Proposition}\label{pr:abstract stability estimate}
Let~\eqref{co:1} and~\eqref{co:2} hold, $X$ and $\tilde{X}$ be two solutions to~\eqref{eq:McKean-Vlasov} with respective drifts $\B$ and $\tilde{\B}$ such that
\begin{equation*}
E\big[|Y_{t_{0}}|_{1}\big] < \infty\quad\text{for $Y:=X-\tilde{X}$}
\end{equation*}
and $E[|\lambda|_{1}]\varrho(\vartheta(\mathcal{L}(X),\mathcal{L}(\tilde{X}))$ is locally integrable. Define $\varrho_{0}\in C(\R_{+})$ by
\begin{equation*}
\varrho_{0}(v) := \rho(v)^{\frac{1}{\alpha}}\vee\varrho(v)^{\frac{1}{\beta}}
\end{equation*}
and suppose that $\Phi_{\rho^{\frac{1}{\alpha}}}(\infty) = \infty$ or $E[|\eta|_{1}\rho(|Y|_{1})]$ is locally integrable. Then $E[|Y|_{1}]$ is locally bounded and
\begin{equation*}
\sup_{s\in [t_{0},t]} E\big[|Y_{s}|_{1}\big] \leq \Psi_{\varrho_{0}}\bigg(E\big[|Y_{t_{0}}|_{1}\big] + \int_{t_{0}}^{t}\!E\big[|\varepsilon_{s}|_{1}\big] + \delta(s)\,ds,\int_{t_{0}}^{t} \!\gamma(s)\,ds\bigg)
\end{equation*}
for any $t\in [t_{0},t_{0}^{+}[$, where $t_{0}^{+} > t_{0}$ stands for the supremum over all $t\geq t_{0}$ for which
\begin{equation*}
\bigg(E\big[|Y_{t_{0}}|_{1}\big] + \int_{t_{0}}^{t}\!E\big[|\varepsilon_{s}|_{1}\big] + \delta(s)\,ds,\int_{t_{0}}^{t} \!\gamma(s)\,ds\bigg)\in D_{\varrho_{0}}.
\end{equation*}
\end{Proposition}

\begin{Remark}
If in fact $\Phi_{\varrho_{0}}(\infty) = \infty$, then $\Phi_{\rho^{\frac{1}{\alpha}}}(\infty) = \infty$ and $D_{\varrho_{0}} = \R_{+}^{2}$. In this case, $Y$ is bounded in $L^{1}(\Omega,\mathcal{F},P)$ as soon as $E[|\varepsilon|_{1}]$, $\gamma$ and $\delta$ are integrable. Further, if
\begin{equation*}
\Phi_{\varrho_{0}}(0) = - \infty,\quad Y_{t_{0}} = 0\quad\text{a.s.}\quad\text{and}\quad E\big[|\varepsilon|_{1}\big] = \delta = 0\quad\text{a.e.},
\end{equation*}
then $t_{0}^{+} = \infty$ and $Y=0$ a.s. This implication will be used to deduce pathwise uniqueness.
\end{Remark}

\begin{Example}\label{ex:power functions}
Let $\alpha_{0}\in ]0,\alpha]$ be such that $\rho(v) = \varrho(v) = v^{\alpha_{0}}$ for all $v\geq 0$ and $\alpha = \beta$. Then $\Phi_{\varrho_{0}}(\infty) = \infty$ and in the case $\alpha_{0} < \alpha$ we get that
\begin{equation*}
\Psi_{\varrho_{0}}(v,w) = \bigg(v^{1-\hat{\alpha}} + (1 - \hat{\alpha})w\bigg)^{\frac{1}{1-\hat{\alpha}}}\quad\text{for all $v,w\geq 0$}
\end{equation*}
with $\hat{\alpha}:= \frac{\alpha_{0}}{\alpha}$. If instead $\alpha_{0}=\alpha$, then $\Psi_{\varrho_{0}}(v,w) = v\exp(w)$ for any $v,w\geq 0$. Thus, Proposition~\ref{pr:abstract stability estimate} provides an estimate for any choice of $\hat{\alpha}\in ]0,1]$.
\end{Example}

To infer pathwise uniqueness for~\eqref{eq:McKean-Vlasov} from the comparison, we specify~\eqref{co:2} for $\B=\tilde{\B}$, $\varepsilon = 0$, $\alpha=1$ and a deterministic choice of $\eta$. Further, if $\B$ does not depend on $\mu\in\mathcal{P}$, then it suffices to pose this condition on compact sets only.
\begin{enumerate}[label=(C.\arabic*), ref=C.\arabic*, leftmargin=\widthof{(C.3)} + \labelsep]
\setcounter{enumi}{2}
\item\label{co:3} There are $\rho,\varrho\in C(\R_{+})$ that are positive on $]0,\infty[$ and vanish at $0$, a measurable locally integrable function $\eta:[t_{0},\infty[\rightarrow\R_{+}$ and an $\R_{+}^{m}$-valued progressively measurable process $\lambda$ with locally integrable paths such that
\begin{equation*}
\mathrm{sgn}(x_{i} - \tilde{x}_{i})\big(\B^{(i)}(x,\mu) - \B^{(i)}(\tilde{x},\tilde{\mu})\big)\leq \eta \rho(|x-\tilde{x}|_{1}) + \lambda^{(i)}\varrho\big(\vartheta(\mu,\tilde{\mu})\big)
\end{equation*}
for any $x,\tilde{x}\in\R^{m}$ and $\mu,\tilde{\mu}\in\mathcal{P}$ a.s.~for all $i\in\{1,\dots,m\}$. In addition, $\rho$ is concave, $\varrho$ is increasing and $E[|\lambda|_{1}]$ is locally integrable.

\item\label{co:4} $\B$ is independent of $\mu\in\mathcal{P}$ and for any $n\in\N$ there are a concave $\rho_{n}\in C(\R_{+})$ that is positive on $]0,\infty[$ and vanishes at $0$ and a measurable locally integrable function $\eta_{n}:[t_{0},\infty[\rightarrow\R_{+}$ satisfying
\begin{equation*}
\mathrm{sgn}(x_{i} - \tilde{x}_{i})\big(\hat{\B}^{(i)}(x) - \hat{\B}^{(i)}(\tilde{x})\big) \leq \eta_{n}\rho_{n}(|x-\tilde{x}|_{1}) 
\end{equation*}
for all $x,\tilde{x}\in\R^{m}$ with $|x|\vee |\tilde{x}| \leq n$ a.s.~for any $i\in\{1,\dots,m\}$, where $\hat{\B} := \B(\cdot,\hat{\mu})$ for some $\hat{\mu}\in\mathcal{P}$.
\end{enumerate}

Under~\eqref{co:1} and~\eqref{co:3}, pathwise uniqueness can be shown with respect to the Borel measurable functional $\Theta:[t_{0},\infty[\times\mathcal{P}\times\mathcal{P}\times\mathcal{P}(\R^{m})\rightarrow [0,\infty]$ defined by
\begin{equation*}
\Theta(s,\mu,\tilde{\mu},\nu):= E\big[|\lambda_{s}|_{1}\big]\varrho\big(\vartheta(\mu,\tilde{\mu})\big) + \mathbbm{1}_{]0,\infty[}\big(\Phi_{\rho}(\infty)\big)\eta(s)\int_{\R^{m}}\rho(|y|_{1})\,\nu(dy).
\end{equation*}

\begin{Corollary}\label{co:pathwise uniqueness}
Suppose that~\eqref{co:1} is satisfied.
\begin{enumerate}[(i)]
\item Let~\eqref{co:3} be valid and $\int_{0}^{1}\frac{1}{\rho(v)\vee\varrho(v)}\,dv = \infty$. Then pathwise uniqueness holds for~\eqref{eq:McKean-Vlasov} relative to $\Theta$.
\item If~\eqref{co:4} is satisfied and $\int_{0}^{1}\!\frac{1}{\rho_{n}(v)}\,dv = \infty$ for each $n\in\N$, then pathwise uniqueness for the SDE~\eqref{eq:McKean-Vlasov} follows.
\end{enumerate}
\end{Corollary}

\begin{Remark}\label{re:pathwise uniqueness}
Let $\B$ and $\Sigma$ be deterministic, that is,~\eqref{eq:deterministic coefficients} holds. Then the corollary yields pathwise uniqueness for~\eqref{eq:deterministic McKean-Vlasov} in the standard sense if the conditions are specified as follows:
\begin{enumerate}[(1)]
\item The Osgood condition~\eqref{co:1} on compact sets is formulated when $\hat{\eta}^{(n)}$ is independent of $\omega\in\Omega$ for all $n\in\N$.

\item The partial uniform continuity condition~\eqref{co:3} is stated when $\lambda$ is deterministic and the required estimate~\eqref{eq:specific domination condition} for $\vartheta$ is replaced by the domination condition~\eqref{eq:domination condition}.
\end{enumerate}
\end{Remark}

Let us consider a \emph{class of drift maps} to which these uniqueness results apply.

\begin{Example}\label{ex:class of drift maps}
Let $F$ and $G$ be two $\R^{m}$-valued admissible maps on $[t_{0},\infty[\times\Omega\times\R$ and $[t_{0},\infty[\times\Omega\times\R^{m}\times\mathcal{P}$, respectively, such that
\begin{equation*}
\B^{(i)}(x,\mu) = F^{(i)}(x_{i}) + G^{(i)}(x,\mu)
\end{equation*}
for all $(x,\mu)\in\R^{m}\times\mathcal{P}$ and $i\in\{1,\dots,m\}$. Then~\eqref{co:3} and~\eqref{co:4} are implied by the following respective conditions:
\begin{enumerate}[(1)]
\item There is $\varrho\in C(\R_{+})$ that is positive on $]0,\infty[$ and vanishes at $0$, a measurable locally integrable function $\eta:[t_{0},\infty[\rightarrow\R_{+}$ and an $\R_{+}^{m}$-valued progressively measurable process $\lambda$ such that
\begin{align*}
\mathrm{sgn}(v - \tilde{v})\big(F^{(i)}(v) - F^{(i)}(\tilde{v})\big) & \leq \eta |v-\tilde{v}|\\
\text{and}\quad |G^{(i)}(x,\mu) - G^{(i)}(\tilde{x},\tilde{\mu})| & \leq \eta |x-\tilde{x}| + \lambda^{(i)}\varrho\big(\vartheta(\mu,\tilde{\mu})\big)
\end{align*}
for any $v,\tilde{v}\in\R$, $x,\tilde{x}\in\R^{m}$, $\mu,\tilde{\mu}\in\mathcal{P}$ and $i\in\{1,\dots,m\}$. Moreover, $\varrho$ is increasing, $\lambda$ has locally integrable paths and $E[|\lambda|_{1}]$ is locally integrable.

\item $G$ is independent of $\mu\in\mathcal{P}$ and for each $n\in\N$ there are $\rho_{n}\in C(\R_{+})$ that is positive on $]0,\infty[$ and vanishes at $0$ and a measurable locally integrable function $\eta_{n}:[t_{0},\infty[\rightarrow\R_{+}$ such that
\begin{equation*}
\mathrm{sgn}(v - \tilde{v})\big(F^{(i)}(v) - F^{(i)}(\tilde{v})\big)\leq \eta_{n}\rho_{n}(|v-\tilde{v}|),\quad |\hat{G}(x) - \hat{G}(\tilde{x})|\leq \eta_{n} \rho_{n}(|x-\tilde{x}|)
\end{equation*}
for all $v,\tilde{v}\in [-n,n]$, $x,\tilde{x}\in\R^{m}$ with $|x|\vee|\tilde{x}|\leq n$ and $i\in\{1,\dots,m\}$, where $\hat{G}:=G(\cdot,\hat{\mu})$ for fixed $\hat{\mu}\in\mathcal{P}$. Further, $\rho_{n}$ is concave and increasing.
\end{enumerate}
For instance, for $l\in\N$, an $\R_{+}^{m\times l}$-valued progressively measurable process $\eta$ and $\alpha\in ]0,\infty[^{l}$, we could take
\begin{equation}\label{eq:power sum representation}
F^{(i)}(v) = - \eta^{(i,1)}(v^{+})^{\alpha_{1}} - \cdots - \eta^{(i,l)}(v^{+})^{\alpha_{l}} 
\end{equation}
for all $v\in\R$ and $i\in\{1,\dots,m\}$, in which case both conditions~(1) and~(2) are met by $F$. In the general case, if~\eqref{co:1} is imposed on $\Sigma$, then Corollary~\ref{co:pathwise uniqueness} entails two assertions:
\begin{enumerate}[(1)]
\setcounter{enumi}{2}
\item Let condition~(1) hold. Then we have pathwise uniqueness for~\eqref{eq:McKean-Vlasov} relative to the Borel measurable functional $\Theta:[t_{0},\infty[\times\mathcal{P}\times\mathcal{P}\rightarrow [0,\infty]$ given by
\begin{equation*}
\Theta(s,\mu,\tilde{\mu}):= E\big[|\lambda_{s}|_{1}\big]\varrho\big(\vartheta(\mu,\tilde{\mu})\big),\quad\text{provided}\quad \int_{0}^{1}\!\frac{1}{v\vee \varrho(v)}\,dv = \infty.
\end{equation*}
In particular, $\varrho(v) = v$ for all $v > 0$ and $\varrho(v)= \alpha v(|\log(v)| + 1)$ for any $v > 0$ with $\alpha >0$ are feasible choices.

\item There is pathwise uniqueness for the SDE~\eqref{eq:McKean-Vlasov} under condition~(2) as soon as $\int_{0}^{1}\!\frac{1}{\rho_{n}(v)}\,dv = \infty$ for each $n\in\N$.
\end{enumerate}
\end{Example}

\subsection{An explicit moment estimate and stability in first moment}

Now we provide a comparison bound, from which stability results in first moment can be inferred. In this regard, we require a \emph{partial uniform error and mixed H\oe lder continuity condition} on $\B$ and $\tilde{\B}$:
\begin{enumerate}[label=(C.\arabic*), ref=C.\arabic*, leftmargin=\widthof{(C.5)} + \labelsep]
\setcounter{enumi}{4}
\item\label{co:5} There are $l\in\N$, $\alpha,\beta\in ]0,1]^{l}$ and progressively measurable processes $\varepsilon$, $\eta$ and $\lambda$ with values in $\R_{+}^{m}$, $\R^{m\times m\times l}$ and $\R_{+}^{m\times l}$, respectively, such that
\begin{equation*}
\begin{split}
\mathrm{sgn}(x_{i} - \tilde{x}_{i})&\big(\B^{(i)}(x,\mu) - \tilde{\B}^{(i)}(\tilde{x},\tilde{\mu})\big) \leq \varepsilon^{(i)}\\
&\quad + \sum_{k=1}^{l}\bigg(\sum_{j=1}^{m}\eta^{(i,j,k)}|x_{j} - \tilde{x}_{j}|^{\alpha_{k}}\bigg) + \lambda^{(i,k)}\vartheta(\mu,\tilde{\mu})^{\beta_{k}}
\end{split}
\end{equation*}
for all $x,\tilde{x}\in\R^{m}$ and $\mu,\tilde{\mu}\in\mathcal{P}$ a.s.~for any $i\in\{1,\dots,m\}$. In addition, $\varepsilon$, $\eta$ and $\lambda$ have locally integrable paths and we have $\eta^{(i,j,k)}\geq 0$, if $i\neq j$, and
\begin{equation*}
E\big[\varepsilon^{(i)}\big],\quad \big[\eta^{(i,j,k)}\big]_{\frac{1}{1-\alpha_{k}}},\quad  E\big[\lambda^{(i,k)}\big]
\end{equation*}
are locally integrable for all $i,j\in\{1,\dots,m\}$ and $k\in\{1,\dots,l\}$.
\end{enumerate}

\begin{Remark}\label{re:condition 5}
If~\eqref{co:5} is satisfied, $\alpha_{1} < \cdots < \alpha_{l}$ and $\beta_{1} < \cdots < \beta_{l}$, then~\eqref{co:2} follows for $\rho,\varrho\in C(\R_{+})$ given by
\begin{equation*}
\rho(v):= v^{\alpha_{1}}\mathbbm{1}_{[0,1]}(v) + v^{\alpha_{l}}\mathbbm{1}_{]1,\infty[}(v)\quad\text{and}\quad\varrho(v):=  v^{\beta_{1}}\mathbbm{1}_{[0,1]}(v) + v^{\beta_{l}}\mathbbm{1}_{]1,\infty[}(v).
\end{equation*}
\end{Remark}

Under~\eqref{co:5}, the two measurable locally integrable functions $\gamma_{1}:[t_{0},\infty[\rightarrow ]-\infty,\infty]$ and $\hat{\delta}_{1}:[t_{0},\infty[\rightarrow [0,\infty]$ defined by
\begin{equation}\label{eq:stability exponential coefficient}
\gamma_{1}(s) := \max_{j = 1,\dots,m}\sum_{k=1}^{l}\alpha_{k}\bigg[\sum_{i=1}^{m}\eta_{s}^{(i,j,k)}\bigg]_{\frac{1}{1-\alpha_{k}}} + \beta_{k}\sum_{i=1}^{m}E[\lambda_{s}^{(i,k)}\big]
\end{equation}
and
\begin{equation}\label{eq:stability drift coefficient}
\hat{\delta}_{1}(s) := \sum_{k=1}^{l}(1-\alpha_{k})\bigg(\sum_{j=1}^{m}\bigg[\sum_{i=1}^{m}\eta_{s}^{(i,j,k)}\bigg]_{\frac{1}{1-\alpha_{k}}}\bigg) + (1-\beta_{k})\sum_{i=1}^{m}E\big[\lambda_{s}^{(i,k)}\big]
\end{equation}
solely depend on the regularity of $\B$ and $\tilde{\B}$. By means of these coefficients we get an \emph{explicit $L^{1}$-comparison estimate} relative to the $[0,\infty]$-valued Borel measurable functional $\Theta$ on $[t_{0},\infty[\times\mathcal{P}\times\mathcal{P}$ given by
\begin{equation}\label{eq:integrability functional}
\Theta(s,\mu,\tilde{\mu}) := \sum_{k=1}^{l} \sum_{i=1}^{m}E\big[\lambda_{s}^{(i,k)}\big]\vartheta(\mu,\tilde{\mu})^{\beta_{k}}.
\end{equation}

\begin{Proposition}\label{pr:moment stability estimate}
Let~\eqref{co:1} and~\eqref{co:5} be valid, $X$ be a solution to~\eqref{eq:McKean-Vlasov} and $\tilde{X}$ solve~\eqref{eq:McKean-Vlasov} with $\tilde{\B}$ instead of $\B$ such that
\begin{equation*}
E\big[|Y_{t_{0}}|_{1}\big] < \infty\quad\text{for $Y:=X-\tilde{X}$}
\end{equation*}
and $\Theta(\cdot,\mathcal{L}(X),\mathcal{L}(\tilde{X}))$ is locally integrable. Then
\begin{equation}\label{eq:specific moment stability estimate}
E\big[|Y_{t}|_{1}\big] \leq e^{\int_{t_{0}}^{t}\!\gamma_{1}(s)\,ds}E\big[|Y_{t_{0}}|_{1}\big] + \int_{t_{0}}^{t}\!e^{\int_{s}^{t}\gamma_{1}(\tilde{s})\,d\tilde{s}}\big(E\big[|\varepsilon_{s}|_{1}\big] + \hat{\delta}_{1}(s)\big)\,ds
\end{equation}
for any $t\geq t_{0}$. In particular, if $\gamma_{1}^{+}$, $E[|\varepsilon|_{1}]$ and $\hat{\delta}_{1}$ are integrable, then $E[|Y|_{1}]$ is bounded. If additionally $\int_{t_{0}}^{\infty}\!\gamma_{1}^{-}(s)\,ds = \infty$, then
\begin{equation*}
\lim_{t\uparrow\infty} E\big[|Y_{t}|_{1}\big] = 0.
\end{equation*}
\end{Proposition}

\begin{Remark}
While $\varepsilon$ serves as error estimate for $\B - \tilde{\B}$, the coefficient $\hat{\delta}_{1}$ arises from all the partial H\oe lder exponents in $]0,1[$ that appear in~\eqref{co:5}. Namely, $\hat{\delta}_{1}(s)$ vanishes for given $s\geq t_{0}$ if and only if each $k\in\{1,\dots,l\}$ satisfies
\begin{equation*}
\mathbbm{1}_{]0,1[}(\alpha_{k})\max_{j=1,\dots,m}\sum_{i=1}^{m}\eta_{s}^{(i,j,k)} \leq 0\quad\text{a.s.}\quad\text{and}\quad \mathbbm{1}_{]0,1[}(\beta_{k})\max_{i=1,\dots,m}\lambda_{s}^{(i,k)} = 0\quad\text{a.s.}
\end{equation*}
\end{Remark}

Although the bound in Proposition~\ref{pr:abstract stability estimate} applies to different types of moduli of continuity that are specified in~\eqref{co:2}, the estimate~\eqref{eq:specific moment stability estimate} is generally sharper, as Example~\ref{ex:power functions} and Remark~\ref{re:condition 5} show, bearing in mind that $\gamma_{1}$ may take negative values.

Based on the \emph{partial mixed H\oe lder continuity condition}~\eqref{co:5} for $\B$, assuming that $\B = \tilde{\B}$ and $\varepsilon = 0$ there, we get (asymptotic) stability in moment as direct consequence.

\begin{Corollary}\label{co:moment stability}
Let~\eqref{co:1} and~\eqref{co:5} be satisfied for $\B=\tilde{\B}$ and $\varepsilon = 0$. Then~\eqref{eq:McKean-Vlasov} is (asymptotically) stable in moment relative to $\Theta$ defined by~\eqref{eq:integrability functional} if $\gamma_{1}^{+}$ and $\hat{\delta}_{1}$ are integrable (and $\int_{t_{0}}^{\infty}\gamma_{1}^{-}(s)\,ds = \infty$).
\end{Corollary}

For a \emph{description of the $L^{1}$-boundedness and the rate of $L^{1}$-convergence for solutions} in Corollary~\ref{co:exponential moment stability} below, let us restrict~\eqref{co:5} to a \emph{partial Lipschitz continuity condition}:
\begin{enumerate}[label=(C.\arabic*), ref=C.\arabic*, leftmargin=\widthof{(C.6)} + \labelsep]
\setcounter{enumi}{5}
\item\label{co:6} There are a measurable locally integrable map $\eta:[t_{0},\infty[\rightarrow\R^{m\times m}$ and an $\R_{+}^{m}$-valued progressively measurable process $\lambda$ with locally integrable paths such that
\begin{equation*}
\mathrm{sgn}(x_{i} - \tilde{x}_{i})\big(\B^{(i)}(x,\mu) - \B^{(i)}(\tilde{x},\tilde{\mu})\big) \leq \bigg(\sum_{j=1}^{m}\eta_{i,j}|x_{j} - \tilde{x}_{j}|\bigg) + \lambda^{(i)}\vartheta(\mu,\tilde{\mu})
\end{equation*}
for any $x,\tilde{x}\in\R^{m}$ and $\mu,\tilde{\mu}\in\mathcal{P}$ a.s.~for all $i\in\{1,\dots,m\}$. Moreover, $\eta_{i,j}\geq 0$ for all $i,j\in\{1,\dots,m\}$ with $i\neq j$ and $E[|\lambda|_{1}]$ is locally integrable.
\end{enumerate}

If the preceding condition holds, then the coefficient $\hat{\delta}_{1}$ in~\eqref{eq:stability drift coefficient} vanishes and the \emph{stability coefficient} $\gamma_{1}$ in~\eqref{eq:stability exponential coefficient} becomes
\begin{equation}\label{eq:special stability exponential coefficient 2}
\gamma_{1} = \max_{j = 1,\dots,m} \eta_{1,j} + \cdots + \eta_{m,j} + E\big[|\lambda|_{1}\big].
\end{equation}
Remarkably, $\gamma_{1}$ is merely influenced by the regularity of $\B$. Further, under~\eqref{co:6}, the functional in~\eqref{eq:integrability functional} is of the form
\begin{equation}\label{eq:integrability functional 2}
\Theta(\cdot,\mu,\tilde{\mu}) = E\big[|\lambda|_{1}\big]\vartheta(\mu,\tilde{\mu})
\end{equation}
for any $\mu,\tilde{\mu}\in\mathcal{P}$. To deduce \emph{exponential moment stability}, we impose an upper bound on $\gamma_{1}$ that involves sums of power functions.
\begin{enumerate}[label=(C.\arabic*), ref=C.\arabic*, leftmargin=\widthof{(C.7)} + \labelsep]
\setcounter{enumi}{6}
\item\label{co:7} Condition~\eqref{co:6} is valid and there are $l\in\N$, $\alpha\in ]0,\infty[^{l}$ and $\hat{\lambda},s\in\R^{l}$ such that $\alpha_{1} < \cdots < \alpha_{l}$, $\hat{\lambda}_{l} < 0$ and
\begin{equation*}
\gamma_{1}(s) \leq \hat{\lambda}_{1}\alpha_{1}(s-s_{1})^{\alpha_{1}-1} + \cdots + \hat{\lambda}_{l}\alpha_{l}(s-s_{l})^{\alpha_{l}-1}\quad\text{for a.e.~$s\geq t_{1}$}
\end{equation*}
for some $t_{1}\geq t_{0}$ satisfying $\max_{k=1,\dots,l} s_{k}\leq t_{1}$.
\end{enumerate}

By using the negativity of the constant $\hat{\lambda}_{l}$ associated to the greatest power $\alpha_{l}$ in the preceding condition, the following stability properties hold.

\begin{Corollary}\label{co:exponential moment stability}
Under~\eqref{co:1}, the following two assertions hold:
\begin{enumerate}[(i)]
\item If~\eqref{co:6} is valid and $\gamma_{1}^{+}$ is integrable, then the difference $Y$ of any two solutions $X$ and $\tilde{X}$ to~\eqref{eq:McKean-Vlasov} satisfies
\begin{equation*}
\sup_{t\geq t_{0}} e^{\int_{t_{0}}^{t}\!\gamma_{1}^{-}(s)\,ds}E\big[|Y_{t}|_{1}\big] < \infty,
\end{equation*}
provided $E[|Y_{t_{0}}|_{1}] < \infty$ and $\Theta(\cdot,\mathcal{L}(X),\mathcal{L}(\tilde{X}))$ is locally integrable. If in addition $\gamma_{1}^{-}$ fails to be integrable, then
\begin{equation*}
\lim_{t\uparrow\infty} e^{\alpha\int_{t_{0}}^{t}\!\gamma_{1}^{-}(s)\,ds}E\big[|Y_{t}|_{1}\big] = 0
\quad\text{for any $\alpha\in [0,1[$.}
\end{equation*}

\item Let~\eqref{co:7} be valid. Then~\eqref{eq:McKean-Vlasov} is $\alpha_{l}$-exponentially stable in moment relative to $\Theta$ with any moment $\alpha_{l}$-Lyapunov exponent in $]\hat{\lambda}_{l},0[$. Further, $\hat{\lambda}_{l}$ serves as Lyapunov exponent as soon as
\begin{equation*}
\max_{k=1,\dots,l}\hat{\lambda}_{k} \leq 0 \quad\text{and}\quad s_{l}\leq t_{0}.
\end{equation*}
\end{enumerate}
\end{Corollary}

Let us conclude with a specification of Example~\ref{ex:class of drift maps}, which for $m=1$ plays a major role in the \emph{volatility modelling} in~\cite[Section~3.1]{BriGraKal24}, as the representation~\eqref{eq:power sum representation} suggests.

\begin{Example}\label{ex:class of drift maps 2}
Let $l\in\N$, $\zeta:[t_{0},\infty[\rightarrow\R^{m}$ be measurable and locally integrable, $\eta$ be an $\R_{+}^{m\times l}$-valued progressively measurable process with locally integrable paths,
\begin{equation*}
f:\R\rightarrow\R^{m\times l}
\end{equation*}
be measurable and $G$ be an $\R^{m}$-valued admissible map on $[t_{0},\infty[\times\Omega\times\R^{m}\times\mathcal{P}$ such that
\begin{equation}\label{eq:sum of functions}
\B^{(i)}(x,\mu) = \zeta_{i} x_{i} + \eta^{(i,1)}f_{i,1}(x_{i}) + \cdots + \eta^{(i,l)}f_{i,l}(x_{i}) + G^{(i)}(x,\mu)
\end{equation}
for all $(x,\mu)\in\R^{m}\times\mathcal{P}$ and $i\in\{1,\dots,m\}$. Then~\eqref{co:5} follows for $\B=\tilde{\B}$ and $\varepsilon = 0$ from the following condition:
\begin{enumerate}[(1)]
\item There are $\alpha_{0},\beta_{0}\in ]0,1]$, $\alpha\in ]0,1]^{l}$, $\tilde{\eta}\in\R^{m\times l}$ and $\R_{+}^{m}$-valued progressively measurable processes $\overline{\eta}$ and $\lambda$ with locally integrable paths such that 
\begin{align*}
\mathrm{sgn}(v-\tilde{v})\big(f_{i,k}(v)-f_{i,k}(\tilde{v})\big) &\leq  \tilde{\eta}_{i,k}|v-\tilde{v}|^{\alpha_{k}}\quad\text{and}\\
|G^{(i)}(x,\mu) - G^{(i)}(\tilde{x},\tilde{\mu})| &\leq \overline{\eta}^{(i)}|x-\tilde{x}|_{1}^{\alpha_{0}} + \lambda^{(i)}\vartheta(\mu,\tilde{\mu})^{\beta_{0}}
\end{align*}
for all $v,\tilde{v}\in\R$, $x,\tilde{x}\in\R^{m}$, $\mu,\tilde{\mu}\in\mathcal{P}$, $i\in\{1,\dots,m\}$ and $k\in\{1,\dots,l\}$. Further, $\sum_{k=1}^{l}\sum_{j=1}^{m}[\eta^{(j,k)}\tilde{\eta}_{j,k}]_{\frac{1}{1-\alpha_{k}}}$, $E[|\overline{\eta}|_{1}]$ and $E[|\lambda|_{1}]$ are locally integrable.
\end{enumerate}
For instance, if $f_{i,k}$ is decreasing for all $i\in\{1,\dots,m\}$ and $k\in\{1,\dots,l\}$, then $\tilde{\eta} = 0$ is possible in~(1). In general, if~\eqref{co:1} is met by $\Sigma$ and~(1) holds, then Proposition~\ref{pr:moment stability estimate} and Corollaries~\ref{co:moment stability} and~\ref{co:exponential moment stability} yield the following statements:
\begin{enumerate}[(1)]
\setcounter{enumi}{1}
\item The difference $Y$ of any two solutions $X$ and $\tilde{X}$ to~\eqref{eq:McKean-Vlasov} for which $E[|Y_{t_{0}}|_{1}] < \infty$ and $E[|\lambda|_{1}]\vartheta(\mathcal{L}(X),\mathcal{L}(\tilde{X}))^{\beta_{0}}$ is locally integrable satisfies the estimate~\eqref{eq:specific moment stability estimate} with
\begin{equation*}
\gamma_{1} = \max_{j = 1,\dots,m} \underline{\eta}_{j} + \alpha_{0}\big[|\overline{\eta}|_{1}\big]_{\frac{1}{1-\alpha_{0}}} + \beta_{0}E\big[|\lambda|_{1}\big],
\end{equation*}
$\hat{\delta}_{1} = (\sum_{k=1}^{l}(1-\alpha_{k})\sum_{j=1}^{m}[\eta^{(j,k)}\tilde{\eta}_{j,k}\big]_{\frac{1}{1-\alpha_{k}}}) + (1-\alpha_{0})m[|\overline{\eta}|_{1}]_{\frac{1}{1-\alpha_{0}}} + (1-\beta_{0})E[|\lambda|_{1}]$ and $\varepsilon = 0$, where the measurable map $\underline{\eta}:[t_{0},\infty[\rightarrow ]-\infty,\infty]^{m}$ is given by
\begin{equation}\label{eq:special coefficients}
\underline{\eta}_{j}(s) :=\zeta_{j}(s) + \sum_{k=1}^{l}\alpha_{k}\big[\eta_{s}^{(j,k)}\tilde{\eta}_{j,k}\big]_{\frac{1}{1-\alpha_{k}}}.
\end{equation}

\item For equation~\eqref{eq:McKean-Vlasov} to be (asymptotically) stable in moment with respect to the Borel measurable functional $\Theta:[t_{0},\infty[\times\mathcal{P}\times\mathcal{P}\rightarrow [0,\infty]$ given by
\begin{equation*}
\Theta(s,\mu,\tilde{\mu}):= E\big[|\lambda_{s}|_{1}\big]\vartheta(\mu,\tilde{\mu})^{\beta_{0}},
\end{equation*}
it is sufficient that the coefficients $\underline{\eta}_{1}^{+},\dots,\underline{\eta}_{m}^{+}$, $[|\overline{\eta}|_{1}]_{\frac{1}{1-\alpha_{0}}}$ and $E[|\lambda|_{1}]$ are integrable (and $\int_{t_{0}}^{\infty}\!\min_{j = 1,\dots,m}\underline{\eta}_{j}^{-}(s) \,ds$ $= \infty$).

\item Assume that $\alpha_{0} = \dots = \alpha_{l} = \beta_{0} = 1$. This ensures the validity of~\eqref{co:6}. Hence, if there are $\hat{\lambda} < 0$ and $\hat{\alpha} > 0$ such that
\begin{equation*}
\max_{j = 1,\dots,m}\underline{\eta}_{j}(s) + \big[|\overline{\eta}_{s}|_{1}\big]_{\infty}  +  E\big[|\lambda_{s}|_{1}\big]\leq \hat{\lambda}\hat{\alpha}(s-t_{0})^{\hat{\alpha}-1}\quad\text{for a.e.~$s\geq t_{0}$},
\end{equation*}
then $\hat{\alpha}$-exponential stability in moment relative to $\Theta$ with Lyapunov exponent $\hat{\lambda}$ holds.
\end{enumerate}
\end{Example}

\subsection{Pathwise stability and moment growth estimates}

In this section, our first aim is to derive pathwise exponential stability for~\eqref{eq:McKean-Vlasov}. For this purpose, we replace the Osgood condition~\eqref{co:1} on compact sets by the following stronger \emph{$\frac{1}{2}$-H\oe lder continuity condition} for the diffusion:
\begin{enumerate}[label=(C.\arabic*), ref=C.\arabic*, leftmargin=\widthof{(C.9)} + \labelsep]
\setcounter{enumi}{7}
\item\label{co:8} There is a measurable locally square-integrable map $\hat{\eta}:[t_{0},\infty[\rightarrow\R_{+}^{m}$ such that $|e_{i}'(\Sigma(x)-\Sigma(\tilde{x}))| \leq \hat{\eta}_{i}|x_{i} -\tilde{x}_{i}|^{\frac{1}{2}}$ for any $x,\tilde{x}\in\R^{m}$ a.s.~for all $i\in\{1,\dots,m\}$.
\end{enumerate}

Moreover, we require that the partial Lipschitz condition~\eqref{co:6} for the drift $\B$ holds, and the regularity coefficients $\lambda$ and $\hat{\eta}$ satisfy a suitable growth condition, which follows if $E[|\lambda|_{1}]$ and $\hat{\eta}$ are locally bounded, for instance.
\begin{enumerate}[label=(C.\arabic*), ref=C.\arabic*, leftmargin=\widthof{(C.10)} + \labelsep]
\setcounter{enumi}{8}
\item\label{co:9} Conditions~\eqref{co:6} and~\eqref{co:8} are satisfied and there is some $\hat{\delta} > 0$ such that $\sup_{t\geq t_{0}}\int_{t}^{t+\hat{\delta}}\!f(s)\,ds < \infty$ for $f\in\{E[|\lambda|_{1}],|\hat{\eta}|_{1}^{2}\}$.
\end{enumerate}

Under the following abstract condition on the stability coefficient $\gamma_{1}$, we obtain a \emph{general pathwise stability estimate} from Theorem~\ref{th:pathwise stability}, a pathwise result for random It{\^o} processes based on the Borel-Cantelli Lemma.
\begin{enumerate}[label=(C.\arabic*), ref=C.\arabic*, leftmargin=\widthof{(C.11)} + \labelsep]
\setcounter{enumi}{9}
\item\label{co:10} Condition~\eqref{co:9} holds and there are $\hat{\varepsilon}\in ]0,1[$ and a strictly increasing sequence $(t_{n})_{n\in\N}$ in $[t_{0},\infty[$ such that $\gamma_{1} \leq 0$ a.e.~on $[t_{1},\infty[$,
\begin{equation*}
\sup_{n\in\N} (t_{n+1} - t_{n}) < \hat{\delta},\quad\lim_{n\uparrow\infty} t_{n} = \infty
\end{equation*}
and $\sum_{n=1}^{\infty}\exp(\frac{\varepsilon}{2}\int_{t_{1}}^{t_{n}}\!\gamma_{1}(s)\,ds) < \infty$ for all $\varepsilon\in ]0,\hat{\varepsilon}[$.
\end{enumerate}

\begin{Proposition}\label{pr:specific pathwise asymptotic analysis}
Assume~\eqref{co:10} and let $X$ and $\tilde{X}$ be two solutions to~\eqref{eq:McKean-Vlasov} such that $E[|\lambda|_{1}]\vartheta(\mathcal{L}(X),\mathcal{L}(\tilde{X}))$ is locally integrable. Then $Y:=X-\tilde{X}$ satisfies
\begin{equation*}
\limsup_{t\uparrow\infty} \frac{1}{\varphi(t)}\log(|Y_{t}|_{1}) \leq \frac{1}{2}\limsup_{n\uparrow\infty} \frac{1}{\varphi(t_{n})}\int_{t_{1}}^{t_{n}}\!\gamma_{1}(s)\,ds\quad\text{a.s.}
\end{equation*}
for any increasing continuous function $\varphi:[t_{1},\infty[\rightarrow\R_{+}$ that is positive on $]t_{1},\infty[$, provided $E[|Y_{t_{0}}|_{1}]$ $< \infty$ or $\lambda = 0$.
\end{Proposition}

Now we employ the upper bound~\eqref{co:7} on $\gamma_{1}$ to derive pathwise exponential stability, since~\eqref{co:10} already follows from~\eqref{co:9} in this case, as will be shown.

\begin{Corollary}\label{co:special pathwise stability}
Let~\eqref{co:7} and~\eqref{co:9} hold and define $\Theta:[t_{0},\infty[\times\mathcal{P}\times\mathcal{P}\rightarrow [0,\infty]$ by~\eqref{eq:integrability functional 2}.
\begin{enumerate}[(i)]
\item Then the McKean-Vlasov SDE~\eqref{eq:McKean-Vlasov} is pathwise $\alpha_{l}$-exponentially stable with Lyapunov exponent $\frac{\hat{\lambda}_{l}}{2}$ with respect to an initial absolute moment and $\Theta$.

\item If in fact $\B$ is independent of $\mu\in\mathcal{P}$, then the SDE~\eqref{eq:McKean-Vlasov} is pathwise $\alpha_{l}$-exponentially stable with Lyapunov exponent $\frac{\hat{\lambda}_{l}}{2}$.
\end{enumerate}
\end{Corollary}

The corollary directly applies to the class of drift coefficients in Example~\ref{ex:class of drift maps 2}.

\begin{Example}\label{ex:class of drift maps 3}
Suppose that $\B$ is of the form~\eqref{eq:sum of functions} and let~\eqref{co:8} hold for $\Sigma$ when $\hat{\eta}$ is locally bounded. Then the following sharpened version of condition~(1) in Example~\ref{ex:class of drift maps 2} implies~\eqref{co:9}:
\begin{enumerate}[(1)]
\setcounter{enumi}{1}
\item There are $\tilde{\eta}\in\R^{m\times l}$, a measurable locally bounded map $\overline{\eta}:[t_{0},\infty[\rightarrow\R_{+}^{m}$ and an $\R_{+}^{m}$-valued progressively measurable process $\lambda$ with locally integrable paths so that 
\begin{align*}
\mathrm{sgn}(v-\tilde{v})\big(f_{i,k}(v)-f_{i,k}(\tilde{v})\big) &\leq  \tilde{\eta}_{i,k}|v-\tilde{v}|\quad\text{and}\\
|G^{(i)}(x,\mu) - G^{(i)}(\tilde{x},\tilde{\mu})| &\leq \overline{\eta}_{i}|x-\tilde{x}|_{1} + \lambda^{(i)}\vartheta(\mu,\tilde{\mu})
\end{align*}
for any $v,\tilde{v}\in\R$, $x,\tilde{x}\in\R^{m}$, $\mu,\tilde{\mu}\in\mathcal{P}$, $i\in\{1,\dots,m\}$ and $k\in\{1,\dots,l\}$. In addition, $\sum_{k=1}^{l}\sum_{j=1}^{m}[\eta^{(j,k)}\tilde{\eta}_{j,k}]_{\infty}$ and $E[|\lambda|_{1}]$ are locally bounded.
\end{enumerate}
Under this condition, the formula~\eqref{eq:special coefficients} reduces to $\underline{\eta}_{j} = \zeta_{j} + \sum_{k=1}^{l}[\eta^{(j,k)}\tilde{\eta}_{j,k}]_{\infty}$ for all $j\in\{1,\dots,m\}$, and we suppose in addition that
\begin{equation*}
\max_{j = 1,\dots,m}\underline{\eta}_{j}(s) + |\overline{\eta}(s)|_{1}  +  E\big[|\lambda_{s}|_{1}\big]\leq \hat{\lambda}\hat{\alpha}(s-t_{0})^{\hat{\alpha}-1}\quad\text{for a.e.~$s\geq t_{0}$}
\end{equation*}
for some $\hat{\lambda} < 0$ and $\hat{\alpha} > 0$. Then Corollary~\ref{co:special pathwise stability} entails the subsequent two assertions:
\begin{enumerate}[(1)]
\setcounter{enumi}{2}
\item Equation~\eqref{eq:McKean-Vlasov} is pathwise $\hat{\alpha}$-exponentially stable with Lyapunov exponent $\frac{\hat{\lambda}}{2}$ relative to an initial absolute moment and $\Theta:[t_{0},\infty[\times\mathcal{P}\times\mathcal{P}\rightarrow [0,\infty]$ given by~\eqref{eq:integrability functional 2}.

\item If $\lambda = 0$, in which case $\B$ is independent of $\mu\in\mathcal{P}$, then the SDE~\eqref{eq:McKean-Vlasov} is pathwise $\hat{\alpha}$-exponentially stable with Lyapunov exponent $\frac{\hat{\lambda}}{2}$.
\end{enumerate}
\end{Example}

Next, we give two first moment bounds for solutions to~\eqref{eq:McKean-Vlasov}, each showing that their absolute moment functions are locally bounded. Hence, local integrability relative to the functional $\Theta$ in each of the Corollaries~\ref{co:pathwise uniqueness},~\ref{co:moment stability},~\ref{co:exponential moment stability} and~\ref{co:special pathwise stability} holds in these cases. First, we require an \emph{Osgood growth condition on compact sets} for $\Sigma$.
\begin{enumerate}[label=(C.\arabic*), ref=C.\arabic*, leftmargin=\widthof{(C.11)} + \labelsep]
\setcounter{enumi}{10}
\item\label{co:11} For any $n\in\N$ there are an increasing $\hat{\phi}_{n}\in C(\R_{+})$ and an $\R_{+}$-valued progressively measurable process $\hat{\upsilon}^{(n)}$ with locally square-integrable paths such that
\begin{equation*}
\text{$\hat{\phi}_{n} > 0$ on $]0,\infty[$,}\quad \int_{0}^{1}\!\frac{1}{\hat{\phi}_{n}(v)^{2}}\,dv = \infty \quad\text{and}\quad |e_{i}'\Sigma(x)| \leq \hat{\upsilon}^{(n)}\hat{\phi}(|x_{i}|)
\end{equation*}
for any $x\in\R^{m}$ with $|x|\leq n$ a.s.~for every $i\in\{1,\dots,m\}$.
\end{enumerate}

\begin{Remark}
If the Osgood condition~\eqref{co:1} on compact sets holds and $\Sigma(0)=0$ a.s., then~\eqref{co:11} follows. In particular, this implication applies to Example~\ref{ex:sum of power functions} when $\varphi(0) = 0$ and $\zeta = 0$.
\end{Remark}

Let us also consider two \emph{partial growth conditions} on $\B$. While the first allows for various kinds of growth behaviour, the second is of affine type and explicitly measures the growth components:
\begin{enumerate}[label=(C.\arabic*), ref=C.\arabic*, leftmargin=\widthof{(C.8)} + \labelsep]
\setcounter{enumi}{11}
\item\label{co:12} There are $\phi,\varphi\in C(\R_{+})$ that are positive on $]0,\infty[$ and vanish at $0$ and $\R_{+}^{m}$-valued progressively measurable processes $\kappa$, $\upsilon$, $\chi$ with locally integrable paths so that
\begin{equation*}
\mathrm{sgn}(x_{i})\B^{(i)}(x,\mu) \leq \kappa^{(i)} + \upsilon^{(i)}\phi(|x|_{1}) + \chi^{(i)}\varphi(\vartheta(\mu,\delta_{0}))
\end{equation*}
for any $(x,\mu)\in\R^{m}\times\mathcal{P}$ a.s.~for each $i\in\{1,\dots,m\}$. Moreover, $\phi^{\frac{1}{\alpha}}$ is concave for some $\alpha\in ]0,1]$, $\varphi$ is increasing and $E[|\kappa|_{1}]$, $E[|\upsilon|_{1}]_{\frac{1}{1-\alpha}}$, $E[|\chi|_{1}]$ are locally integrable.

\item\label{co:13} There are $l\in\N$, $\alpha,\beta\in ]0,1]^{l}$ and progressively measurable processes, $\kappa$, $\upsilon$ and $\chi$ with values in $\R_{+}^{m}$, $\R^{m\times m\times l}$ and $\R_{+}^{m\times l}$, respectively, such that
\begin{equation*}
\begin{split}
\mathrm{sgn}(x_{i})\B^{(i)}(x,\mu) \leq \kappa^{(i)} +  \sum_{k=1}^{l}\bigg(\sum_{j=1}^{m}\upsilon^{(i,j,k)}|x_{j}|^{\alpha_{k}}\bigg) + \chi^{(i,k)}\vartheta(\mu,\delta_{0})^{\beta_{k}}
\end{split}
\end{equation*}
for all $(x,\mu)\in\R^{m}\times\mathcal{P}$ a.s.~for any $i\in\{1,\dots,m\}$. Further, $\kappa$, $\upsilon$ and $\chi$ have locally integrable paths and it holds that $\upsilon^{(i,j,k)}\geq 0$, if $i\neq j$, and
\begin{equation*}
E\big[\kappa^{(i)}\big],\quad \big[\upsilon^{(i,j,k)}\big]_{\frac{1}{1-\alpha_{k}}},\quad E\big[\chi^{(i,k)}\big]
\end{equation*}
are locally integrable for any $i,j\in\{1,\dots,m\}$ and $k\in\{1,\dots,l\}$.
\end{enumerate}

If~\eqref{co:12} holds, we introduce for given $\beta\in ]0,1]$ two measurable locally integrable functions by
\begin{equation*}
g := \alpha\big[|\upsilon|_{1}\big]_{\frac{1}{1-\alpha}} + \beta E\big[|\chi|_{1}\big] \quad\text{and}\quad h := (1-\alpha)\big[|\upsilon|_{1}\big]_{\frac{1}{1-\alpha}} + (1-\beta)E\big[|\chi|_{1}\big]
\end{equation*}
and infer a \emph{quantitative first moment estimate} from Theorem~\ref{th:abstract moment estimate}, which reduces to an explicit bound in the framework of Example~\ref{ex:power functions}.

\begin{Lemma}\label{le:abstract growth estimate}
Let~\eqref{co:11} and~\eqref{co:12} be valid and $X$ be a solution to~\eqref{eq:McKean-Vlasov} for which $E[|X_{t_{0}}|_{1}] < \infty$ and $E[|\chi|_{1}]\vartheta(\mathcal{L}(X),\delta_{0})$ is locally integrable. Define $\varphi_{0}\in C(\R_{+})$ by
\begin{equation*}
\varphi_{0}(v) :=\phi(v)^{\frac{1}{\alpha}}\vee\varphi(v)^{\frac{1}{\beta}}
\end{equation*}
and suppose that $\Phi_{\varphi_{0}}(\infty) = \infty$. Then $E[|X|_{1}]$ is locally bounded and 
\begin{equation*}
\sup_{s\in [t_{0},t]} E\big[|X_{s}|_{1}\big]\leq \Psi_{\varphi_{0}}\bigg(E\big[|X_{t_{0}}|_{1}\big] + \int_{t_{0}}^{t}\!E\big[|\kappa_{s}|_{1}\big] + h(s)\,ds,\int_{t_{0}}^{t}\!g(s)\,ds\bigg)
\end{equation*}
for all $t\geq t_{0}$. In particular, if $E[|\kappa|_{1}]$, $g$ and $h$ are integrable, then $E[|X|_{1}]$ is bounded.
\end{Lemma}

Under~\eqref{co:13}, two measurable locally integrable functions $g_{1}:[t_{0},\infty[\rightarrow ]-\infty,\infty]$ and $h_{1}:[t_{0},\infty[\rightarrow [0,\infty]$ can be defined by
\begin{equation}\label{eq:special moment exponential coefficient}
g_{1}(s) := \max_{j=1,\dots,m}\sum_{k=1}^{l}\alpha_{k}\bigg[\sum_{i=1}^{m}\upsilon_{s}^{(i,j,k)}\bigg]_{\frac{1}{1-\alpha_{k}}} + \beta_{k}\sum_{i=1}^{m}E\big[\chi_{s}^{(i,k)}\big]
\end{equation}
and
\begin{equation}\label{eq:special moment drift coefficient}
h_{1}(s) := \sum_{k=1}^{l}(1-\alpha_{k})\bigg(\sum_{j=1}^{m}\bigg[\sum_{i=1}^{m}\upsilon_{s}^{(i,j,k)}\bigg]_{\frac{1}{1-\alpha_{k}}}\bigg) + (1-\beta_{k})\sum_{i=1}^{m}E\big[\chi_{s}^{(i,k)}\big].
\end{equation}
These formulas are in spirit the same as those for the stability coefficients in~\eqref{eq:stability exponential coefficient} and~\eqref{eq:stability drift coefficient}, since the subsequent \emph{explicit $L^{1}$-growth bound} follows, just as the $L^{1}$-comparison estimate in Proposition~\ref{pr:moment stability estimate}, from Theorem~\ref{th:moment stability estimate}.

\begin{Lemma}\label{le:moment growth estimate}
Let~\eqref{co:11} and~\eqref{co:13} be satisfied and $X$ be a solution to~\eqref{eq:McKean-Vlasov} such that $E[|X_{t_{0}}|_{1}] < \infty$ and $\sum_{k=1}^{l}\sum_{i=1}^{m}E[\chi^{(i,k)}]\vartheta(\mathcal{L}(X),\delta_{0})^{\beta_{k}}$ is locally integrable. Then
\begin{equation}\label{eq:moment growth estimate}
E\big[|X_{t}|_{1}\big] \leq e^{\int_{t_{0}}^{t}\!g_{1}(s)\,ds}E\big[|X_{t_{0}}|_{1}\big] + \int_{t_{0}}^{t}\!e^{\int_{s}^{t}\!g_{1}(\tilde{s})\,d\tilde{s}}\big(E\big[|\kappa_{s}|_{1}\big] + h_{1}(s)\big)\,ds
\end{equation}
for each $t\geq t_{0}$. In particular, suppose that $g_{1}^{+}$, $E[|\kappa|_{1}]$ and $h_{1}$ are integrable. Then $E[|X|_{1}]$ is bounded, and  $\lim_{t\uparrow\infty} E[|X_{t}|_{1}] = 0$ as soon as $\int_{t_{0}}^{\infty}\!g_{1}^{-}(s)\,ds = \infty$.
\end{Lemma}

\begin{Example}
Let $\B$ satisfy the representation~\eqref{eq:sum of functions} of Example~\ref{ex:class of drift maps 2}. Then~\eqref{co:13} is implied by the following condition:
\begin{enumerate}[(1)]
\item There are $\alpha_{0},\beta_{0}\in ]0,1]$, $\alpha\in ]0,1]^{l}$, $\tilde{\upsilon}\in\R^{m\times l}$ and $\R_{+}^{m}$-valued progressively measurable processes $\kappa$, $\overline{\upsilon}$ and $\chi$ with locally integrable paths such that 
\begin{equation*}
\mathrm{sgn}(v)f_{i,k}(v) \leq  \tilde{\upsilon}_{i,k}|v|^{\alpha_{k}}\quad\text{and}\quad |G^{(i)}(x,\mu)| \leq \kappa^{(i)} + \overline{\upsilon}^{(i)}|x|_{1}^{\alpha_{0}} + \chi^{(i)}\vartheta(\mu,\delta_{0})^{\beta_{0}}
\end{equation*}
for all $v\in\R$, $(x,\mu)\in\R^{m}\times\mathcal{P}$, $i\in\{1,\dots,m\}$ and $k\in\{1,\dots,l\}$. Moreover, $E[|\kappa|_{1}]$, $\sum_{k=1}^{l}\sum_{j=1}^{m}[\eta^{(j,k)}\tilde{\upsilon}_{j,k}]_{\frac{1}{1-\alpha_{k}}}$, $E[|\overline{\upsilon}|_{1}]$ and $E[|\chi|_{1}]$ are locally integrable.
\end{enumerate}
Certainly, we may choose $\tilde{\upsilon} = 0$ in~(1) whenever $f_{i,k}\geq 0$ on $]-\infty,0[$ and $f_{i,k}\leq 0$ on $]0,\infty[$ for all $i\in\{1,\dots,m\}$ and $k\in\{1,\dots,l\}$. Moreover, if~\eqref{co:11} holds for $\Sigma$ and~(1) is satisfied, then two facts follow from Lemma~\ref{le:moment growth estimate}:
\begin{enumerate}[(1)]
\setcounter{enumi}{1}
\item For any solution $X$ to~\eqref{eq:McKean-Vlasov} for which $E[|X_{t_{0}}|_{1}] < \infty$ and $E[|\lambda|_{1}]\vartheta(\mathcal{L}(X),\delta_{0})^{\beta_{0}}$ is locally integrable the bound~\eqref{eq:moment growth estimate} holds for
\begin{equation*}
g_{1} = \max_{j = 1,\dots,m} \underline{\upsilon}_{j} + \alpha_{0}\big[|\overline{\upsilon}|_{1}\big]_{\frac{1}{1-\alpha_{0}}} + \beta_{0}E\big[|\chi|_{1}\big]
\end{equation*}
and $h_{1} = (\sum_{k=1}^{l}(1-\alpha_{k})\sum_{j=1}^{m}[\eta^{(j,k)}\tilde{\upsilon}_{j,k}]_{\frac{1}{1-\alpha_{k}}}) + (1-\alpha_{0})m[|\overline{\upsilon}|_{1}]_{\frac{1}{1-\alpha_{0}}} + (1-\beta_{0})E[|\chi|_{1}]$, where the measurable map $\underline{\upsilon}:[t_{0},\infty[\rightarrow ]-\infty,\infty]^{m}$ is defined by
\begin{equation*}
\underline{\upsilon}_{j}(s) := \zeta_{j}(s) + \sum_{k=1}^{l}\alpha_{k}\big[\eta_{s}^{(j,k)}\tilde{\upsilon}_{j,k}\big]_{\frac{1}{1-\alpha_{k}}}.
\end{equation*}

\item In particular, $E[|X|_{1}]$ is bounded whenever $\underline{\upsilon}_{1}^{+},\dots,\underline{\upsilon}_{m}^{+}$, $[|\overline{\upsilon}|_{1}]_{\frac{1}{1-\alpha_{0}}}$ and $E[|\chi|_{1}]$ are integrable. If in addition
\begin{equation*}
\int_{t_{0}}^{\infty}\min_{j = 1,\dots,m}\underline{\upsilon}_{j}^{-}(s)\,ds = \infty,\quad\text{then}\quad \lim_{t\uparrow\infty} E\big[|X_{t}|_{1}\big] = 0.
\end{equation*}
\end{enumerate}
\end{Example}

\subsection{Strong solutions with locally bounded absolute moment functions}

In what follows, let $b:[t_{0},\infty[\times\R^{m}\times\mathcal{P}\rightarrow\R^{m}$ and $\sigma:[t_{0},\infty[\times\R^{m}\rightarrow\R^{m\times d}$ be Borel measurable and $\xi:\Omega\rightarrow\R^{m}$ be $\mathcal{F}_{t_{0}}$-measurable. We derive a strong solution $X$ to~\eqref{eq:deterministic McKean-Vlasov} such that $X_{t_{0}} = \xi$ a.s.~and the measurable absolute moment function
\begin{equation*}
[t_{0},\infty[\rightarrow [0,\infty],\quad t\mapsto E[|X_{t}|]
\end{equation*}
is finite and locally bounded but not necessarily continuous, by combining the preceding results with a fixed-point approach. Thereby, neither $b(s,\cdot,\mu)$ nor $\sigma(s,\cdot)$ need to be locally Lipschitz continuous for any $(s,\mu)\in [t_{0},\infty[\times\mathcal{P}$.

For a Borel measurable map $\mu:[t_{0},\infty[\rightarrow\mathcal{P}$ we define an $\R^{m}$-valued measurable map $b_{\mu}$ on $[t_{0},\infty[\times\R^{m}$ by $b_{\mu}(s,x):= b(s,x,\mu(s))$ and show that the induced SDE
\begin{equation}\label{eq:mu-SDE}
dX_{t} = b_{\mu}(t,X_{t})\,dt + \sigma(t,X_{t})\,dW_{t}\quad\text{for $t\geq t_{0}$}
\end{equation}
admits a solution. To this end, $\sigma$ should vanish at the origin of $\R^{m}$ at any time and satisfy an Osgood continuity condition on compact sets:
\begin{enumerate}[label=(D.\arabic*), ref=D.\arabic*, leftmargin=\widthof{(D.1)} + \labelsep]
\item\label{det.co:1} $\sigma(\cdot,0) = 0$ and for any $n\in\N$ there are an increasing $\hat{\rho}_{n}\in C(\R_{+})$ and a measurable locally bounded function $\hat{\eta}_{n}:[t_{0},\infty[\rightarrow\R_{+}$ such that
\begin{equation*}
\text{$\hat{\rho}_{n} > 0$ on $]0,\infty[$,}\quad \int_{0}^{1}\!\frac{1}{\hat{\rho}_{n}(v)^{2}}\,dv = \infty\quad\text{and}\quad |e_{i}'(\sigma(\cdot,x) - \sigma(\cdot,\tilde{x}))| \leq \hat{\eta}_{n}\hat{\rho}_{n}(|x_{i} - \tilde{x}_{i}|)
\end{equation*}
for all $x,\tilde{x}\in\R^{m}$ with $|x|\vee|\tilde{x}|\leq n$ for all $i\in\{1,\dots,m\}$.
\end{enumerate}

\begin{Remark}
This condition still allows for Example~\ref{ex:sum of power functions} when $\Sigma=\sigma$, $\varphi(0) = 0$, $\zeta = 0$ and $\eta = \hat{\eta}$ for some measurable locally bounded map $\hat{\eta}:[t_{0},\infty[\rightarrow\R^{d\times l}$.
\end{Remark}

We introduce a partial growth condition and a continuity and boundedness condition as well as a partial Osgood condition on compact sets on $b$:
\begin{enumerate}[label=(D.\arabic*), ref=D.\arabic*, leftmargin=\widthof{(D.3)} + \labelsep]
\setcounter{enumi}{1}
\item\label{det.co:2} There are $\phi,\varphi\in C(\R_{+})$ that are positive on $]0,\infty[$ and vanish at $0$ and measurable locally integrable maps $\kappa,\upsilon,\chi:[t_{0},\infty[\rightarrow\R_{+}^{m}$ such that
\begin{equation*}
\mathrm{sgn}(x_{i})b_{i}(\cdot,x,\mu) \leq \kappa_{i} + \upsilon_{i}\phi(|x|_{1}) + \chi_{i}\varphi(\vartheta(\mu,\delta_{0}))
\end{equation*}
for all $(x,\mu)\in\R^{m}\times\mathcal{P}$ and $i\in\{1,\dots,m\}$. Further, $\phi$ is concave, $\varphi$ is increasing and $\int_{1}^{\infty}\!\frac{1}{\phi(v)}\,dv = \infty$.

\item\label{det.co:3} $b(s,\cdot,\mu)$ is continuous for all $(s,\mu)\in [t_{0},\infty[\times\mathcal{P}$ and for any $n\in\mathbb{N}$ there is $c_{n} \geq 0$ such that
\begin{equation*}
|b(s,x,\mu)|\leq c_{n}
\end{equation*}
for each $(s,x,\mu)\in [t_{0},t_{0} + n]\times\R^{m}\times\mathcal{P}$ with $|x|\leq n$ and $\vartheta(\mu,\delta_{0})\leq n$.

\item\label{det.co:4} For each $n\in\N$ there are a concave $\rho_{n}\in C(\R_{+})$ that is positive on $]0,\infty[$ and a measurable locally integrable function $\eta_{n}:[t_{0},\infty[\rightarrow\R_{+}$ satisfying
\begin{equation*}
\int_{0}^{1}\!\frac{1}{\rho_{n}(v)}\,dv = \infty\quad\text{and}\quad\mathrm{sgn}(x_{i}-\tilde{x}_{i})\big(b_{i}(\cdot,x,\mu) - b_{i}(\cdot,\tilde{x},\mu)\big) \leq \eta_{n}\rho_{n}(|x-\tilde{x}|_{1})
\end{equation*}
for all $x,\tilde{x}\in\R^{m}$ with $|x|\vee|\tilde{x}|\leq n$, $\mu\in\mathcal{P}$ and $i\in\{1,\dots,m\}$.
\end{enumerate}

Let $B_{b,loc}(\mathcal{P})$ be the set of all Borel measurable maps $\mu:[t_{0},\infty[\rightarrow\mathcal{P}$ for which $\vartheta(\mu,\delta_{0})$ is locally bounded. By means of a local weak existence result from~\cite{IkeWat81} and Corollary~\ref{co:pathwise uniqueness}, we settle questions of uniqueness and existence of solutions to~\eqref{eq:mu-SDE}.

\begin{Proposition}\label{pr:mu-SDE}
For $\mu\in B_{b,loc}(\mathcal{P})$ the following three assertions hold:
\begin{enumerate}[(i)]
\item Let~\eqref{det.co:1} and~\eqref{det.co:4} be valid. Then pathwise uniqueness for the SDE~\eqref{eq:mu-SDE} follows.

\item Let~\eqref{det.co:1}-\eqref{det.co:3} hold. Then~\eqref{eq:mu-SDE} admits a weak solution $X$ with $\mathcal{L}(X_{t_{0}}) = \mathcal{L}(\xi)$. Further, if $\xi$ is integrable, then $\vartheta_{1}(\mathcal{L}(X),\delta_{0})$ is locally bounded.

\item If~\eqref{det.co:1}-\eqref{det.co:4} are satisfied, then there is a unique strong solution $X^{\xi,\mu}$ to~\eqref{eq:mu-SDE} such that $X_{t_{0}}^{\xi\,\mu} = \xi$ a.s.
\end{enumerate}
\end{Proposition}

In the sequel, let $B_{b,loc}(\mathcal{P}_{1}(\R^{m}))$ denote the convex space of all Borel measurable maps $\mu:[t_{0},\infty[\rightarrow\mathcal{P}_{1}(\R^{m})$ for which $\int_{\R^{m}}|x|\,\mu(dx)$ is locally bounded, equipped with the topology of local uniform convergence. 

That means, a sequence $(\mu_{n})_{n\in\N}$ in this space converges locally uniformly to some $\mu\in B_{b,loc}(\mathcal{P}_{1}(\R^{m}))$ if $\lim_{n\uparrow\infty} \sup_{s\in [t_{0},t]} \vartheta_{1}(\mu_{n},\mu)(s) = 0$ for every $t\geq t_{0}$. Then, as the Wasserstein metric $\vartheta_{1}$ is complete, $B_{b,loc}(\mathcal{P}_{1}(\R^{m}))$ is completely metrisable.

To construct a solution to~\eqref{eq:deterministic McKean-Vlasov} from a local uniform limit of a Picard iteration in $B_{b,loc}(\mathcal{P}_{1}(\R^{m}))$, we strengthen the partial Osgood condition~\eqref{det.co:4} on compact sets by a partial Lipschitz condition on $b$:
\begin{enumerate}[label=(D.\arabic*), ref=D.\arabic*, leftmargin=\widthof{(D.5)} + \labelsep]
\setcounter{enumi}{4}
\item\label{det.co:5} There are measurable locally integrable maps $\eta$ and $\lambda$ on $[t_{0},\infty[$ with values in $\R^{m\times m}$ and $\R_{+}^{m}$, respectively, such that
\begin{equation*}
\mathrm{sgn}(x_{i} - \tilde{x}_{i})\big(b_{i}(\cdot,x,\mu) - b_{i}(\cdot,\tilde{x},\tilde{\mu})\big) \leq \bigg(\sum_{j=1}^{m}\eta_{i,j}|x_{j} - \tilde{x}_{j}|\bigg) + \lambda_{i}\vartheta(\mu,\tilde{\mu})
\end{equation*}
for any $x,\tilde{x}\in\R^{m}$, $\mu,\tilde{\mu}\in\mathcal{P}$ and $i\in\{1,\dots,m\}$. Moreover, we have $\eta_{i,j}\geq 0$ for all $i,j\in\{1,\dots,m\}$ with $i\neq j$.
\end{enumerate}

If~\eqref{det.co:5} is satisfied, which implies~\eqref{co:6} for $\B=b$, then we use the formula in~\eqref{eq:special stability exponential coefficient 2} when $\lambda = 0$ for the definition of $\gamma_{1,0}:[t_{0},\infty[\rightarrow\R$. Namely, we set 
\begin{equation*}
\gamma_{1,0} := \max_{j = 1,\dots,m} \eta_{1,j} + \cdots + \eta_{m,j}.
\end{equation*}

Further, if the following partial affine growth condition for $b$ holds, which is stronger than~\eqref{det.co:2}, then we can deduce an estimate for the Picard iteration.
\begin{enumerate}[label=(D.\arabic*), ref=D.\arabic*, leftmargin=\widthof{(D.6)} + \labelsep]
\setcounter{enumi}{5}
\item\label{det.co:6} There are $l\in\N$, $\alpha,\beta\in ]0,1]^{l}$ and measurable locally integrable maps $\kappa$, $\upsilon$ and $\chi$ on $[t_{0},\infty[$ with values in $\R_{+}^{m}$, $\R^{m\times m\times l}$ and $\R_{+}^{m\times l}$, respectively, such that
\begin{equation*}
\mathrm{sgn}(x_{i})b_{i}(\cdot,x,\mu) \leq \kappa_{i} + \sum_{k=1}^{l}\bigg(\sum_{j=1}^{m}\upsilon_{i,j,k}|x_{j}|^{\alpha_{k}}\bigg) + \chi_{i,k}\vartheta(\mu,\delta_{0})^{\beta_{k}}
\end{equation*}
for any $(x,\mu)\in\R^{m}\times\mathcal{P}$ and $i\in\{1,\dots,m\}$. In addition, it holds that $\upsilon_{i,j,k}\geq 0$ for all $i,j\in\{1,\dots,m\}$ with $i\neq j$ and $k\in\{1,\dots,l\}$.
\end{enumerate}

Since~\eqref{det.co:6} entails~\eqref{co:13} for $\B = b$, we may use the formulas~\eqref{eq:special moment exponential coefficient} and~\eqref{eq:special moment drift coefficient} for the functions $g_{1}$ and $h_{1}$ when all appearing coefficients are deterministic. Namely,
\begin{equation*}
g_{1} = \max_{j = 1,\dots,m} \sum_{k=1}^{l}\alpha_{k}\bigg(\bigg(\sum_{i=1}^{m}\upsilon_{i,j,k}\bigg)^{+} - \bigg(\sum_{i=1}^{m}\upsilon_{i,j,k}\bigg)^{-}\mathbbm{1}_{\{1\}}(\alpha_{k})\bigg) + \beta_{k}\sum_{i=1}^{m}\chi_{i,k}
\end{equation*}
and
\begin{equation*}
h_{1} = \sum_{k=1}^{l}(1-\alpha_{k})\bigg(\sum_{j=1}^{m}\bigg(\sum_{i=1}^{m}\upsilon_{i,j,k}\bigg)^{+}\bigg) + (1-\beta_{k})\sum_{i=1}^{m}\chi_{i,k}.
\end{equation*}
Based on these considerations, we obtain a \emph{strong existence result with explicit error and growth estimates}.

\begin{Theorem}\label{th:strong existence}
Let~\eqref{det.co:1}-\eqref{det.co:3} and~\eqref{det.co:5} hold, $\mathcal{P}_{1}(\R^{m})\subseteq\mathcal{P}$, $\mu_{0}\in B_{b,loc}(\mathcal{P})$ and $E[|\xi|_{1}] < \infty$. Further, define $\Theta:[t_{0},\infty[\times\mathcal{P}\times\mathcal{P}\rightarrow\R_{+}$ by $\Theta(\cdot,\mu,\tilde{\mu}):= |\lambda|_{1}\vartheta(\mu,\tilde{\mu})$.
\begin{enumerate}[(i)]
\item Then pathwise uniqueness for~\eqref{eq:deterministic McKean-Vlasov} relative to $\Theta$ holds and there exists a unique strong solution $X^{\xi}$ to~\eqref{eq:deterministic McKean-Vlasov} such that $X_{t_{0}}^{\xi} = \xi$ a.s.~and $E[|X^{\xi}|]$ is locally bounded.

\item The map $[t_{0},\infty[\rightarrow\mathcal{P}_{1}(\R^{m})$, $t\mapsto\mathcal{L}(X_{t}^{\xi})$ is the local uniform limit of the sequence $(\mu_{n})_{n\in\mathbb{N}}$ in $B_{b,loc}(\mathcal{P}_{1}(\R^{m}))$ recursively given by $\mu_{n} := \mathcal{L}(X^{\xi,\mu_{n-1}})$ and
\begin{equation}\label{eq:error estimate}
\sup_{s\in [t_{0},t]}\vartheta_{1}(\mu_{n}(s),\mathcal{L}(X_{s}^{\xi})) \leq \Delta(t) \sum_{i=n}^{\infty} \frac{1}{i!}\bigg(\int_{t_{0}}^{t}\!e^{\int_{s}^{t}\!\gamma_{1,0}^{+}(\tilde{s})\,d\tilde{s}}|\lambda(s)|_{1}\,ds\bigg)^{i}
\end{equation}
for all $t\geq t_{0}$ with $\Delta(t):=\sup_{s\in [t_{0},t]}\vartheta(\mathcal{L}(X^{\xi,\mu_{0}}),\mu_{0})(s)$.

\item If in fact~\eqref{det.co:6} holds, then $(\mu_{n})_{n\in\mathbb{N}}$ is a sequence in the closed convex space $M$ of all $\mu\in B_{b,loc}(\mathcal{P}_{1}(\R^{m}))$ satisfying
\begin{equation}\label{eq:growth estimate}
\vartheta_{1}(\mu(t),\delta_{0})\leq e^{\int_{t_{0}}^{t}\!g_{1}(s)\,ds} E\big[|\xi|_{1}\big] + \int_{t_{0}}^{t}\!e^{\int_{s}^{t}\!g_{1}(\tilde{s})\,d\tilde{s}}(|\kappa|_{1} + h_{1})(s)\,ds
\end{equation}
for each $t\geq t_{0}$ as soon as $\mu_{0}\in M$.
\end{enumerate}
\end{Theorem}

\begin{Remark}
For $\mu_{0} = \delta_{0}$ it holds that $\Delta(t) \leq \sup_{s\in [t_{0},t]} E\big[|X_{s}^{\xi,\delta_{0}}|\big]$ for any $t\geq t_{0}$. If instead $\mu_{0} = \mathcal{L}(X^{\xi})$, then $\Delta = 0$ and $\mu_{n} = \mu_{0}$ for all $n\in\N$.
\end{Remark}

Let us conclude with a deterministic variant of Example~\ref{ex:class of drift maps 2}.

\begin{Example}
Let $l\in\N$, $\eta:[t_{0},\infty[\rightarrow\R_{+}^{m\times l}$ be measurable and locally integrable, $f:\R\rightarrow\R^{m\times l}$ be continuous and $\hat{g}:[t_{0},\infty[\times\R^{m}\times\mathcal{P}\rightarrow\R^{m}$ be Borel measurable such that
\begin{equation*}
b_{i}(s,x,\mu) = \eta_{i,1}(s)f_{i,1}(x_{i}) + \cdots + \eta_{i,l}(s)f_{i,l}(x_{i}) + \hat{g}_{i}(s,x,\mu)
\end{equation*}
for any $(s,x,\mu)\in [t_{0},\infty[\times\R^{m}\times\mathcal{P}$ and $i\in\{1,\dots,m\}$. Then the following three assertions hold:
\begin{enumerate}[(1)]
\item Suppose that $\tilde{\eta}\in\R^{m\times l}$ and $\overline{\eta},\lambda:[t_{0},\infty[\rightarrow\R_{+}^{m}$ are measurable locally integrable maps such that $\mathrm{sgn}(v-\tilde{v})(f_{i,k}(v) - f_{i,k}(\tilde{v})) \leq \tilde{\eta}_{i,k}|v-\tilde{v}|$ and
\begin{align*}
|\hat{g}_{i}(\cdot,x,\mu) - \hat{g}_{i}(\cdot,\tilde{x},\tilde{\mu})| &\leq \overline{\eta}_{i}|x-\tilde{x}|_{1} + \lambda_{i}\vartheta(\mu,\tilde{\mu})
\end{align*}
for all $v,\tilde{v}\in\R$, $x,\tilde{x}\in\R^{m}$, $\mu,\tilde{\mu}\in\mathcal{P}$, $i\in\{1,\dots,m\}$ and $k\in\{1,\dots,l\}$. Then the partial Lipschitz condition~\eqref{det.co:5} for $b$ is valid.

\item Let $\tilde{\upsilon}\in\R^{m\times l}$ and $\kappa,\overline{\upsilon},\chi:[t_{0},\infty[\rightarrow\R_{+}^{m}$ be measurable locally integrable maps satisfying $\mathrm{sgn}(v)f_{i,k}(v) \leq \tilde{\upsilon}_{i,k}|v|$ and
\begin{equation*}
|\hat{g}_{i}(\cdot,x,\mu)| \leq \kappa_{i} + \overline{\upsilon}_{i}|x|_{1} + \chi_{i}\vartheta(\mu,\delta_{0})
\end{equation*}
for any $v\in\R$, $(x,\mu)\in\R^{m}\times\mathcal{P}$, $i\in\{1,\dots,m\}$ and $k\in\{1,\dots,l\}$. Then the partial affine growth condition~\eqref{det.co:6} for $b$ is satisfied.

\item Let the conditions in~(1) and~(2) hold and the maps $\eta$, $\kappa$, $\overline{\upsilon}$ and $\chi$ be actually locally bounded. Then the continuity and boundedness condition~\eqref{det.co:3} for $b$ follows.
\end{enumerate}

For example, if $f_{i,k}$ is decreasing, $f_{i,k}\geq 0$ on $]-\infty,0[$ and $f_{i,k}\leq 0$ on $]0,\infty[$ for all $i\in\{1,\dots,m\}$ and $k\in\{1,\dots,l\}$, then $\tilde{\eta} = \tilde{\upsilon} = 0$ is feasible in~(1) and~(2). More specifically, we may take
\begin{equation*}
f_{i,1}(v) = - v^{n_{i,1}},\dots, f_{i,l}(v) = - v^{n_{i,l}}
\end{equation*}
for all $v\in\R$ and $i\in\{1,\dots,m\}$ and some $n\in\N^{m\times l}$ with \emph{odd coordinates}. In general, if $\sigma$ satisfies~\eqref{det.co:1}, $\mathcal{P}_{1}(\R^{m})\subseteq\mathcal{P}$, $E[|\xi|_{1}] < \infty$ and the conditions in (1)-(3) hold, then all assertions of Theorem~\ref{th:strong existence} apply and the coefficients reduce to
\begin{equation*}
\gamma_{1,0} = \max_{j = 1,\dots,m} \underline{\eta}_{j} + |\overline{\eta}|_{1},\quad g_{1} = \max_{j = 1,\dots,m}\underline{\upsilon}_{j} + |\overline{\upsilon}|_{1} + |\chi|_{1}\quad\text{and}\quad h_{1} = 0
\end{equation*}
with the measurable locally bounded maps $\underline{\eta},\underline{\upsilon}:[t_{0},\infty[\rightarrow\R_{+}^{m}$ coordinatewise given by $\underline{\eta}_{j} := \sum_{k=1}^{l}\eta_{j,k}\tilde{\eta}_{j,k}$ and $\underline{\upsilon}_{j} := \sum_{k=1}^{l}\eta_{j,k}\tilde{\upsilon}_{j,k}$.
\end{Example}

\section{Moment and pathwise asymptotic estimations for random It{\^o} processes}\label{se:4}

\subsection{Auxiliary moment bounds}\label{se:4.1}

In the sequel, let $\B$ and $\Sigma$ be two progressively measurable processes with values in $\R^{m}$ and $\R^{m\times d}$, respectively, such that $\int_{t_{0}}^{\cdot}\!|\B_{s}| + |\Sigma_{s}|^{2}\,ds < \infty$. We will derive \emph{quantitative $L^{1}$-estimates} for an $\R^{m}$-valued adapted continuous process $Y$ satisfying
\begin{equation*}
Y = Y_{t_{0}} + \int_{t_{0}}^{\cdot}\!\B_{s}\,ds + \int_{t_{0}}^{\cdot}\!\Sigma_{s}\,dW_{s}\quad\text{a.s.,}
\end{equation*}
which we call a \emph{random It{\^o} process} with drift $\B$ and diffusion $\Sigma$. First, let us recall an approximation of the identity function on $\R_{+}$, used by Yamada and Watanabe~\cite{YamWat71} to prove pathwise uniqueness for SDEs.

For any $i\in\{1,\dots,m\}$ and each increasing $\hat{\rho}_{i}\in C(\R_{+})$ that is positive on $]0,\infty[$ and satisfies $\int_{0}^{1}\!\hat{\rho}_{i}(v)^{-2}\,dv = \infty$, there are a strictly decreasing zero sequence $(a_{i,n})_{n\in\N_{0}}$ in $]0,1]$ and an increasing sequence $(\psi_{i,n})_{n\in\N}$ of non-negative functions in $C^{2}(\R_{+})$ such that
\begin{equation*}
\psi_{i,n}'\in [0,1],\quad\psi_{i,n}' = \psi_{i,n}'\mathbbm{1}_{]0,a_{i,n-1}[} + \mathbbm{1}_{[a_{i,n-1},\infty[}\quad\text{and}\quad 0\leq \psi_{i,n}'' \leq \frac{2}{n}\hat{\rho}_{i}^{-2}\mathbbm{1}_{]a_{i,n},a_{i,n-1}[}
\end{equation*}
for any $n\in\N$ and hence, $\psi_{i,n}(0) = \psi_{i,n}'(0) = \psi_{i,n}''(0) = 0$. These conditions ensure that $\sup_{n\in\N} \psi_{i,n}(x) = x$ and $\lim_{n\uparrow\infty}\psi_{i,n}'(x) = 1$ for all $x > 0$, which we combine with an application of It{\^o's} formula.

\begin{Lemma}\label{le:Ito product rule}
Let $\psi\in C^{2}(\R_{+})$ satisfy $\psi'(0) = \psi''(0) = 0$ and $U$ be an $\R^{m}$-valued adapted locally absolutely continuous process. Then
\begin{equation*}
\begin{split}
u\psi(|U'Y|) &= u(t_{0})\psi(|U_{t_{0}}'Y_{t_{0}}|) + \int_{t_{0}}^{\cdot}\!\psi(|U_{s}'Y_{s}|)\,du(s)\\
&\quad + \int_{t_{0}}^{\cdot}\!u(s)\bigg(\psi'(|U_{s}'Y_{s}|)\mathrm{sgn}(U_{s}'Y_{s})\big(\dot{U}_{s}'Y_{s} + U_{s}'\B_{s}\big) + \frac{1}{2}\psi''(|U_{s}'Y_{s}|)|U_{s}'\Sigma_{s}|^{2}\bigg)\,ds\\
&\quad + \int_{t_{0}}^{\cdot}\!u(s)\psi'(|U_{s}'Y_{s}|)\mathrm{sgn}(U_{s}'Y_{s})U_{s}'\Sigma_{s}\,dW_{s}\quad\text{a.s.}
\end{split}
\end{equation*}
for any continuous function $u:[t_{0},\infty[\rightarrow\R$ that is locally of bounded variation.
\end{Lemma}

\begin{proof}
We define $\varphi\in C^{2,2}(\R^{m}\times\R^{m})$ by $\varphi(x,y) := \psi(|x'y|)$ with first- and second-order derivatives with respect to the first coordinate
\begin{equation*}
D_{x}\varphi(x,y) = \psi'(|x'y|)\mathrm{sgn}(x'y)y'\quad\text{and}\quad D_{x}^{2}\varphi(x,y) = \psi''(|x'y|)y y'
\end{equation*}
for any $x,y\in\R^{m}$. Its first- and second-order derivatives relative to the second coordinate satisfy $D_{y}\varphi(x,y) = D_{x}\varphi(y,x)$ and $D_{y}^{2}\varphi(x,y) = D_{x}^{2}\varphi(y,x)$, as $\varphi(x,y)$ $= \varphi(y,x)$. Thus,
\begin{align*}
\varphi(U,Y) &= \varphi(U_{t_{0}},Y_{t_{0}}) + \int_{t_{0}}^{\cdot}\!D_{x}\varphi(Y_{s},U_{s})\Sigma_{s}\,dW_{s}\\
&\quad + \int_{t_{0}}^{\cdot}\!D_{x}\varphi(U_{s},Y_{s})\dot{U}_{s} + D_{x}\varphi(Y_{s},U_{s})\B_{s} + \frac{1}{2}\mathrm{tr}\big(D_{x}^{2}\varphi(Y_{s},U_{s})\Sigma_{s}\Sigma_{s}'\big)\,ds\quad\text{a.s.,}
\end{align*}
by It{\^o'}s formula. As $u$ is locally of bounded variation, the asserted identity follows from It{\^o'}s product rule.
\end{proof}

Now we can state an auxiliary moment estimate.

\begin{Proposition}\label{pr:auxiliary moment estimate}
Suppose that $Z$ and $\hat{\eta}$ are two progressively measurable processes with values in $\R^{m}$ and $\R_{+}$, respectively, and $\tau$ is a stopping time satisfying $\int_{t_{0}}^{\cdot}\!|Z_{s}| + \hat{\eta}_{s}^{2}\,ds < \infty$ and
\begin{align*}
\mathrm{sgn}(Y_{s}^{(i)})\B_{s}^{(i)} \leq Z_{s}^{(i)}\quad\text{on $\{Y_{s}^{(i)} \neq 0\}$}\quad\text{and}\quad |e_{i}'\Sigma_{s}| \leq \hat{\eta}_{s}\hat{\rho}_{i}(|Y_{s}^{(i)}|)
\end{align*}
for all $s\in [t_{0},\tau[$ a.s.~for any $i\in\{1,\dots,m\}$. If $u:[t_{0},\infty[\rightarrow\R_{+}$ is locally absolutely continuous, then
\begin{equation}\label{eq:auxiliary moment estimate}
\begin{split}
E\big[u(t\wedge\tau)|Y_{t}^{\tau}|_{1}\big] &\leq u(t_{0})E\big[|Y_{t_{0}}|_{1}\big]\\
&\quad + E\bigg[\int_{t_{0}}^{t\wedge\tau}\!\dot{u}(s)|Y_{s}|_{1} + u(s)\sum_{i=1}^{m} Z_{s}^{(i)}\mathbbm{1}_{\{Y_{s}^{(i)}\neq 0\}}\,ds\bigg]
\end{split}
\end{equation}
for each $t\geq t_{0}$ for which $\int_{t_{0}}^{t\wedge\tau}\!|\dot{u}(s)|Y_{s}|_{1} + u(s)\sum_{i=1}^{m}Z_{s}^{(i)}\mathbbm{1}_{\{Y_{s}^{(i)}\neq 0\}}|\,ds$ is integrable.
\end{Proposition}

\begin{proof}
For fixed $n\in\N$ we infer from the preceding lemma that the $\R_{+}$-valued adapted continuous process $X^{(n)} := \sum_{i=1}^{m}\psi_{i,n}(|Y^{(i)}|)$ is a semimartingale satisfying
\begin{align*}
u(\cdot\wedge\tau)X_{\cdot\wedge\tau}^{(n)} &\leq u(t_{0})X_{t_{0}}^{(n)} + \int_{t_{0}}^{\cdot\wedge\tau}\!\dot{u}(s)X_{s}^{(n)} + u(s)\bigg(\sum_{i=1}^{m}\psi_{i,n}'(|Y_{s}^{(i)}|)Z_{s}^{(i)} + \frac{m}{n}\hat{\eta}_{s}^{2}\bigg)\,ds\\
&\quad + \int_{t_{0}}^{\cdot\wedge\tau}\!u(s)\sum_{i=1}^{m}\psi_{i,n}'(|Y_{s}^{(i)}|)\mathrm{sgn}(Y_{s}^{(i)})e_{i}'\Sigma_{s}\,dW_{s}\quad\text{a.s.}
\end{align*}
Let us now assume that $E[|Y_{t_{0}}|_{1}] < \infty$, as otherwise the right-hand terms in~\eqref{eq:auxiliary moment estimate} are infinite. In this context, we readily notice that $E[X_{t_{0}}^{(n)}]\leq E[|Y_{t_{0}}|_{1}]$ and
\begin{equation*}
\sum_{i=1}^{m}|\psi_{i,n}'(|Y_{s}^{(i)}|)\mathrm{sgn}(Y_{s}^{(i)})e_{i}'\Sigma_{s}|\leq \hat{\eta}_{s}\sum_{i=1}^{m}\hat{\rho}_{i}(|Y_{s}|_{1})\quad\text{for all $s\in [t_{0},\tau[$ a.s.}
\end{equation*}
So, we define a stopping time by $\tau_{k}:=\inf\{t\geq t_{0}\,|\, |Y_{t}|_{1}\geq k\text{ or }\int_{t_{0}}^{t}\!|Z_{s}| +\hat{\eta}_{s}^{2}\,ds\geq k\}\wedge\tau$ for given $k\in\N$ to get that
\begin{align*}
E\big[u(t\wedge\tau_{k})X_{t\wedge \tau_{k}}^{(n)}\big] &\leq u(t_{0})E[X_{t_{0}}^{(n)}]\\
&\quad + E\bigg[\int_{t_{0}}^{t\wedge\tau_{k}}\!\dot{u}(s)X_{s}^{(n)} + u(s)\bigg(\sum_{i=1}^{m}\psi_{i,n}'(|Y_{s}^{(i)}|)Z_{s}^{(i)} + \frac{m}{n}\hat{\eta}_{s}^{2}\bigg)\,ds\bigg]
\end{align*}
for fixed $t\geq t_{0}$. By monotone and dominated convergence, we may take the limit $n\uparrow\infty$ to deduce~\eqref{eq:auxiliary moment estimate} when $\tau$ is replaced by $\tau_{k}$, as $(X^{(n)})_{n\in\N}$ is an increasing sequence converging pointwise to $|Y|_{1}$. 

Finally, Fatou's lemma, another application of the dominated convergence theorem and the fact that $\sup_{k\in\N}\tau_{k} = \tau$ give the claimed bound, under the stated integrability condition.
\end{proof}

\begin{Remark}\label{re:estimation}
If there are $i\in\{1,\dots,m\}$, a measurable function $\rho:\R_{+}\rightarrow\R$ and a progressively measurable process $\alpha$ such that $\rho(0)=0$ and $Z^{(i)}  = \alpha  + \rho(|Y^{(i)}|)$, then
\begin{equation*}
Z_{s}^{(i)}\mathbbm{1}_{\{Y_{s}^{(i)}\neq 0\}} \leq \alpha_{s}^{+} + \rho(|Y_{s}^{(i)}|)\quad\text{for all $s\geq t_{0}$.}
\end{equation*}
\end{Remark}

We recall a Burkholder-Davis-Gundy inequality for stochastic integrals driven by $W$ from~\cite{Mao08}[Theorem 7.3]. Namely, for $p \geq 2$ let $\overline{w}_{p}:= (p^{p+1}/(2(p-1)^{p-1}))^{p/2}$, if $p > 2$, and $\overline{w}_{p} := 4$, if $p=2$. Then
\begin{equation}\label{eq:special BDG-inequality}
E\bigg[\sup_{\tilde{s}\in [t_{0},t]}\bigg|\int_{t_{0}}^{\tilde{s}}\! X_{s}\,dW_{s}\bigg|^{p}\bigg] \leq \overline{w}_{p}E\bigg[\bigg(\int_{t_{0}}^{t}\!|X_{s}|^{2}\,ds\bigg)^{\frac{p}{2}}\bigg]
\end{equation}
for each $\R^{m\times d}$-valued progressively measurable process $X$ and any $t\geq t_{0}$ for which $\int_{t_{0}}^{t}\!|X_{s}|^{2}\,ds < \infty$. Now we conclude with an auxiliary moment estimate in the supremum norm.

\begin{Proposition}\label{pr:auxiliary supremum moment estimate}
Let $p \geq 1$, $Z$ be an $\R^{m}$-valued progressively measurable process with locally integrable paths, $\hat{\eta}:[t_{0},\infty[\rightarrow\R_{+}^{m}$ be measurable and locally square-integrable and $\tau$ be a stopping time such that
\begin{equation*}
\mathrm{sgn}(Y_{s}^{(i)})\B_{s}^{(i)} \leq Z_{s}^{(i)}\quad\text{on $\{Y_{s}^{(i)} \neq 0\}$} \quad\text{and}\quad |e_{i}'\Sigma_{s}| \leq \hat{\eta}_{i}(s)\hat{\rho}_{i}(|Y_{s}^{(i)}|)
\end{equation*}
for any $s\in [t_{0},\tau[$ a.s.~for all $i\in\{1,\dots,m\}$. Then any locally absolutely continuous function $u:[t_{0},\infty[\rightarrow\R_{+}$ satisfies
\begin{equation}\label{eq:auxiliary supremum moment estimate}
\begin{split}
&E\bigg[\bigg(\sup_{s\in [t_{0},t]} u(s\wedge\tau)|Y_{s}^{\tau}|_{1} - u(t_{0})|Y_{t_{0}}|_{1}\bigg)^{p}\bigg]^{\frac{1}{p}}\\
&\leq E\bigg[\bigg(\int_{t_{0}}^{t\wedge\tau}\!\bigg(\dot{u}(s)|Y_{s}|_{1} + u(s)\sum_{i=1}^{m}Z_{s}^{(i)}\mathbbm{1}_{\{Y_{s}^{(i)}\neq 0\}}\bigg)^{+}\,ds\bigg)^{p}\bigg]^{\frac{1}{p}}\\
&\quad + \bigg(\int_{t_{0}}^{t}\!|\hat{\eta}(s)|_{1}^{2}\,ds\bigg)^{\frac{1}{2} - \frac{1}{p_{0}}}\bigg(\overline{w}_{p_{0}}\int_{t_{0}}^{t}\!|\hat{\eta}(s)|_{1}^{2}u(s)^{p_{0}}E\big[\hat{\rho}\big(|Y_{s}|_{1}\big)^{p_{0}}\mathbbm{1}_{\{\tau > s\}}\big]\,ds\bigg)^{\frac{1}{p_{0}}}
\end{split}
\end{equation}
for all $t \geq t_{0}$ with $p_{0}:=p\vee 2$ and $\hat{\rho} := \max_{i = 1,\dots,m} \hat{\rho}_{i}$.
\end{Proposition}

\begin{proof}
For given $k,n\in\N$ we set $\tau_{k}:=\inf\{t\geq t_{0}\,|\, |Y_{t}|_{1} \geq k\text{ or } \int_{t_{0}}^{t}\!|Z_{s}|\,ds \geq k\}\wedge\tau$ and $X^{(n)} := \sum_{i=1}^{m}\psi_{i,n}(|Y^{(i)}|)$. Then Lemma~\ref{le:Ito product rule} and Minkowski's inequality show that
\begin{equation}\label{eq:auxiliary supremum moment estimate 2}
\begin{split}
E\bigg[\bigg(\sup_{s\in [t_{0},t]} &u(s\wedge\tau_{k})X_{s\wedge\tau_{k}}^{(n)} - u(t_{0})X_{t_{0}}^{(n)}\bigg)^{p}\bigg]^{\frac{1}{p}}\\
&\leq E\bigg[\bigg(\int_{t_{0}}^{t\wedge\tau_{k}}\!\bigg(\dot{u}(s)X_{s}^{(n)} + u(s)\sum_{i=1}^{m}\psi_{i,n}'(|Y_{s}^{(i)}|)Z_{s}^{(i)}\bigg)^{+}\,ds\bigg)^{p}\bigg]^{\frac{1}{p}}\\
&\quad + \frac{1}{n}\int_{t_{0}}^{t}\!u(s)\sum_{i=1}^{m}\hat{\eta}_{i}(s)^{2}\,ds + E\bigg[\bigg(\sup_{s\in [t_{0},t]} I_{s\wedge\tau_{k}}^{(n)}\bigg)^{p}\bigg]^{\frac{1}{p}}
\end{split}
\end{equation}
for fixed $t\geq t_{0}$, where $I^{(n)}$ denotes a continuous local martingale with $I_{t_{0}}^{(n)} = 0$ that is indistinguishable from the stochastic integral
\begin{equation*}
\int_{t_{0}}^{\cdot}\!u(s)\sum_{i=1}^{m}\psi_{i,n}'(|Y_{s}^{(i)}|)\mathrm{sgn}(Y_{s}^{(i)})e_{i}'\Sigma_{s}\,dW_{s}.
\end{equation*}
Moreover, as $\sum_{i=1}^{m}\psi_{i,n}'(|Y_{s}^{(i)}|)|e_{i}'\Sigma_{s}|\leq |\hat{\eta}(s)|_{1}\hat{\rho}(|Y_{s}|_{1})$ for all $s\in [t_{0},t]$ with $s < \tau_{k}$ a.s., it follows from H\oe lder's inequality,~\eqref{eq:special BDG-inequality} and Jensen's inequality that
\begin{equation}\label{eq:auxiliary supremum moment estimate 3}
\begin{split}
\overline{w}_{p_{0}}^{-1} E\bigg[\sup_{s\in [t_{0},t]} &|I_{s\wedge\tau_{k}}^{(n)}|^{p}\bigg]^{\frac{p_{0}}{p}} \leq E\bigg[\bigg(\int_{t_{0}}^{t\wedge\tau_{k}}\!|\hat{\eta}(s)|_{1}^{2}u(s)^{2}\hat{\rho}\big(|Y_{s}|_{1}\big)\,ds\bigg)^{\frac{p_{0}}{2}}\bigg]\\
&\leq \bigg(\int_{t_{0}}^{t}\!|\hat{\eta}(s)|_{1}^{2}\,ds\bigg)^{\frac{p_{0}}{2} - 1}\int_{t_{0}}^{t}\!|\hat{\eta}(s)|_{1}^{2} u(s)^{p_{0}}E\big[\hat{\rho}\big(|Y_{s}|_{1}\big)^{p_{0}}\mathbbm{1}_{\{\tau_{k} > s\}}\big]\,ds.
\end{split}
\end{equation}
Now recall that any sequence $(x_{n})_{n\in\N}$ of real-valued functions on $[t_{0},t]$ and each function $x:[t_{0},t]\rightarrow\R$ such that $x(s)\leq\liminf_{n\uparrow\infty} x_{n}(s)$ for all $s\in [t_{0},t]$ satisfies
\begin{equation*}
\sup_{s\in [t_{0},t]} x(s) \leq \liminf_{n\uparrow\infty} \sup_{s\in [t_{0},t]} x_{n}(s).
\end{equation*} 
In combination with Fatou's lemma, this shows that~\eqref{eq:auxiliary supremum moment estimate} follows when $\tau$ is replaced by $\tau_{k}$ from~\eqref{eq:auxiliary supremum moment estimate 2},~\eqref{eq:auxiliary supremum moment estimate 3} and dominated convergence. As $\sup_{k\in\N} \tau_{k} = \tau$, monotone convergence yields the asserted estimate. 
\end{proof}

\subsection{Quantitative first moment estimates}

To deduce an $L^{1}$-estimate based on Bihari's inequality from the results of Section~\ref{se:4.1}, we fix $l\in\N$ and $\alpha,\beta\in ]0,1]^{l}$ and introduce two assumptions on the random It{\^o} process $Y$:
\begin{enumerate}[label=(A.\arabic*), ref=A.\arabic*, leftmargin=\widthof{(A.2)} + \labelsep]
\item\label{Assumption 1} For any $n\in\N$ there are increasing $\hat{\rho}_{1,n},\dots,\hat{\rho}_{m,n}\in C(\R_{+})$ and an $\R_{+}$-valued progressively measurable process $\hat{\eta}^{(n)}$ with locally square-integrable paths so that
\begin{equation*}
\hat{\rho}_{i,n} > 0\quad\text{on $]0,\infty[$,}\quad \int_{0}^{1}\frac{1}{\hat{\rho}_{i,n}(v)^{2}}\,dv = \infty\quad\text{and}\quad |e_{i}'\Sigma_{s}|\leq \hat{\eta}_{s}^{(n)}\hat{\rho}_{i,n}(|Y_{s}^{(i)}|)
\end{equation*}
for all $s\geq t_{0}$ with $|Y_{s}|_{1}\leq n$ a.s.~for every $i\in\{1,\dots,m\}$.

\item\label{Assumption 2} There exist $\rho_{1},\dots,\rho_{l},\varrho_{1},\dots,\varrho_{l}\in C(\R_{+})$, a measurable map $\theta:[t_{0},\infty[\rightarrow\R_{+}^{l}$ and
\begin{equation*}
\text{an $\R_{+}^{m}$-valued process $\kappa$}\quad\text{and}\quad\text{two $\R_{+}^{m\times l}$-valued processes $\eta$, $\lambda$}
\end{equation*}
that are all progressively measurable and have locally integrable paths such that
\begin{align*}
\mathrm{sgn}(Y^{(i)})\B^{(i)} \leq \kappa^{(i)} + \sum_{k=1}^{l} \eta^{(i,k)}\rho_{k}(|Y|_{1}) + \lambda^{(i,k)}\varrho_{k}\circ\theta_{k}\quad\text{a.s.}
\end{align*}
for each $i\in\{1,\dots,m\}$. Further, $\rho_{k}$, $\varrho_{k}$ are positive on $]0,\infty[$ and vanish at $0$, $\rho_{k}^{\frac{1}{\alpha_{k}}}$ is concave, $\varrho_{k}$ is increasing and
\begin{equation*}
E\big[|\kappa|_{1}\big],\quad \bigg[\sum_{i=1}^{m}\eta^{(i,k)}\bigg]_{\frac{1}{1-\alpha_{k}}},\quad \sum_{i=1}^{m} E\big[\lambda^{(i,k)}\big]
\end{equation*}
are locally integrable for all $k\in\{1,\dots,l\}$.
\end{enumerate}

For a measurable map $\theta:[t_{0},\infty[\rightarrow\R_{+}^{l}$ and an $\R_{+}^{m\times l}$-valued progressively measurable process $\lambda$, as in~\eqref{Assumption 2}, we rely on a domination condition:
\begin{enumerate}[label=(A.\arabic*), ref=A.\arabic*, leftmargin=\widthof{(A.3)} + \labelsep]
\setcounter{enumi}{2}
\item\label{Assumption 3} We have $\theta_{k}(s) \leq  E[|Y_{s}|_{1}]$ for all $s\geq t_{0}$ with $\sum_{i=1}^{m}E[\lambda_{s}^{(i,k)}] > 0$ for any $k\in\{1,\dots,l\}$.
\end{enumerate}

If~\eqref{Assumption 2} is satisfied, then we may define two measurable locally integrable functions $\gamma,\delta:[t_{0},\infty[\rightarrow [0,\infty]$ via
\begin{align*}
\gamma(s) &:= \sum_{k=1}^{l}\alpha_{k}\bigg[\sum_{i=1}^{m}\eta_{s}^{(i,k)}\bigg]_{\frac{1}{1-\alpha_{k}}} + \beta_{k}\sum_{i=1}^{m}E\big[\lambda_{s}^{(i,k)}\big]\quad\text{and}\\
\delta(s) &:= \sum_{k=1}^{l}(1-\alpha_{k})\bigg[\sum_{i=1}^{m}\eta_{s}^{(i,k)}\bigg]_{\frac{1}{1-\alpha_{k}}} + (1-\beta_{k})\sum_{i=1}^{m}E\big[\lambda_{s}^{(i,k)}\big].
\end{align*}
We also recall the definitions~\eqref{eq:rho-map 1} and~\eqref{eq:rho-map 2}. This leads to a general estimation result.

\begin{Theorem}\label{th:abstract moment estimate}
Let~\eqref{Assumption 1}-\eqref{Assumption 3} hold, $E[|Y_{t_{0}}|_{1}] < \infty$, $\sum_{k=1}^{l}\sum_{i=1}^{m} E[\lambda^{(i,k)}]\varrho_{k}\circ\theta_{k}$ be locally integrable and $\rho_{0},\varrho_{0}\in C(\R_{+})$ be defined by
\begin{equation*}
\rho_{0}(v):=\max_{k = 1,\dots,l} \rho_{k}(v)^{\frac{1}{\alpha_{k}}}\quad\text{and}\quad\varrho_{0}(v):=\rho_{0}(v)\vee\max_{k = 1,\dots,l}\varrho_{k}(v)^{\frac{1}{\beta_{k}}}.
\end{equation*}
If $\Phi_{\rho_{0}}(\infty) = \infty$ or $\sum_{k=1}^{l}\sum_{i=1}^{m}E[\eta^{(i,k)}\rho_{k}(|Y|_{1})]$ is locally integrable, then $E[|Y|_{1}]$ is locally bounded and
\begin{equation*}
\sup_{s\in [t_{0},t]}E\big[|Y_{s}|_{1}\big] \leq \Psi_{\varrho_{0}}\bigg(E\big[|Y_{t_{0}}|_{1}\big] + \int_{t_{0}}^{t}\!E\big[|\kappa_{s}|_{1}\big] + \delta(s)\,ds,\int_{t_{0}}^{t}\!\gamma(s)\,ds\bigg)
\end{equation*}
for any $t\in [t_{0},t_{0}^{+}[$, where $t_{0}^{+} > t_{0}$ denotes the supremum over all $t \geq t_{0}$ for which
\begin{equation*}
\bigg(E\big[|Y_{t_{0}}|_{1}\big] + \int_{t_{0}}^{t}\!E\big[|\kappa_{s}|_{1}\big] + \delta(s)\,ds,\int_{t_{0}}^{t}\!\gamma(s)\,ds\bigg)\in D_{\varrho_{0}}.
\end{equation*} 
\end{Theorem}

\begin{proof}
We introduce the stopping time $\tau_{n}:=\inf\{t \geq t_{0}\,|\, |Y_{t}|_{1}\geq n\}$ for fixed $n\in\N$ and set $\hat{\kappa} := E[|\kappa_{s}|_{1}] + \sum_{k=1}^{l}\sum_{i=1}^{m}E[\lambda^{(i,k)}]\varrho_{k}\circ\theta_{k}$. Then Proposition~\ref{pr:auxiliary moment estimate} and Remark~\ref{re:estimation} imply that
\begin{equation}\label{eq:abstract moment estimate 1}
\begin{split}
E\big[|Y_{t}^{\tau_{n}}|_{1}\big] &\leq E\big[|Y_{t_{0}}|_{1}\big] + \int_{t_{0}}^{t}\!\hat{\kappa}(s) + \sum_{k=1}^{l}\sum_{i=1}^{m}E\big[\eta_{s}^{(i,k)}\rho_{k}(|Y_{s}|_{1})\mathbbm{1}_{\{\tau_{n} > s\}}\big]\,ds
\end{split}
\end{equation}
for given $t\geq t_{0}$. Thereby, we notice that the local integrability of the measurable function $[t_{0},\infty[\rightarrow [0,\infty]$, $s\mapsto \sum_{i=1}^{m}E[\eta_{s}^{(i,k)}\rho_{k}(|Y_{s}|_{1})\mathbbm{1}_{\{\tau_{n} > s\}}]$ follows from~\eqref{eq:essential inequality}, which yields
\begin{equation}\label{eq:abstract moment estimate 2}
\sum_{i=1}^{m}E\big[\eta_{s}^{(i,k)}\rho_{k}(|Y_{s}|_{1})\mathbbm{1}_{\{\tau > s\}}\big]\leq \bigg[\sum_{i=1}^{m}\eta^{(i,k)}\bigg]_{\frac{1}{1-\alpha_{k}}}\big(1-\alpha_{k} + \alpha_{k}\rho_{k}\big(E\big[|Y_{s}^{\tau}|_{1}\big]\big)^{\frac{1}{\alpha_{k}}}\big)
\end{equation}
for all $s\in [t_{0},t]$, each $k\in\{1,\dots,l\}$ and every stopping time $\tau$ for which $E[|Y^{\tau}|_{1}]$ is finite, because $\rho_{k}^{1/\alpha_{k}}$ is concave, by assumption.

Thus, let us set $\hat{\delta}:=\sum_{k=1}^{l}(1-\alpha_{k})[\sum_{i=1}^{m}\eta^{(i,k)}]_{(1-\alpha_{k})^{-1}}$. If $\Phi_{\rho_{0}}(\infty) = \infty$ holds, then we apply Bihari's inequality to~\eqref{eq:abstract moment estimate 1} and infer from Fatou's lemma that
\begin{align*}
E\big[|Y_{t}|_{1}\big]&\leq \liminf_{n\uparrow\infty} E\big[|Y_{t}^{\tau_{n}}|_{1}\big]\\
&\leq \Psi_{\rho_{0}}\bigg(E\big[|Y_{t_{0}}|_{1}\big] + \int_{t_{0}}^{t}\!(\hat{\kappa} + \hat{\delta})(s)\,ds,\int_{t_{0}}^{t}\!\sum_{k=1}^{l}\alpha_{k}\bigg[\sum_{i=1}^{m}\eta^{(i,k)}\bigg]_{\frac{1}{1-\alpha_{k}}}\,ds\bigg),
\end{align*}
as the domain of $\Psi_{\rho_{0}}$ satisfies $D_{\rho_{0}} = \R_{+}^{2}$. For this reason, $E[|Y|_{1}]$ is locally bounded in this case. By choosing $\tau=\infty$ in~\eqref{eq:abstract moment estimate 2}, we see that it suffices to consider the case when $\sum_{k=1}^{l}\sum_{i=1}^{m}E[\eta^{(i,k)}\rho_{k}(|Y|_{1})]$ is locally integrable. Then
\begin{align*}
E\big[|Y_{t}|_{1}\big] \leq E\big[|Y_{t_{0}}|_{1}\big] + \int_{t_{0}}^{t}\!(\hat{\kappa} + \hat{\delta})(s) + \sum_{k=1}^{l}\alpha_{k}\bigg[\sum_{i=1}^{m}\eta^{(i,k)}\bigg]_{\frac{1}{1-\alpha_{k}}}\rho_{k}\big(E\big[|Y_{s}|_{1}\big]\big)^{\frac{1}{\alpha_{k}}}\big)\,ds
\end{align*}
follows from~\eqref{eq:abstract moment estimate 1}, Fatou's lemma and~\eqref{eq:abstract moment estimate 2}. Thereby, we readily checked that $E[|Y|_{1}]$ is actually locally bounded. Finally, Young's inequality gives us that
\begin{equation*}
\sum_{i=1}^{m}E\big[\lambda^{(i,k)}\big]\varrho_{k}\circ\theta_{k} \leq \sum_{i=1}^{m}E\big[\lambda^{(i,k)}\big]\big(1 - \beta_{k} + \beta_{k}\varrho_{k}\big(E\big[|Y|_{1}\big]\big)^{\frac{1}{\beta_{k}}}\big)
\end{equation*}
on $[t_{0},t]$ for all $k\in\{1,\dots,l\}$ and we conclude the proof with another application of Bihari's inequality, since $\hat{\delta} + \sum_{k=1}^{l}(1-\beta_{k})\sum_{i=1}^{m}E[\lambda^{(i,k)}] = \delta$.
\end{proof}

For a stability analysis we consider another condition, which explicitly measures the dependence on each coordinate and implies~\eqref{Assumption 2}:
\begin{enumerate}[label=(A.\arabic*), ref=A.\arabic*, leftmargin=\widthof{(A.4)} + \labelsep]
\setcounter{enumi}{3}
\item\label{Assumption 4} There are a measurable map $\theta:[t_{0},\infty[\rightarrow\R_{+}^{l}$ and progressively measurable processes $\kappa$, $\eta$ and $\lambda$ with values in $\R_{+}^{m}$, $\R^{m\times m\times l}$ and $\R_{+}^{m\times l}$, respectively,  such that
\begin{equation*}
\mathrm{sgn}(Y^{(i)})\B^{(i)} \leq \kappa^{(i)} + \sum_{k=1}^{l}\bigg(\sum_{j=1}^{m}\eta^{(i,j,k)} |Y^{(j)}|^{\alpha_{k}}\bigg) + \lambda^{(i,k)}\theta_{k}^{\beta_{k}}\quad\text{a.s.}
\end{equation*}
for any $i\in\{1,\dots,m\}$. Moreover, the paths of $\kappa$, $\eta$, $\lambda$ are locally integrable and we have $\eta^{(i,j,k)}\geq 0$, if $i\neq j$, and \begin{equation*}
E[\kappa^{(i)}],\quad [\eta^{(i,j,k)}]_{\frac{1}{1-\alpha_{k}}},\quad E[\lambda^{(i,k)}]
\end{equation*}
are locally integrable for any $i,j\in\{1,\dots,m\}$ and $k\in\{1,\dots,l\}$.
\end{enumerate}

If~\eqref{Assumption 4} and~\eqref{Assumption 3} hold, then we may utilise the two functions $\gamma_{1}$ and $\hat{\delta}_{1}$ given by~\eqref{eq:stability exponential coefficient} and~\eqref{eq:stability drift coefficient} to get an explicit moment estimate.

\begin{Theorem}\label{th:moment stability estimate}
Let~\eqref{Assumption 1},~\eqref{Assumption 4},~\eqref{Assumption 3} be valid, $E[|Y_{t_{0}}|_{1}] < \infty$ and $\sum_{k=1}^{l}\sum_{i=1}^{m}E[\lambda^{(i,k)}]\theta_{k}^{\beta_{k}}$ be locally integrable. Then
\begin{equation*}
E\big[|Y_{t}|_{1}\big] \leq e^{\int_{t_{0}}^{t}\!\gamma_{1}(s)\,ds}E\big[|Y_{t_{0}}|_{1}\big] + \int_{t_{0}}^{t}\!e^{\int_{s}^{t}\!\gamma_{1}(\tilde{s})\,d\tilde{s}}\big(E\big[|\kappa_{s}|_{1}\big] + \hat{\delta}_{1}(s)\big)\,ds
\end{equation*}
for all $t\geq t_{0}$. In particular if $\gamma_{1}^{+}$, $E[|\kappa_{s}|_{1}]$ and $\hat{\delta}_{1}$ are integrable, then $E[|Y|_{1}] $ is bounded. If in addition $\int_{t_{0}}^{\infty}\!\gamma_{1}^{-}(s)\,ds = \infty$, then $\lim_{t\uparrow\infty} E\big[|Y_{t}|_{1}\big] = 0$.
\end{Theorem}

\begin{proof}
First, we observe that~\eqref{Assumption 2} holds when the appearing process $\eta$ there is replaced by the $\R^{m\times l}$-valued process $\tilde{\eta}$ defined coordinatewise by $\tilde{\eta}^{(i,k)} := \sum_{j=1}^{m} (\eta^{(i,j,k)})^{+}$ and it holds that
\begin{equation*}
\rho_{k}(v) = v^{\alpha_{k}}\quad\text{and}\quad\varrho_{k}(v) = v^{\beta_{k}}
\end{equation*}
for all $v\geq 0$. As $\rho_{0}\in C(\R_{+})$ given by $\rho_{0}(v) = v$ satisfies $\Phi_{\rho_{0}}(\infty) = \infty$, Theorem~\ref{th:abstract moment estimate} shows us that $E[|Y|_{1}]$ is locally bounded. 

Thus, we define an $\R^{m\times l}$-valued process $\hat{\eta}$ coordinatewise by $\hat{\eta}^{(j,k)}:=\sum_{i=1}^{m}\eta^{(i,j,k)}$. Then Proposition~\ref{pr:auxiliary moment estimate} and Remark~\ref{re:estimation} imply that
\begin{equation}\label{eq:moment stability estimate 1}
\begin{split}
&u(t)E\big[|Y_{t}|_{1}\big] \leq u(t_{0})E\big[|Y_{t_{0}}|_{1}\big] + \int_{t_{0}}^{t}\!u(s)E\big[|\kappa_{s}|_{1}\big] + \dot{u}(s)E\big[|Y_{s}|_{1}\big]\,ds\\
&\quad + \int_{t_{0}}^{t}\!u(s)\bigg(\sum_{k=1}^{l}\bigg(\sum_{j=1}^{m}E\big[\hat{\eta}_{s}^{(j,k)}|Y_{s}^{(j)}|^{\alpha_{k}}\big]\bigg) + \sum_{i=1}^{m}E\big[\lambda_{s}^{(i,k)}\big]\theta_{k}(s)^{\beta_{k}}\bigg)\,ds
\end{split}
\end{equation}
for fixed $t\geq t_{0}$ and any locally absolutely continuous function $u:[t_{0},\infty[\rightarrow\R_{+}$. From~\eqref{eq:essential inequality} we directly obtain that
\begin{equation*}
E\big[\hat{\eta}^{(j,k)}|Y^{(j)}|^{\alpha_{k}}\big] \leq \big[\hat{\eta}^{(j,k)}\big]_{\frac{1}{1-\alpha_{k}}}\big(1-\alpha_{k} + \alpha_{k} E\big[|Y^{(j)}|\big]\big)
\end{equation*}
and $\sum_{i=1}^{m}E[\lambda^{(i,k)}]\theta_{k}^{\beta_{k}} \leq \sum_{i=1}^{m}E[\lambda^{(i,k)}](1-\beta_{k} + \beta_{k}E[|Y|_{1}])$ on $[t_{0},t]$ for all $j\in\{1,\dots,m\}$ and $k\in\{1,\dots,l\}$. As a consequence,
\begin{align*}
&\sum_{k=1}^{l}\bigg(\sum_{j=1}^{m}E\big[\hat{\eta}^{(j,k)}|Y^{(j)}|^{\alpha_{k}}\big]\bigg) + \sum_{i=1}^{m}E\big[\lambda^{(i,k)}\big]\theta_{k}^{\beta_{k}} \leq \hat{\delta}_{1} + \gamma_{1}E\big[|Y|_{1}\big]
\end{align*}
a.e.~on $[t_{0},t]$. Thus, by choosing $u(\tilde{s}) = \exp(-\int_{t_{0}}^{\tilde{s}}\!\gamma_{1}(s)\,ds)$ for any $\tilde{s}\geq t_{0}$ in~\eqref{eq:moment stability estimate 1}, we get the asserted estimate after dividing by $u(t)$.

Next, let us assume that $\gamma_{1}^{+}$, $E[|\kappa_{s}|_{1}]$ and $\hat{\delta}_{1}$ are integrable. Then the second assertion follows from the bound
\begin{equation*}
\sup_{t\geq t_{0}} E\big[|Y_{t}|_{1}\big] \leq e^{\int_{t_{0}}^{\infty}\!\gamma_{1}^{+}(s)\,ds}\bigg(E\big[|Y_{t_{0}}|_{1}\big] + \int_{t_{0}}^{\infty}\! E\big[|\kappa_{s}|_{1}\big] + \hat{\delta}_{1}(s)\,ds\bigg).
\end{equation*}
For the last claim, let additionally $\int_{t_{0}}^{\infty}\!\gamma_{1}^{-}(s)\,ds = \infty$. Then $\lim_{t\uparrow\infty}\exp(\int_{s}^{t}\!\gamma_{1}(\tilde{s})\,d\tilde{s})= 0$ for every $s\geq t_{0}$, by monotone convergence. Thus,
\begin{equation*}
\lim_{t\uparrow\infty}\int_{t_{0}}^{t}\!e^{\int_{s}^{t}\!\gamma_{1}(\tilde{s})\,d\tilde{s}}\big(E\big[|\kappa_{s}|_{1}\big] + \hat{\delta}_{1}(s)\big)\,ds = 0
\end{equation*}
follows from dominated convergence, which completes the proof.
\end{proof}

\begin{Remark}\label{re:moment stability}
Suppose that $\kappa=0$ a.s., and let $\hat{\delta}_{1}$ vanish a.e., which holds if $\alpha_{k} = \beta_{k} = 1$ for any $k\in\{1,\dots,l\}$. If $\gamma_{1}^{+}$ is integrable, then
\begin{equation*}
\sup_{t\geq t_{0}}e^{\int_{t_{0}}^{t}\!\gamma_{1}^{-}(s)\,ds}E\big[|Y_{t}|_{1}\big] \leq  e^{\int_{t_{0}}^{\infty}\gamma_{1}^{+}(s)\,ds}E\big[|Y_{t_{0}}|_{1}\big] < \infty.
\end{equation*}
If additionally $\int_{t_{0}}^{\infty}\gamma_{1}^{-}(s)\,ds = \infty$, then from $a\gamma_{1}^{-} + \gamma_{1} = \gamma_{1}^{+} - (1-a)\gamma_{1}^{-}$ we infer that
\begin{equation*}
\lim_{t\uparrow\infty} e^{a\int_{t_{0}}^{t}\!\gamma_{1}^{-}(s)\,ds}E\big[|Y_{t}|_{1}\big] = 0\quad\text{for all $a\in [0,1[$}.
\end{equation*}
These two facts give more insight into the rate of convergence.
\end{Remark}

\subsection{Moment estimates in the supremum norm}

In this section we deduce absolute $p$-th moment estimates in the supremum norm for $p\in [1,2]$. To this end, we require a H\oe lder condition instead of the weaker Osgood condition~\eqref{Assumption 1} on compact sets and restrict~\eqref{Assumption 2} and~\eqref{Assumption 4} as follows:
\begin{enumerate}[label=(A.\arabic*), ref=A.\arabic*, leftmargin=\widthof{(A.5)} + \labelsep]
\setcounter{enumi}{4}
\item\label{Assumption 5} There exists a measurable locally square-integrable function $\hat{\eta}:[t_{0},\infty[\rightarrow\R_{+}^{m}$ such that $|e_{i}'\Sigma| \leq \hat{\eta}_{i}|Y^{(i)}|^{\frac{1}{2}}$ a.s.~for all $i\in\{1,\dots,m\}$.

\item\label{Assumption 6} Assumption~\eqref{Assumption 2} holds when $\rho_{k}(v) = v^{\alpha_{k}}$ for all $v\geq 0$ and $k\in\{1,\dots,l\}$ and $\eta = \overline{\eta}$ for some measurable locally integrable map $\overline{\eta}:[t_{0},\infty[\rightarrow\R_{+}^{m\times l}$.

\item\label{Assumption 7} Assumption~\eqref{Assumption 4} is valid and there exists a measurable locally integrable map $\overline{\eta}:[t_{0},\infty[\rightarrow\R^{m\times m\times l}$ such that $\eta = \overline{\eta}$.
\end{enumerate}

Under~\eqref{Assumption 5} and~\eqref{Assumption 6}, we define two $[0,\infty]$-valued measurable functions $f_{p}$ and $g_{p}$ on the set of all $(t_{1},t)\in [t_{0},\infty[^{2}$ with $t_{1}\leq t$ by
\begin{equation*}
f_{p}(t_{1},t) := E\bigg[\bigg(\int_{t_{1}}^{t}\!|\kappa_{s}|_{1} + \sum_{k=1}^{l}\sum_{i=1}^{m}\lambda_{s}^{(i,k)}\varrho_{k}(\theta_{k}(s))\,ds\bigg)^{p}\bigg]^{\frac{1}{p}} + 2\bigg(\int_{t_{1}}^{t}\!|\hat{\eta}(s)|_{1}^{2}E[|Y_{s}|_{1}]\,ds\bigg)^{\frac{1}{2}}
\end{equation*}
and
\begin{equation*}
g_{p}(t_{1},t) := f_{p}(t_{1},t) + \sum_{k=1}^{l}E\big[|Y_{t_{1}}|_{1}^{\alpha_{k}p}\big]^{\frac{1}{p}} \int_{t_{1}}^{t}\!\sum_{i=1}^{m}\overline{\eta}_{i,k}(s)\,ds.
\end{equation*}
In addition, let us set $\underline{\alpha} := \min_{k=1,\dots,l}\alpha_{k}$ and $\overline{\alpha} := \max_{k = 1,\dots,l}\alpha_{k}$.

\begin{Proposition}\label{pr:abstract supremum moment estimate}
Let~\eqref{Assumption 5},~\eqref{Assumption 6} and~\eqref{Assumption 3} be valid, $\sum_{k=1}^{l}\sum_{i=1}^{m}E[\lambda^{(i,k)}]\varrho_{k}\circ\theta_{k}$ be locally integrable and $\rho_{0}\in C(\R_{+})$ be given by $\rho_{0}(v):= v^{\underline{\alpha}} \mathbbm{1}_{[0,1]}(v)+ v^{\overline{\alpha}}\mathbbm{1}_{]1,\infty[}(v)$. If
\begin{equation}\label{eq:expectations}
E\big[|Y_{t_{0}}|_{1}^{p}\big],\quad E\bigg[\bigg(\int_{t_{0}}^{t}\!|\kappa_{s}|_{1}\,ds\bigg)^{p}\bigg] \quad\text{and}\quad E\bigg[\bigg(\int_{t_{0}}^{t}\!\sum_{k=1}^{l}\sum_{i=1}^{m}\lambda_{s}^{(i,k)}\,ds\bigg)^{p}\bigg]
\end{equation}
are finite, then $\sup_{s\in [t_{0},t]} |Y_{s}|_{1}$ is $p$-fold integrable and
\begin{equation*}
E\bigg[\bigg(\sup_{s\in [t_{1},t]}|Y_{s}|_{1} - |Y_{t_{1}}|_{1}\bigg)^{p}\bigg] \leq \Psi_{\rho_{0}}\bigg((l+1)^{p-1}g_{p}(t_{1},t)^{p},(l+1)^{p-1}\sum_{k=1}^{l}\bigg(\int_{t_{1}}^{t}\!\sum_{i=1}^{m}\overline{\eta}_{i,k}(s)\,ds\bigg)^{p}\bigg)
\end{equation*}
for any $t_{1},t\geq t_{0}$ with $t_{1}\leq t$. In particular, $E[|Y|_{1}^{p}]$ is continuous.
\end{Proposition}

\begin{proof}
If the integrability assertion is true, then $\lim_{n\uparrow\infty}E[|Y_{t_{n}}|_{1}^{p}] = E[|Y_{t}|_{1}^{p}]$ for every sequence $(t_{n})_{n\in\N}$ in $[t_{0},\infty[$ that converges to some $t\geq t_{0}$, by dominated convergence. For this reason, it suffices to show the first two claims.

We set $\tau_{n}:=\inf\{t\geq t_{1}\,|\,|Y_{t}|_{1} \geq n\}$ for given $t_{1}\geq t_{0}$ and $n\in\N$. Then Proposition~\ref{pr:auxiliary supremum moment estimate} and the inequalities of Minkowski and Jensen give
\begin{align*}
E\bigg[\bigg(\sup_{s\in [t_{1},t]}  |Y_{s}^{\tau_{n}}|_{1} &- |Y_{t_{1}}|_{1}\bigg)^{p}\bigg]^{\frac{1}{p}}\leq 2 \bigg(\int_{t_{1}}^{t}\!|\hat{\eta}(s)|_{1}^{2}E\big[|Y_{s}^{\tau_{n}}|_{1}\big]\,ds\bigg)^{\frac{1}{2}}\\
&\quad + E\bigg[\bigg(\int_{t_{1}}^{t\wedge\tau_{n}}\!|\kappa_{s}|_{1} + \sum_{k=1}^{l}\sum_{i=1}^{m}\overline{\eta}_{i,k}(s)|Y_{s}|_{1}^{\alpha_{k}} + \lambda_{s}^{(i,k)}\varrho_{k}(\theta_{k}(s))\,ds\bigg)^{p}\bigg]^{\frac{1}{p}}\\
&\leq f_{p}(t_{1},t) + \sum_{k=1}^{l}\bigg(\int_{t_{1}}^{t}\!\sum_{i=1}^{m}\overline{\eta}_{i,k}(s)\bigg)^{1-\frac{1}{p}}\bigg(\int_{t_{1}}^{t}\!\sum_{i=1}^{m}\overline{\eta}_{i,k}(s)E\big[|Y_{s}^{\tau_{n}}|_{1}^{\alpha_{k} p}\big]\,ds\bigg)^{\frac{1}{p}}
\end{align*}
for fixed $t\geq t_{1}$. We recall that $\Phi_{\rho_{0}}(\infty) = \infty$ and $D_{\rho_{0}} = \R_{+}^{2}$. Thus, if $E[|Y_{t_{1}}|_{}^{p}] < \infty$, then another application of Minkowski's inequality together with Bihari's inequality yield the asserted bound for
\begin{align*}
E\bigg[\bigg(\sup_{s\in [t_{1},t]}|Y_{s}^{\tau_{n}}|_{1} - |Y_{t_{1}}|_{1}\bigg)^{p}\bigg]
\end{align*}
and from Fatou's lemma we readily infer the claimed result. In this context, we may use the fact that $E[|Y|_{1}]$ is locally bounded, by Theorem~\ref{th:abstract moment estimate}. This ensures that $g_{p}(t_{1},t)$ is finite in this case.

Further, by choosing $t_{1} = t_{0}$ the $p$-fold integrability of $\sup_{s\in [t_{0},t]} |Y_{s}|_{1}$ follows from that of $|Y_{t_{0}}|_{1}$ and the finiteness of $g_{p}(t_{0},t)$, which completes the proof.
\end{proof}

If~\eqref{Assumption 7} holds, then the definitions~\eqref{eq:stability exponential coefficient} and~\eqref{eq:stability drift coefficient} for the choice $\lambda = 0$ lead us to two measurable locally integrable functions
\begin{equation*}
\gamma_{1,0} := \max_{j = 1,\dots,m} \sum_{k=1}^{l}\alpha_{k}\bigg(\bigg(\sum_{i=1}^{m}\overline{\eta}_{i,j,k}\bigg)^{+} - \bigg(\sum_{i=1}^{m}\overline{\eta}_{i,j,k}\bigg)^{-}\mathbbm{1}_{\{1\}}(\alpha_{k})\bigg)
\end{equation*}
and
\begin{equation*}
\hat{\delta}_{1,0} := \sum_{k=1}^{l}(1-\alpha_{k})\sum_{j=1}^{m}\bigg(\sum_{i=1}^{m}\overline{\eta}_{i,j,k}\bigg)^{+}.
\end{equation*}
For a measurable locally integrable function $\gamma:[t_{0},\infty[\rightarrow\R$ we introduce an $[0,\infty]$-valued measurable function $h_{\gamma,p}$ on the set of all $(t_{1},t)\in [t_{0},\infty[^{2}$ with $t_{1}\leq t$ by
\begin{align*}
h_{\gamma,p}(t_{1},t) &:= E\bigg[\bigg(\int_{t_{1}}^{t}\!e^{-\int_{t_{1}}^{s}\!\gamma(s_{0})\,ds_{0}}\bigg(|\kappa_{s}|_{1} + \hat{\delta}_{1,0}(s) + \sum_{k=1}^{l}\sum_{i=1}^{m}\lambda_{s}^{(i,k)}\theta_{k}(s)^{\beta_{k}}\bigg)\,ds\bigg)^{p}\bigg]^{\frac{1}{p}}\\
&\quad + 2\bigg(\int_{t_{1}}^{t}\!|\hat{\eta}(s)|_{1}^{2}e^{-2\int_{t_{1}}^{s}\!\gamma(s_{0})\,ds_{0}}E\big[|Y_{s}|_{1}\big]\,ds\bigg)^{\frac{1}{2}}
\end{align*}
and state the analogue of Proposition~\ref{pr:abstract supremum moment estimate} when~\eqref{Assumption 7} instead of~\eqref{Assumption 6} holds.

\begin{Lemma}\label{le:supremum moment stability estimate}
Let~\eqref{Assumption 5},~\eqref{Assumption 7} and~\eqref{Assumption 3} be satisfied and $\sum_{k=1}^{l}\sum_{i=1}^{m}E[\lambda^{(i,k)}]\theta_{k}^{\beta_{k}}$ be locally integrable. If the expectations in~\eqref{eq:expectations} are finite, then
\begin{align*}
&E\bigg[\bigg(\sup_{s\in [t_{1},t]}e^{-\int_{t_{1}}^{s}\!\gamma(s_{0})\,ds_{0}}|Y_{s}|_{1} - |Y_{t_{1}}|_{1}\bigg)^{p}\bigg]^{\frac{1}{p}} \leq h_{\gamma,p}(t_{1},t)\\
&\quad + \bigg(\int_{t_{1}}^{t}\!(\gamma_{1,0} - \gamma)^{+}(s)\,ds\bigg)^{1 - \frac{1}{p}}\bigg(\int_{t_{1}}^{t}\!(\gamma_{1,0} - \gamma)^{+}(s)e^{-p\int_{t_{1}}^{s}\!\gamma(s_{0})\,ds_{0}}E\big[|Y_{s}|_{1}^{p}\big]\,ds\bigg)^{\frac{1}{p}}
\end{align*}
for each measurable locally integrable function $\gamma:[t_{0},\infty[\rightarrow\R$ and all $t_{1},t\geq t_{0}$ with $t_{1}\leq t$.
\end{Lemma}

\begin{proof}
As~\eqref{Assumption 7} is a special case of~\eqref{Assumption 6}, Proposition~\ref{pr:abstract supremum moment estimate} entails that $\sup_{s\in [t_{0},t]} |Y_{s}|_{1}$ is $p$-fold integrable. Moreover, we readily see that
\begin{equation*}
\sum_{k=1}^{l}\sum_{i,j=1}^{m}\overline{\eta}_{i,j,k}|Y^{(j)}|^{\alpha_{k}} \leq \hat{\delta}_{1,0} + \gamma_{1,0}|Y|_{1},
\end{equation*}
by Young's inequality. Hence, the claim follows immediately from Proposition~\ref{pr:auxiliary supremum moment estimate} and the inequalities of Minkowski and Jensen.
\end{proof}

We consider a last restriction that still allows for the mixed H\oe lder condition~in~\eqref{Assumption 4} on a finite time interval:
\begin{enumerate}[label=(A.\arabic*), ref=A.\arabic*, leftmargin=\widthof{(A.3)} + \labelsep]
\setcounter{enumi}{7}
\item\label{Assumption 8} Assumptions~\eqref{Assumption 5} and~\eqref{Assumption 7} hold and there are $t_{1}\geq t_{0}$, $\hat{\delta} > 0$ and $\overline{c}_{0}\geq 0$ such that
\begin{equation*}
\kappa^{(j)} = (1-\beta_{k})\lambda^{(j,k)} = 0\quad\text{and}\quad (1-\alpha_{k})\sum_{i=1}^{m}\overline{\eta}_{i,j,k}\leq 0\quad\text{on $[t_{1},\infty[$}
\end{equation*}
for any $j\in\{1,\dots,m\}$ and $k\in\{1,\dots,l\}$ and
\begin{equation*}
\sup_{t\geq t_{1}}\int_{t}^{t+\hat{\delta}}\!\sum_{i=1}^{m}E\big[\lambda_{s}^{(i,k)}\big]\,ds\vee\int_{t}^{t+\hat{\delta}}\!|\hat{\eta}(s)|_{1}^{2}\,ds\ \leq \overline{c}_{0}
\end{equation*}
for every $k\in\{1,\dots,l\}$ with $\beta_{k} = 1$.
\end{enumerate}

Then the following first moment estimate in the supremum norm will contribute to the pathwise asymptotic analysis of $Y$ in the next section. 

\begin{Proposition}\label{pr:supremum moment stability estimate}
Let~\eqref{Assumption 8} and~\eqref{Assumption 3} hold, $E[|Y_{t_{0}}|_{1}] < \infty$ and $\sum_{k=1}^{l}\sum_{i=1}^{m}E[\lambda^{(i,k)}]\theta_{k}^{\beta_{k}}$ be locally integrable. Further, suppose that there are a measurable locally integrable function $\gamma:[t_{0},\infty[\rightarrow\R$ and $\overline{c}_{\gamma,-1},\dots,\overline{c}_{\gamma,3}\geq 0$ such that
\begin{equation*}
\int_{t_{2}}^{t}\!(\gamma_{1,0} - \gamma)^{+}(s)\,ds \leq \overline{c}_{\gamma,-1},\quad \int_{t_{2}}^{t}\!(\gamma_{1,0} - q\gamma)(s)\,ds\leq \overline{c}_{\gamma,q},\quad \int_{t_{2}}^{t}\!\gamma(s)\,ds \leq \overline{c}_{\gamma,3}
\end{equation*}
for all $t_{2},t\geq t_{1}$ with $t_{2}\leq t < \hat{\delta}$ and $q\in\{0,1,2\}$. Then there is $\overline{c} > 0$ such that
\begin{align*}
E\bigg[\sup_{s\in [t_{2},t]} |Y_{s}|_{1}\bigg] \leq \overline{c} \varphi\bigg(E\big[|Y_{t_{0}}|_{1}\big] + \int_{t_{0}}^{t_{1}}\!E\big[|\kappa_{s}|_{1}\big] + \hat{\delta}_{1}(s)\,ds\bigg)e^{\frac{1}{2}\int_{t_{1}}^{t_{2}}\!\gamma_{1}(s)\,ds}
\end{align*}
for any $t_{2},t\geq t_{1}$ with $t_{2}\leq t < \hat{\delta}$ and $\varphi:\R_{+}\rightarrow\R_{+}$ given by $\varphi(x) := x + \sqrt{x}$.
\end{Proposition}

\begin{proof}
As $E[|\kappa|_{1}] = \hat{\delta}_{1,0} = (1-\beta_{k})\sum_{i=1}^{m}E[\lambda^{(i,k)}] = 0$ a.e.~on $[t_{1},\infty[$ for any $k\in\{1,\dots,l\}$, from Lemma~\ref{le:supremum moment stability estimate} we directly get that
\begin{equation}\label{eq:auxiliary supremum moment estimate 4}
\begin{split}
E\bigg[\sup_{s\in [t_{2},t]} 
e^{-\int_{t_{2}}^{s}\!\gamma(s_{0})\,ds_{0}}|Y_{s}|_{1}\bigg] &\leq E\big[|Y_{t_{2}}|_{1}\big] + \int_{t_{2}}^{t}\!\hat{\gamma}_{c}(s)e^{-\int_{t_{2}}^{s}\!\gamma(s_{0})\,ds_{0}}E\big[|Y_{s}|_{1}\big]\,ds\\
&\quad + 2\bigg(\int_{t_{2}}^{t}\!|\hat{\eta}(s)|_{1}^{2}e^{-2\int_{t_{2}}^{s}\!\gamma(s_{0})\,ds_{0}}E\big[|Y_{s}|_{1}\big]\,ds\bigg)^{\frac{1}{2}}
\end{split}
\end{equation}
for the function $\hat{\gamma}_{1} := (\gamma_{1,0} - \gamma)^{+} + \sum_{k=1,\,\beta_{k}=1}^{l}\sum_{i=1}^{m}E[\lambda^{(i,k)}]$. Further, because $\hat{\delta}_{1} = \hat{\delta}_{1,0}$ $+\, \sum_{k=1}^{l}(1-\beta_{k})\sum_{i=1}^{m}E[\lambda^{(i,k)}]$, we infer from the moment stability estimate of Theorem~\ref{th:moment stability estimate} that
\begin{align*}
e^{-q\int_{t_{2}}^{s}\!\gamma(s_{0})\,ds_{0} - c_{\gamma,q}}E\big[|Y_{s}|_{1}\big] &\leq e^{\int_{t_{0}}^{t_{2}}\!\gamma_{1}(s_{0})\,ds_{0}}E\big[|Y_{t_{0}}|_{1}\big]\\
&\quad + \int_{t_{0}}^{t_{1}}\!e^{\int_{s_{0}}^{t_{2}}\!\gamma_{1}(s_{1})\,ds_{1}}\big(E\big[|\kappa_{s_{0}}|_{1}\big] + \hat{\delta}_{1}(s_{0})\big)\,ds_{0}
\end{align*}
for any $s\in [t_{2},t]$ and each $q\in\{0,1,2\}$, where $c_{\gamma,q} := \overline{c}_{\gamma,q} + \overline{c}_{0}\sum_{k=1,\,\beta_{k}=1}^{l} 1$. Hence, the first two terms on the right-hand side in~\eqref{eq:auxiliary supremum moment estimate 4} do not exceed
\begin{equation*}
\overline{c}_{1}\bigg(E\big[|Y_{t_{0}}|_{1}\big] + \int_{t_{0}}^{t_{1}}\!E\big[|\kappa_{s}|_{1}\big] + \hat{\delta}_{1}(s)\,ds\bigg)e^{\frac{1}{2}\int_{t_{1}}^{t_{2}}\!\gamma_{1}(s)\,ds}
\end{equation*}
with $\overline{c}_{1}:=\exp(\int_{t_{0}}^{t_{1}}\!\gamma_{1}^{+}(s)\,ds + c_{\gamma,0}/2)(1 + e^{c_{\gamma,1}}(\overline{c}_{\gamma,-1} + \overline{c}_{0}\sum_{k=1,\,\beta_{k}=1}^{l}1))$. Moreover, the third expression in~\eqref{eq:auxiliary supremum moment estimate 4} is bounded by
\begin{equation*}
\overline{c}_{2}\bigg(E\big[|Y_{t_{0}}|_{1}\big] + \int_{t_{0}}^{t_{1}}\!E\big[|\kappa_{s}|_{1}\big] + \hat{\delta}_{1}(s)\,ds\bigg)^{\frac{1}{2}}e^{\frac{1}{2}\int_{t_{1}}^{t_{2}}\!\gamma_{1}(s)\,ds}
\end{equation*}
for $\overline{c}_{2}:= 2\overline{c}_{0}^{1/2}\exp(\frac{1}{2}(\int_{t_{0}}^{t_{1}}\!\gamma_{1}^{+}(s)\,ds + c_{\gamma,2}))$. Since $\exp(-\int_{t_{2}}^{s}\!\gamma(s_{0})\,ds_{0})\geq\exp(-\overline{c}_{\gamma,3})$ for all $s\in [t_{2},t]$, the assertion follows for $\overline{c}:=\exp(\overline{c}_{\gamma,3})(\overline{c}_{1}\vee\overline{c}_{2})$.
\end{proof}

\subsection{Pathwise asymptotic behaviour}

To aim of this section is to deduce the limiting behaviour of $Y$ from the moment estimate of Proposition~\ref{pr:supremum moment stability estimate}, by using the following application of the Borel-Cantelli Lemma.

\begin{Lemma}\label{le:pathwise asymptotic behaviour}
Let $A\in\mathcal{F}$ and $X$ be an $\R_{+}$-valued right-continuous process for which there are a strictly increasing sequence $(t_{n})_{n\in\N}$ in $[t_{0},\infty[$ with $\lim_{n\uparrow\infty} t_{n} = \infty$ and a sequence $(c_{n})_{n\in\N}$ in $]0,\infty[$ such that
\begin{equation}\label{eq:probability series condition}
\sum_{n=1}^{\infty} P\bigg(\sup_{s\in ]t_{n},t_{n+1}]} X_{s}\mathbbm{1}_{A} > c_{n}\bigg) < \infty.
\end{equation}
Then for any lower semicontinuous function $\varphi:]t_{1},\infty[\rightarrow ]0,\infty[$ it holds that
\begin{equation}\label{eq:pathwise asymptotic behaviour}
\limsup_{t\uparrow\infty} \frac{\log(X_{t})}{\varphi(t)} \leq \limsup_{n\uparrow\infty} \frac{\log(c_{n})}{\inf_{s\in ]t_{n},t_{n+1}]} \varphi(s)}\quad\text{a.s.~on $A$.}
\end{equation}
\end{Lemma}

\begin{proof}
By the Borel-Cantelli Lemma, there is a null set $N\in\mathcal{F}$ such that for any fixed $\omega\in N^{c}\cap A$ there is $n_{0}\in\mathbb{N}$ so that $\sup_{s\in ]t_{n},t_{n+1}]}X_{s}(\omega)\leq c_{n}$ for all $n\in\mathbb{N}$ with $n\geq n_{0}$. Hence,
\begin{equation*}
\sup_{n\in\mathbb{N}:\,n\geq n_{1}}\sup_{s\in ]t_{n},t_{n+1}]} \frac{\log(X_{s}(\omega))}{\varphi(s)} \leq \sup_{n\in\mathbb{N}:\,n\geq n_{1}} \frac{\log(c_{n})}{\inf_{s\in ]t_{n},t_{n+1}]}\varphi(s)}
\end{equation*}
for every $n_{1}\in\mathbb{N}$ with $n_{1}\geq n_{0}$. This in turn shows us that
\begin{equation*}
\limsup_{t\uparrow\infty} \frac{\log(X_{t}(\omega))}{\varphi(t)} = \inf_{n\in\mathbb{N}:\,n\geq n_{0}}\sup_{s > t_{n}} \frac{\log(X_{s}(\omega))}{\varphi(s)}\leq \limsup_{n\uparrow\infty}\frac{\log(c_{n})}{\inf_{s\in ]t_{n},t_{n+1}]}\varphi(s)}.
\end{equation*}
\end{proof}

\begin{Remark}\label{re:pathwise asymptotic behaviour}
Suppose that instead of~\eqref{eq:probability series condition} there are $\hat{c} > 0$ and $\hat{\varepsilon}\in ]0,1[$ such that $E[\sup_{s\in ]t_{n},t_{n+1}]} X_{s}\mathbbm{1}_{A}] \leq \hat{c} c_{n}$ for every $n\in\mathbb{N}$ and 
\begin{equation}\label{eq:epsilon-series}
\sum_{n=1}^{\infty} c_{n}^{\varepsilon} < \infty\quad\text{for all $\varepsilon\in ]0,\hat{\varepsilon}[$.}
\end{equation}
Then~\eqref{eq:pathwise asymptotic behaviour} follows as well. Indeed, for $\varepsilon\in ]0,\hat{\varepsilon}[$ Chebyshev's inequality and Lemma~\ref{le:pathwise asymptotic behaviour} yield a null set $N_{\varepsilon}\in\mathcal{F}$ such that
\begin{equation}\label{eq:epsilon-asymptotic behaviour}
\limsup_{t\uparrow\infty} \frac{\log(X_{t}(\omega))}{\varphi(t)} \leq (1-\varepsilon)\limsup_{n\uparrow\infty} \frac{\log(c_{n})}{\inf_{s\in ]t_{n},t_{n+1}]} \varphi(s)}
\end{equation}
for all $\omega\in N_{\varepsilon}^{c}\cap A$. So, any $\omega\in A$ that lies in the complement of $N := \bigcup_{\varepsilon\in\mathbb{Q}\cap ]0,\hat{\varepsilon}[} N_{\varepsilon}$ satisfies~\eqref{eq:epsilon-asymptotic behaviour} for $\varepsilon = 0$, which is the sharpest bound.
\end{Remark}

More specifically, one may derive the conditions in Remark~\ref{re:pathwise asymptotic behaviour} in the case that there are $n_{0}\in\mathbb{N}$ and a decreasing function $\psi:[n_{0},\infty[\rightarrow\R_{+}$ such that $c_{n} = \psi(n)$ for all $n\in\mathbb{N}$ with $n\geq n_{0}$. Then~\eqref{eq:epsilon-series} holds for some $\hat{\varepsilon}\in ]0,1[$ if and only if
\begin{equation}\label{eq:integral test condition}
\int_{n_{0}}^{\infty}\!\psi(v)^{\varepsilon}\,dv < \infty\quad\text{for all $\varepsilon\in ]0,\hat{\varepsilon}[$,}
\end{equation}
as the integral test for the convergence of series shows. We conclude with a pathwise estimate and stress the fact that the fraction $\frac{1}{2}$ comes from the H\oe lder condition~\eqref{Assumption 5}.

\begin{Theorem}\label{th:pathwise stability}
Let~\eqref{Assumption 8} and~\eqref{Assumption 3} hold and $\sum_{k=1}^{l}\sum_{i=1}^{m}E[\lambda^{(i,k)}]\theta_{k}^{\beta_{k}}$ be locally integrable. Suppose that $\gamma_{1}\leq 0$ a.e.~on $[t_{1},\infty[$ and there is a strictly increasing sequence $(t_{n})_{n\in\N\setminus\{1\}}$ in $[t_{1},\infty[$ such that
\begin{equation*}
\sup_{n\in\N} (t_{n+1} - t_{n}) < \hat{\delta},\quad \lim_{n\uparrow\infty} t_{n} = \infty
\end{equation*}
and $\sum_{n=1}^{\infty}\exp((\varepsilon/2)\int_{t_{1}}^{t_{n}}\!\gamma_{1}(s)\,ds) < \infty$ for every $\varepsilon\in ]0,\hat{\varepsilon}[$ and some $\hat{\varepsilon}\in ]0,1[$. If $E[|Y_{t_{0}}|_{1}]$ $< \infty$ or $\lambda = 0$, then
\begin{equation*}
\limsup_{t\uparrow\infty} \frac{1}{\varphi(t)}\log(|Y_{t}|_{1}) \leq \frac{1}{2}\limsup_{n\uparrow\infty} \frac{1}{\varphi(t_{n})}\int_{t_{1}}^{t_{n}}\!\gamma_{1}(s)\,ds\quad\text{a.s.}
\end{equation*}
for any increasing continuous function $\varphi:[t_{1},\infty[\rightarrow\R_{+}$ that is positive on $]t_{1},\infty[$.
\end{Theorem}

\begin{proof}
As $\gamma_{1} = \gamma_{1,0} + \sum_{k=1}^{l}\sum_{i=1}^{m}\beta_{k}E[\lambda^{(i,k)}]$, we have $\gamma_{1,0} \leq 0$ a.e.~on $[t_{1},\infty[$. Thus, if $|Y_{t_{0}}|_{1}$ is integrable, then we may choose $\gamma = \hat{\alpha}\gamma_{1,0}$ with $\hat{\alpha}\in [0,\frac{1}{2}]$ in Proposition~\ref{pr:supremum moment stability estimate} to get $\hat{c} > 0$ such that
\begin{equation*}
E\bigg[\sup_{s\in [t_{n},t_{n+1}]} |Y_{s}|_{1}\bigg] \leq \hat{c} e^{\frac{1}{2}\int_{t_{1}}^{t_{n}}\!\gamma_{1}(s)\,ds}
\end{equation*}
for every $n\in\N$. In the case that $\lambda = 0$ we set $A_{k}:=\{ |Y_{t_{0}}|_{1} \leq k\}$ for fixed $k\in\N$ and note that $Y\mathbbm{1}_{A_{k}}$ is a random It{\^o} process with drift $\B\mathbbm{1}_{A_{k}}$ and diffusion $\Sigma\mathbbm{1}_{A_{k}}$.

Since~\eqref{Assumption 4} and~\eqref{Assumption 5} are satisfied by $(\B\mathbbm{1}_{A_{k}},\Sigma\mathbbm{1}_{A_{k}})$ instead of $(\B,\Sigma)$ such that~\eqref{Assumption 8} holds, Proposition~\ref{pr:supremum moment stability estimate} gives $\hat{c}_{k} > 0$ so that
\begin{equation*}
E\bigg[\sup_{s\in [t_{n},t_{n+1}]} |Y_{s}\mathbbm{1}_{A_{k}}|_{1}\bigg] \leq \hat{c}_{k} e^{\frac{1}{2}\int_{t_{1}}^{t_{n}}\!\gamma_{1}(s)\,ds}
\end{equation*}
for each $n\in\N$. In either case, the asserted implication follows from Lemma~\ref{le:pathwise asymptotic behaviour} and Remark~\ref{re:pathwise asymptotic behaviour} by noting that $\bigcup_{k\in\N} A_{k} = \Omega$.
\end{proof}

\begin{Remark}\label{re:pathwise stability}
Let $\gamma:[t_{0},\infty[\rightarrow\R_{+}$ be a measurable locally integrable function satisfying $\gamma > 0$ a.e.~on $]t_{1},t_{1} + \tilde{\delta}[$ for some $\tilde{\delta} > 0$. If $\gamma_{1} \leq -\gamma$ a.e.~on $[t_{1},\infty[$, then
\begin{equation*}
\limsup_{t\uparrow\infty} \frac{\log(|Y_{t}|_{1})}{\int_{t_{0}}^{t}\!\gamma(s)\,ds} \leq - \frac{1}{2}(1-\hat{\gamma})\quad\text{a.s.}
\end{equation*}
with $\hat{\gamma} := \int_{t_{0}}^{t_{1}}\!\gamma(s)\,ds/\int_{t_{0}}^{\infty}\!\gamma(s)\,ds$. In this case, the sharpest bound is attained for $\hat{\gamma} = 0$, which occurs if and only if $\gamma = 0$ a.e.~on $[t_{0},t_{1}]$ or $\gamma$ fails to be integrable.
\end{Remark}

\section{Proofs of the main results}\label{se:5}

\subsection{Proofs for admissible spaces of probability measures}

\begin{proof}[Proof of Proposition~\ref{pr:admissibility}]
Let $S$ be a metrisable space, $(\tilde{\Omega},\tilde{\mathcal{F}},\tilde{P})$ be a probability space and $X:S\times\tilde{\Omega}\rightarrow\R^{m}$ be a continuous process so that $\mathcal{L}(X_{s})\in\mathcal{P}$ for all $s\in S$. By condition (i), the sequence $(F_{n})_{n\in\mathbb{N}}$ of $\mathcal{P}$-valued maps on $S$ defined via $F_{n}(s) := \mathcal{L}(\varphi_{n}\circ X_{s})$ satisfies $\lim_{n\uparrow\infty} F_{n}(s) = \mathcal{L}(X_{s})$ in $\mathcal{P}$ for any given $s\in S$.

Since $(X_{s_{k}})_{k\in\mathbb{N}}$ converges pointwise to $X_{s}$ for any sequence $(s_{k})_{k\in\mathbb{N}}$ in $S$ converging to $s$, it follows from~(ii) that $\lim_{k\uparrow\infty} F_{n}(s_{k}) = F_{n}(s)$ in $\mathcal{P}$ for each $n\in\mathbb{N}$. Thus, the map $S\rightarrow\mathcal{P}$, $s\mapsto \mathcal{L}(X_{s})$ is Borel measurable as pointwise limit of a sequence of continuous maps.
\end{proof}

\begin{proof}[Proof of Corollary~\ref{co:admissibility}]
We show that the two conditions of Proposition~\ref{pr:admissibility} are met by the sequence $(\varphi_{n})_{n\in\N}$ of radial retractions, introduced in Example~\ref{ex:radial retraction}. 

Let $\mu\in\mathcal{P}$ and note that~(i) directly yields $\mu\circ\varphi_{n}^{-1}\in\mathcal{P}$ for fixed $n\in\N$. If we define $\phi_{n}:\R^{m}\rightarrow\R^{m}\times\R^{m}$ by $\phi_{n}(x) := (\varphi_{n}(x),x)$, then $\theta_{n}:=\mu\circ\phi_{n}^{-1}$ belongs to $\mathcal{P}(\mu\circ\varphi_{n}^{-1},\mu)$ and
\begin{equation*}
\int_{\R^{m}\times\R^{m}}\!\rho(|x-y|)\,d\theta_{n}(x,y) = \int_{\R^{m}}\!\rho(|\varphi_{n}(x) - x|)\,\mu(dx),
\end{equation*}
by the measure transformation formula. Since $\rho(|\varphi_{n}(x) - x|) \leq \rho(|x|)$ for all $x\in\R^{m}$, the integral on the right-hand side converges to zero as $n\uparrow\infty$, by dominated convergence. Thus, from~(ii) we infer that $\lim_{n\uparrow\infty} \mu\circ\varphi_{n}^{-1} = \mu$ in $\mathcal{P}$.

Now let $(\mu_{k})_{k\in\N}$ be a sequence in $\mathcal{P}$ that converges stochastically to some $\mu\in\mathcal{P}$. That is, there is a sequence $(\theta_{k})_{k\in\N}$ of Borel measures on $\R^{m}\times\R^{m}$ such that $\theta_{k}\in\mathcal{P}(\mu_{k},\mu)$ for any $k\in\N$ and~\eqref{eq:stochastic convergence of measures} holds.

For fixed $n\in\N$ and $\psi_{n}:\R^{m}\times\R^{m}\rightarrow\R^{m}\times\R^{m}$ given by $\psi_{n}(x,y) := (\varphi_{n}(x),\varphi_{n}(y))$, the measure $\hat{\theta}_{k}:=\theta_{k}\circ\psi_{n}^{-1}$ lies in $\mathcal{P}(\mu_{k}\circ\varphi_{n}^{-1},\mu\circ\varphi_{n}^{-1})$ and the measure transformation formula ensures that
\begin{equation*}
\int_{\R^{m}\times\R^{m}}\!\rho(|x-y|)\,d\hat{\theta}_{k}(x,y) = \int_{\R^{m}\times\R^{m}}\!\rho(|\varphi_{n}(x) - \varphi_{n}(y)|)\,d\theta_{k}(x,y)
\end{equation*}
for all $k\in\N$. For given $\varepsilon > 0$ we choose $\delta > 0$ and $k_{0}\in\N$ such that $\rho(2\delta) < \varepsilon/2$ and $\theta_{k}(\{(x,y)\in\R^{m}\times\R^{m}\,|\,|x-y|\geq \delta\}) < (\varepsilon/2)(1 + \rho(2n))^{-1}$ for any $k\in\N$ with $k\geq k_{0}$. Then the integral on the right-hand side cannot exceed $\varepsilon$ for every such $k\in\N$. This shows $\lim_{k\uparrow\infty}\mu_{k}\circ\varphi_{n}^{-1} = \mu\circ\varphi_{n}^{-1}$ in $\mathcal{P}$ and the assertion follows.
\end{proof}

\subsection{Proofs of the moment estimates, uniqueness and moment stability}

\begin{proof}[Proof of Proposition~\ref{pr:abstract stability estimate}]
We define two progressively measurable processes $\hat{\B}$ and $\hat{\Sigma}$ with values in $\R^{m}$ and $\R^{m\times d}$, respectively, by
\begin{equation}\label{eq:drift and diffusion difference}
\hat{\B}_{s} := \B_{s}\big(X_{s},\mathcal{L}(X_{s})\big) - \tilde{\B}_{s}\big(\tilde{X}_{s},\mathcal{L}(\tilde{X}_{s})\big)\quad\text{and}\quad \hat{\Sigma}_{s} := \Sigma_{s}(X_{s}) - \Sigma_{s}(\tilde{X}_{s}).
\end{equation}
Then the difference $Y$ of $X$ and $\tilde{X}$ is a random It{\^o} process with drift $\hat{\B}$ and diffusion $\hat{\Sigma}$ such that
\begin{equation*}
\mathrm{sgn}(Y^{(i)})\hat{\B}^{(i)} \leq \varepsilon^{(i)} + \eta^{(i)}\rho(|Y|_{1}) + \lambda^{(i)}\varrho\circ\theta\quad\text{a.s.}
\end{equation*}
for all $i\in\{1,\dots,m\}$ with the measurable function $\theta :=\vartheta(\mathcal{L}(X),\mathcal{L}(\tilde{X}))$. Hence, the assertion follows from an application of Theorem~\ref{th:abstract moment estimate}.
\end{proof}

\begin{proof}[Proof of Corollary~\ref{co:pathwise uniqueness}]
To show uniqueness in both cases, we suppose that $X$ and $\tilde{X}$ are two solutions to~\eqref{eq:McKean-Vlasov} such that $X_{t_{0}} = \tilde{X}_{t_{0}}$ a.s. In case (i) we require, as stated, the local integrability of the function $\Theta(\cdot,\mathcal{L}(X),\mathcal{L}(\tilde{X}),\mathcal{L}(X - \tilde{X}))$, which equals
\begin{equation*}
E\big[|\lambda|_{1}\big]\varrho\big(\vartheta(\mathcal{L}(X),\mathcal{L}(\tilde{X}))\big) + \mathbbm{1}_{]0,\infty[}\big(\Phi_{\rho}(\infty)\big)\eta E\big[\rho(|X-\tilde{X}|_{1})\big].
\end{equation*}
Then $E[|X_{t} - \tilde{X}_{t}|_{1}]$ vanishes for all $t\geq t_{0}$, due to Proposition~\ref{pr:abstract stability estimate}. Indeed, $\varrho_{0} := \rho\vee\varrho$ satisfies $(0,w)\in D_{\varrho_{0}}$ and $\Psi_{\varrho_{0}}(0,w) = 0$ for all $w\geq 0$, as $\Phi_{\varrho_{0}}(0) = - \infty$. So, $X= \tilde{X}$ a.s., by the continuity of paths.

In case (ii) we set $\tau_{n} := \inf\{t\geq t_{0}\,|\,|X_{t}|\geq n\text{ or } |\tilde{X}_{t}| \geq n\}$ for fixed $n\in\N$ and observe that the difference $Y$ of $X^{\tau_{n}}$ and $\tilde{X}^{\tau_{n}}$ is a random It{\^o} process with drift $\tilde{\B}$ and diffusion $\tilde{\Sigma}$ given by
\begin{equation*}
\tilde{\B}_{s} := \big(\hat{\B}_{s}(X_{s}) - \hat{\B}_{s}(\tilde{X}_{s})\big)\mathbbm{1}_{\{\tau_{n} > s\}}\quad\text{and}\quad \tilde{\Sigma}_{s} := \big(\Sigma_{s}(X_{s}) - \Sigma_{s}(\tilde{X}_{s})\big)\mathbbm{1}_{\{\tau_{n} > s\}}.
\end{equation*}
Then $\mathrm{sgn}(Y^{(i)})\tilde{\B}^{(i)} \leq \eta_{n}\rho_{n}(|Y|_{1})$ a.s.~for each $i\in\{1,\dots,m\}$. Thus, Theorem~\ref{th:abstract moment estimate} gives $Y = 0$ a.s. From $\sup_{n\in\N} \tau_{n} = \infty$ we conclude that $X = \tilde{X}$ a.s.
\end{proof}

\begin{proof}[Proof of Proposition~\ref{pr:moment stability estimate}]
As in the proof of Proposition~\ref{pr:abstract stability estimate}, we let the two processes $\hat{\B}$ and $\hat{\Sigma}$ with values in $\R^{m}$ and $\R^{m\times d}$, respectively, be given by~\eqref{eq:drift and diffusion difference}. Then
\begin{equation*}
\mathrm{sgn}(Y^{(i)})Y^{(i)} \leq \varepsilon^{(i)} + \sum_{k=1}^{l}\bigg(\sum_{j=1}^{m}\eta^{(i,j,k)}|Y^{(j)}|^{\alpha_{k}}\bigg) + \lambda^{(i,k)}\vartheta\big(\mathcal{L}(X),\mathcal{L}(\tilde{X})\big)^{\beta_{k}}\quad\text{a.s.}
\end{equation*}
for any $i\in\{1,\dots,m\}$. For this reason, all assertions are implied by Theorem~\ref{th:moment stability estimate}.
\end{proof}

\begin{proof}[Proof of Corollary~\ref{co:moment stability}]
By Definition~\ref{de:stability}, the stability assertions directly follow from Proposition~\ref{pr:moment stability estimate} in the case that $\B=\tilde{\B}$ and $\varepsilon=0$.
\end{proof}

\begin{proof}[Proof of Corollary~\ref{co:exponential moment stability}]
(i) Based on Proposition~\ref{pr:moment stability estimate}, we immediately infer both claims from the reasoning in Remark~\ref{re:moment stability}.

(ii) Let $X$ and $\tilde{X}$ be two solutions to~\eqref{eq:McKean-Vlasov} for which $E[|X_{t_{0}}-\tilde{X}_{t_{0}}|]$ is finite and $E[|\lambda|_{1}]\vartheta(\mathcal{L}(X),\mathcal{L}(\tilde{X}))$ is locally integrable. Then Proposition~\ref{pr:moment stability estimate} gives
\begin{equation*}
E\big[|X_{t} - \tilde{X}_{t}|\big] \leq \sqrt{m}e^{\int_{t_{0}}^{t}\!\gamma_{1}(s)\,ds}E\big[|X_{t_{0}} - \tilde{X}_{t_{0}}|\big]
\end{equation*}
for each $t\geq t_{0}$. Thus, for the first assertion let $P(X_{t_{0}}\neq \tilde{X}_{t_{0}}) > 0$, as otherwise $E[|X-\tilde{X}|]$ vanishes. Then from~\eqref{co:7} we get that
\begin{equation*}
\limsup_{t\uparrow\infty}\frac{1}{t^{\alpha_{l}}}\log\big(E\big[|X_{t} - \tilde{X}_{t}|\big]\big) \leq \limsup_{t\uparrow\infty} \sum_{k=1}^{l} \hat{\lambda}_{k}\frac{(t-s_{k})^{\alpha_{k}} - (t_{1} - s_{k})^{\alpha_{k}}}{t^{\alpha_{l}}} =  \hat{\lambda}_{l},
\end{equation*}
since $\lim_{t\uparrow\infty} (t-s_{k})^{\alpha_{k}}/t^{\alpha_{l}}$ $= \mathbbm{1}_{\{l\}}(k)$ for any $k\in\{1,\dots,l\}$. Now Remark~\ref{re:Lyapunov exponent} yields the correct result. 

For the second assertion it is sufficient to consider the case $l=1$. First, we set $\hat{c}_{0}:=\max_{t\in [t_{0},t_{1}] }\exp(\int_{t_{0}}^{t}\!\gamma_{1}(s)\,ds - \hat{\lambda}_{1}(t-t_{0})^{\alpha_{1}})$ and get that
\begin{equation}\label{eq:specific exponential moment stability}
E\big[|X_{t} - \tilde{X}_{t}|\big] \leq \sqrt{m}\hat{c}_{0}e^{\hat{\lambda}_{1}(t-t_{0})^{\alpha_{1}}}E\big[|X_{t_{0}} - \tilde{X}_{t_{0}}|\big]
\end{equation}
for each $t\in [t_{0},t_{1}]$. Since $\int_{t_{1}}^{t}\!\gamma_{1}(s)\,ds \leq \hat{\lambda}_{1}((t- s_{1})^{\alpha_{1}} - (t_{1} - s_{1})^{\alpha_{1}})$ for all fixed $t\geq t_{1}$, we see that~\eqref{eq:specific exponential moment stability} holds if we replace $\hat{c}_{0}$ by $\hat{c}:= \hat{c}_{0}\vee\exp(\int_{t_{0}}^{t_{1}}\gamma_{1}(s)\,ds -\hat{\lambda}_{1}(t_{1}-s_{1})^{\alpha_{1}})$.
\end{proof}

\subsection{Proofs for pathwise stability and the moment growth estimates}

\begin{proof}[Proof of Proposition~\ref{pr:specific pathwise asymptotic analysis}]
As we have seen, $Y$ is a random It{\^o} process with drift $\hat{\B}$ and diffusion $\hat{\Sigma}$ given by~\eqref{eq:drift and diffusion difference} when $\tilde{\B} = \B$. Therefore, Theorem~\ref{th:pathwise stability} entails the claim.
\end{proof}

For the proof of Corollary~\ref{co:special pathwise stability} we need to check whether the series in~\eqref{co:10} converges when the upper bound for $\gamma_{1}$ in~\eqref{co:7} is used.

\begin{Lemma}\label{le:integrability of exponential power functions}
Let $l\in\mathbb{N}$, $\alpha\in ]0,\infty[^{l}$ and $\beta,s\in\R^{l}$ satisfy $\alpha_{1} < \cdots < \alpha_{l}$, $\beta_{l} < 0$ and $\max_{k=1,\dots,l} s_{k} \leq t_{1}$ for some $t_{1}\geq 0$. Then
\begin{equation*}
\int_{0}^{\infty}\!\exp\bigg(\varepsilon\sum_{k=1}^{l}\beta_{k}\int_{t_{1}}^{t_{1} + \delta t}\!\alpha_{k}(s-s_{k})^{\alpha_{k}-1}\,ds\bigg)\,dt < \infty\quad\text{for all $\delta,\varepsilon > 0$.}
\end{equation*}
\end{Lemma}

\begin{proof}
It suffices to consider the case $\varepsilon = 1$, since $\beta$ may be replaced by $\varepsilon\beta$. We set $c:= - \sum_{k=1}^{l}\beta_{k}(t_{1} - s_{k})^{\alpha_{k}}$ and readily note that there is $t_{2} > 0$ such that
\begin{equation*}
\sum_{k=1}^{l}\beta_{k}\int_{t_{1}}^{t_{1} + \delta t}\alpha_{k}(s-s_{k})^{\alpha_{k}-1}\,ds = c + \sum_{k=1}^{l}\beta_{k}(t_{1} + \delta t - s_{k})^{\alpha_{k}} \leq \frac{\beta_{l}}{2}(\delta t)^{\alpha_{l}}
\end{equation*}
for all $t\geq t_{2}$. Further, a substitution shows us that $\int_{0}^{\infty}e^{- c(\delta t)^{\alpha_{l}}}\,dt$ $= \frac{1}{\alpha_{l}\delta} c^{-\frac{1}{\alpha_{l}}}\Gamma(\frac{1}{\alpha_{l}})$ for any $c > 0$, where $\Gamma$ is the gamma function. So, the claim follows.
\end{proof}

\begin{proof}[Proof of Corollary~\ref{co:special pathwise stability}]
To apply Proposition~\ref{pr:specific pathwise asymptotic analysis}, we verify~\eqref{co:10}. For this purpose, we choose $\hat{t}_{1}\geq t_{1},$ such that $\gamma_{1}\leq 0$ a.e.~on $[\hat{t}_{1},\infty[$ and define a sequence $(t_{n})_{n\in\N\setminus\{1\}}$ in $[t_{0},\infty[$ by $t_{n}:= \hat{t}_{1} + \tilde{\delta}(n-1)$ for some $\tilde{\delta} > 0$. The integral test for the convergence of series, recalled in~\eqref{eq:integral test condition}, shows that
\begin{equation*}
\sum_{n=1}^{\infty}e^{\varepsilon\int_{\hat{t}_{1}}^{t_{n}}\!\gamma_{1}(s)\,ds} < \infty\quad\Leftrightarrow\quad \int_{0}^{\infty}\exp\bigg(\varepsilon\int_{\hat{t}_{1}}^{\hat{t}_{1} + \tilde{\delta}t}\!\gamma_{1}(s)\,ds\bigg)\,dt < \infty
\end{equation*}
for any given $\varepsilon > 0$ and the latter condition is always satisfied, due to the imposed upper bound on $\gamma_{1}$ in~\eqref{co:7} and Lemma~\ref{le:integrability of exponential power functions}. 
 
Thus, for~\eqref{co:10} to be valid, it suffices to take $\tilde{\delta} < \hat{\delta}$. Then, by Proposition~\ref{pr:specific pathwise asymptotic analysis}, the difference $Y$ of any two solutions $X$ and $\tilde{X}$ to~\eqref{eq:McKean-Vlasov} for which $E[|\lambda|_{1}]\vartheta(\mathcal{L}(X),\mathcal{L}(\tilde{X}))$ is locally integrable satisfies
\begin{equation*}
\limsup_{t\uparrow\infty} \frac{\log\big(|Y_{t}|_{1}\big)}{t^{\alpha_{l}}} \leq  \frac{1}{2}\limsup_{n\uparrow\infty} \sum_{k=1}^{l}\hat{\lambda}_{k}\frac{(t_{n} - s_{k})^{\alpha_{k}} - (\hat{t}_{1} - s_{k})^{\alpha_{k}}}{t_{n}^{\alpha_{l}}} = \frac{\hat{\lambda}_{l}}{2}\quad\text{a.s.}
\end{equation*}
as soon as $E[|Y_{t_{0}}|_{1}] < \infty$ or $\B$ is independent of $\mu\in\mathcal{P}$.
\end{proof}

\begin{proof}[Proof of Lemma~\ref{le:abstract growth estimate}]
Because $X$ is a random It{\^o} process with drift $\hat{\B}$ and $\hat{\Sigma}$ defined by $\hat{\B} := \B(X,\mathcal{L}(X))$ and $\hat{\Sigma}:=\Sigma(X)$, the claim is a direct consequence of Theorem~\ref{th:abstract moment estimate}.
\end{proof}

\begin{proof}[Proof of Lemma~\ref{le:moment growth estimate}]
By the same reasoning as in Lemma~\ref{le:abstract growth estimate}, the assertions follow from an application of Theorem~\ref{th:moment stability estimate}. 
\end{proof}

\subsection{Derivation of strong solutions}

\begin{proof}[Proof of Proposition~\ref{pr:mu-SDE}]
(i) As the partial uniform continuity condition~\eqref{co:4} holds in the case that $\B = b_{\mu}$, pathwise uniqueness for~\eqref{eq:mu-SDE} is implied by Corollary~\ref{co:pathwise uniqueness}.

(ii) We shall first suppose that $\xi$ is essentially bounded. Then the support of $\mathcal{L}(\xi)$ is compact and it essentially follows from Theorem 2.3 in~\cite{IkeWat81}[Chapter IV] that there is a local weak solution $\tilde{X}$ to~\eqref{eq:mu-SDE}.

By using the one-point compactification of $\R^{m}$, we can view $\tilde{X}$ as an $\R^{m}\cup\{\infty\}$-valued adapted continuous process on a filtered probability space $(\tilde{\Omega},\tilde{\mathcal{F}},(\tilde{\mathcal{F}}_{t})_{t\geq 0},\tilde{P})$ that allows for an $(\tilde{\mathcal{F}}_{t})_{t\geq 0}$-Brownian motion $\tilde{W}$ such that the usual and the following conditions are satisfied:
\begin{enumerate}[(1)]
\item If $\tilde{X}_{s}(\omega) = \infty$ for some $(s,\omega)\in [t_{0},\infty[\times\Omega$, then $\tilde{X}_{t}(\omega) = \infty$ for all $t \geq s$.

\item $\mathcal{L}(\tilde{X}_{t_{0}}) = \mathcal{L}(\xi)$ and the supremum $\tau$ of the sequence $(\tau_{n})_{n\in\mathbb{N}}$ of stopping times given by $\tau_{n} := \inf\{t\geq t_{0}\,|\,|\tilde{X}_{t}| \geq n\}$ satisfies $\tau > t_{0}$ a.s.

\item For any $n\in\mathbb{N}$ the process $\tilde{X}^{\tau_{n}}$ is a solution to~\eqref{eq:McKean-Vlasov} relative to $\tilde{W}$ when $\B$ and $\Sigma$ are replaced by the admissible maps
\begin{equation*}
[t_{0},\infty[\times\tilde{\Omega}\times\R^{m}\rightarrow\R^{m},\quad (s,\tilde{\omega},x)\mapsto b_{\mu}(s,x)\mathbbm{1}_{\{\tau_{n} > s\}}(\tilde{\omega})
\end{equation*}
and $[t_{0},\infty[\times\tilde{\Omega}\times\R^{m}\rightarrow\R^{m\times d}$, $(s,\tilde{\omega},x)\mapsto\sigma(s,x)\mathbbm{1}_{\{\tau_{n} > s\}}(\tilde{\omega})$, respectively.
\end{enumerate}

We readily observe that~\eqref{det.co:2} implies the partial growth condition~\eqref{co:12} for $b_{\mu}$ instead of $\B$. In fact, we may define $\kappa_{\mu}:[t_{0},\infty[\rightarrow\R_{+}^{m}$ coordinatewise by $(\kappa_{\mu})_{i} := \kappa_{i} + \chi_{i}\varphi(\vartheta(\mu,\delta_{0}))$ and get that
\begin{equation*}
\mathrm{sgn}(x_{i})(b_{\mu})_{i}(s,x)\mathbbm{1}_{\{\hat{\tau} > s\}} \leq (\kappa_{\mu})_{i}(s) + \upsilon_{i}(s)\phi(|x|_{1})
\end{equation*}
for any $(s,x)\in [t_{0},\infty[\times\R^{m}$, all $i\in\{1,\dots,m\}$ and each stopping time $\hat{\tau}$. Further, $|e_{i}'\sigma(s,\cdot)\mathbbm{1}_{\{\hat{\tau} > s\}}| \leq |e_{i}'\sigma(s,\cdot)|$. Hence, we infer from Fatou's lemma that
\begin{equation}\label{eq:weak solution estimate}
\tilde{E}[|\tilde{X}_{t}^{\tau}|_{1}] \leq \liminf_{n\uparrow\infty} \tilde{E}[|\tilde{X}_{t}^{\tau_{n}}|_{1}]\\
\leq \Psi_{\phi}\bigg(E\big[|\xi|_{1}\big] + \int_{t_{0}}^{t}\!|\kappa_{\mu}(s)|_{1}\,ds,\int_{t_{0}}^{t}\!|\upsilon(s)|_{1}\,ds\bigg)
\end{equation}
for all $t\geq t_{0}$, by the virtue of Lemma~\ref{le:abstract growth estimate}. In particular, $\tau = \infty$ and $\tilde{X}\in\R^{m}$ $\tilde{P}$-a.s. So, $X:[t_{0},\infty[\times\tilde{\Omega}\rightarrow\R^{m}$ given by $X_{t}(\tilde{\omega}):=\tilde{X}_{t}(\tilde{\omega})$, if $\tau(\tilde{\omega}) = \infty$, and $X_{t}(\tilde{\omega}):= 0$, otherwise, is a weak solution to~\eqref{eq:mu-SDE} and  $\tilde{E}[|X|]$ is locally bounded.

Now we remove the boundedness hypothesis on $\xi$. As we have shown that to each $x\in\R^{m}$ there is a weak solution to~\eqref{eq:mu-SDE} with initial value condition $x$, it follows from Remark 2.1 in~\cite{IkeWat81}[Chapter IV] that there exists a weak solution $X$ to~\eqref{eq:mu-SDE} satisfying $X_{t_{0}} = \xi$ a.s. Further, its absolute moment function is locally bounded if $\xi$ is integrable, as it cannot exceed the right-hand estimate in~\eqref{eq:weak solution estimate} in this case, due to Lemma~\ref{le:abstract growth estimate}.

(iii) By the first two assertions, we have pathwise uniqueness for~\eqref{eq:mu-SDE} and there is a weak solution for any $\R^{m}$-valued $\mathcal{F}_{t_{0}}$-measurable random vector serving as initial condition. As postulated by Theorem 1.1 in~\cite{IkeWat81}[Chapter IV], there is a unique strong solution $X^{\xi,\mu}$ to~\eqref{eq:mu-SDE} with $X_{t_{0}}^{\xi,\mu} = \xi$ a.s.
\end{proof}

\begin{proof}[Proof of Theorem~\ref{th:strong existence}]
(i) and (ii) Pathwise uniqueness is an immediate consequence of Corollary~\ref{co:pathwise uniqueness} and Remark~\ref{re:pathwise uniqueness}. In particular, there is at most a unique solution $X$ to~\eqref{eq:deterministic McKean-Vlasov} such that $X_{t_{0}} = \xi$ a.s.~and $E[|X|]$ is locally bounded, since $\lambda$ is locally integrable.

Regarding existence, we recall that for any $\mu\in B_{b,loc}(\mathcal{P})$ there is a unique strong solution $X^{\xi,\mu}$ to~\eqref{eq:mu-SDE} such that $X_{t_{0}}^{\xi,\mu} = \xi$ a.s.~and, as $\xi$ is integrable, $E[|X^{\xi,\mu}|]$ is locally bounded, by Proposition~\ref{pr:mu-SDE}. The definition of $b_{\mu}$ entails that $X^{\xi,\mu}$ solves~\eqref{eq:deterministic McKean-Vlasov} if $\mu$ is a fixed-point of the operator
\begin{equation*}
\Psi:B_{b,loc}(\mathcal{P})\rightarrow B_{b,loc}(\mathcal{P}_{1}(\R^{m})),\quad  \Psi(\nu)(t):=\mathcal{L}(X_{t}^{\xi,\nu}).
\end{equation*}
In this case, $X^{\xi,\mu}$ must be a strong solution. For any $\mu,\tilde{\mu}\in B_{b,loc}(\mathcal{P})$ we readily see that condition~\eqref{co:5} holds for $(b_{\mu},b_{\tilde{\mu}})$ instead of $(\B,\tilde{\B})$. For this reason, Proposition~\ref{pr:moment stability estimate} yields that
\begin{equation}\label{eq:fixed-point condition}
\vartheta_{1}\big(\Psi(\mu),\Psi(\tilde{\mu})\big)(t) \leq E\big[|X_{t}^{\xi,\mu} - X_{t}^{\xi,\tilde{\mu}}|\big] \leq \int_{t_{0}}^{t}\!e^{\int_{s}^{t}\!\gamma_{1,0}(\tilde{s})\,d\tilde{s}}|\lambda(s)|_{1}\vartheta(\mu,\tilde{\mu})(s)\,ds
\end{equation}
for each $t\geq t_{0}$. We notice that the function $[s,\infty[\rightarrow\R_{+}$, $t\mapsto \exp(\int_{s}^{t}\!\gamma_{1,0}^{+}(\tilde{s})\,d\tilde{s})|\lambda(s)|_{1}$ is increasing for any $s\geq t_{0}$. Consequently, Gronwall's inequality entails that $\Psi$ admits at most a unique fixed-point.

As $B_{b,loc}(\mathcal{P}_{1}(\R^{m}))$ is completely metrisable, existence and the error estimate~\eqref{eq:error estimate}, which implies the local uniform convergence assertion, follow from an application of the fixed-point theorem for time evolution operators in~\cite{Kal24}. In fact,
\begin{equation}\label{eq:auxiliary error estimate}
\vartheta_{1}(\mu_{m},\mu_{n})(t) \leq \Delta(t)\sum_{i=n}^{m-1}\frac{1}{i}\bigg(\int_{t_{0}}^{t}e^{\int_{s}^{t}\gamma_{1,0}(\tilde{s})\,d\tilde{s}}|\lambda(s)|_{1}\,ds\bigg)^{i}
\end{equation}
for any $m,n\in\N$ with $m > n$ and $t\geq t_{0}$, by induction and the triangle inequality. So, $(\mu_{n})_{n\in\N}$ is a Cauchy sequence in $B_{b,loc}(\mathcal{P}_{1}(\R^{m}))$ and, due to~\eqref{eq:fixed-point condition}, its limit $\mu$ is a fixed-point of $\Psi$. Hence,~\eqref{eq:error estimate} follows from~\eqref{eq:auxiliary error estimate} by taking the limit $m\uparrow\infty$.

(iii) Since the bound in~\eqref{eq:growth estimate} is independent of $\mu\in B_{b,loc}(\mathcal{P}_{1}(\R^{m}))$, the set $M$ is closed and convex. We set $g_{1,1} := \sum_{k=1}^{l}\beta_{k}\sum_{i=1}^{m}\chi_{i,k}$ and $g_{1,0} = g_{1} - g_{1,1}$. Then
\begin{equation*}
\vartheta_{1}(\Psi(\mu)(t),\delta_{0}) \leq e^{\int_{t_{0}}^{t}\!g_{1,0}(s)\,ds}E\big[|\xi|_{1}\big] + \int_{t_{0}}^{t}\!e^{\int_{s}^{t}\!g_{1,0}(\tilde{s})\,d\tilde{s}}\big(|\kappa|_{1} + h_{1} + g_{1,1}\vartheta_{1}(\mu,\delta_{0})\big)(s)\,ds
\end{equation*}
for all $\mu\in B_{b,loc}(\mathcal{P})$ and $t\geq t_{0}$, by Lemma~\ref{le:moment growth estimate} and Young's inequality. Then it follows from the Fundamental Theorem of Calculus for Lebesgue-Stieltjes integrals and Fubini's theorem that $\Psi$ maps $M$ into itself. Hence, the claim holds.
\end{proof}

\noindent
{\bf Acknowledgements}
\smallskip

\noindent
The authors thank an anonymous referee for instructive comments that were helpful to improve the paper.

\let\OLDthebibliography\thebibliography
\renewcommand\thebibliography[1]{
  \OLDthebibliography{#1}
  \setlength{\parskip}{0pt}
  \setlength{\itemsep}{2pt}}

\end{document}